\documentclass[12pt,a4paper]{amsart}
\usepackage[T1]{fontenc}
\usepackage[utf8]{inputenc}
\usepackage[margin=25mm]{geometry}
\usepackage{amsmath,amssymb,amsthm,mathtools}
\usepackage{mathrsfs}
\usepackage{booktabs,longtable,array}
\usepackage[expansion=false]{microtype}
\usepackage[colorlinks=true,linkcolor=blue,citecolor=blue,urlcolor=blue]{hyperref}
\allowdisplaybreaks[2]
\makeatletter
\renewcommand{\paragraph}{\@startsection{paragraph}{4}%
  {\normalparindent}{\z@}{-\fontdimen2\font}{\normalfont}}
\makeatother

\numberwithin{equation}{section}
\newtheorem{theorem}{Theorem}[section]
\newtheorem{proposition}[theorem]{Proposition}
\newtheorem{lemma}[theorem]{Lemma}
\newtheorem{corollary}[theorem]{Corollary}
\newtheorem*{mainresult}{Main theorem}
\newcommand{\eb}{\epsilon_{\rm b}}
\newcommand{\OO}{\mathcal O}
\newcommand{\sfq}{\mathrm{sf}}
\newcommand{\eps}{\varepsilon}
\newcommand{\Rea}{\operatorname{Re}}
\newcommand{\Ima}{\operatorname{Im}}
\newcommand{\rad}{\operatorname{rad}}
\newcommand{\Res}{\operatorname*{Res}}
\newcommand{\Src}[2]{\textup{\cite[#1~#2]{QRH}}}
\newcommand{\Lfun}{L_F}

\newcommand{\pos}[1]{\left(#1\right)_+}
\hypersetup{pdftitle={Improved zero-free half-planes for the quasi-Riemann hypothesis}}
\title[Improved zero-free half-planes]{Improved zero-free half-planes for the quasi-Riemann hypothesis}
\author[Baiying Liu]{Baiying Liu}
\address{Department of Mathematics\\
Purdue University\\
West Lafayette, IN, 47907, USA}
\email{liu2053@purdue.edu}
\date{10 October 2026}
\begin{document}
\begin{abstract}
Recently, a uniform
$\frac{7}{8}$ zero-free half-plane for finite-order Hecke $L$-functions over
$\mathbb Q(\sqrt{-3})$ and its transfer to Dirichlet $L$-functions was given by OpenAI in \cite{QRH}. In its proof, there are several parameters $b=\frac{1}{8}, \ell=\frac{1}{6}$, and $\kappa$.
In this paper, via varying $b$, $\ell$, and $\kappa$, following the same strategy, we slightly improve the bound
$\frac{7}{8}=0.875$
to 
\[
 B_{\mathrm{new}}=\frac{18\xi_{\mathrm{new}}+7}{18\xi_{\mathrm{new}}+9}
 =0.874957019420098946\ldots,
\]
where $\xi_{\mathrm{new}}$ is the unique root in $(0.38858,0.38859)$ of
$1188\xi^3-2160\xi^2+171\xi+190=0$.
Explicitly, with $\arccos\in[0,\pi]$,
\[
 \xi_{\mathrm{new}}=\frac{20}{33}+\frac{\sqrt{1391}}{33}
 \cos\!\left(\frac13\arccos\frac{28470}{1391^{3/2}}
                         -\frac{2\pi}{3}\right).
\]
The limiting parameters are
\[
\begin{gathered}
 b_{\mathrm{new}}=\frac{8(3\xi_{\mathrm{new}}-1)(12\xi_{\mathrm{new}}^2-40\xi_{\mathrm{new}}+15)}
                   {9(2\xi_{\mathrm{new}}+1)(24\xi_{\mathrm{new}}^2-20\xi_{\mathrm{new}}+5)},\\
 \ell_{\mathrm{new}}=\frac{5-6\xi_{\mathrm{new}}}{9+18\xi_{\mathrm{new}}},\qquad
 \kappa_{\mathrm{new}}=\frac{18\xi_{\mathrm{new}}+5}{18\xi_{\mathrm{new}}+9}
             =2B_{\mathrm{new}}-1.
\end{gathered}
\]
We also prove its optimality within the explicitly specified exponent
comparison system. 
This result has been formalized by Lean \cite{LeanExtension}.
\end{abstract}
\maketitle

\tableofcontents

\section*{Introduction}
%\phantomsection\addcontentsline{toc}{section}{Introduction}
\label{sec:combined-introduction}

In this paper, we study how varying the parameters in the cubic-theta argument of
\cite{QRH} improves its zero-free half-plane.
After recalling the original argument, we derive three successive
improvements and determine the optimal boundary within the specified
system of estimates.

The quasi-Riemann hypothesis asks whether there is a fixed
$B<1$ such that $\zeta(s)\ne0$ for $\Rea s>B$, with its pole
at $s=1$ allowed.
The uniform character version asks for the same $B$ independently
of conductor and height. A fixed half-plane supplies a uniform gap
between zeros and the line $\Rea s=1$; classical zero-free regions
instead narrow towards that line as conductor or height increases.
For general analytic background, see \cite[Chapter~5]{IK};
for the Dedekind zeta function of an abelian extension, and hence
its Hecke $L$-function factors, a precise zero-free region allowing
an exceptional real zero is
\cite[Theorem~3.1]{TZ}. Zero-density estimates bound the number
of zeros in a region; for example, see
\cite[Theorem~1.2]{GM}.

We follow the 30 September 2026 preprint \cite{QRH}, retaining its
numbering in all citations to it. The companion \cite{Alternate}
gives a different $11/12$ argument. OpenAI's public \texttt{math} repository has a formalization
scope note \cite[Scope]{LeanScope} that lists the $7/8$ statements for
zeta, Dirichlet and finite-order Hecke functions, with principal
poles excluded; the paper's later applications are outside that scope.

Let $F=\mathbb Q(\sqrt{-3})$. Let $\mathcal Z_F$ be the set of nontrivial
zeros of all primitive finite-order Hecke $L$-functions over $F$, of
arbitrary conductor; poles are excluded. Put
\[
 \beta_*:=\sup\bigl(\{1/2\}\cup\{\Rea\rho:\rho\in\mathcal Z_F\}\bigr).
\]
We study the dependence of the cubic-theta argument on its length
parameters and on the moment parameter $\kappa$. The stated $7/8$
bound of \cite{QRH} is the starting point for three deductions from
its intermediate estimates. The strongest deduction is the following.

\begin{mainresult}
The family supremum satisfies
\[
 \beta_*\le B_{\mathrm{new}}
 :=\frac{18\xi_{\mathrm{new}}+7}{18\xi_{\mathrm{new}}+9}
 =0.874957019420098946\ldots,
\]
where $\xi_{\mathrm{new}}\in(0.38858,0.38859)$ is the unique root there of
\[
 1188\xi^3-2160\xi^2+171\xi+190=0.
\]
Consequently every finite-order Hecke $L$-function over $F$ and every
Dirichlet $L$-function, including $\zeta$, is zero-free for
$\Rea s>B_{\mathrm{new}}$. Principal poles at $s=1$ are allowed.
\end{mainresult}
The main theorem is proved in Theorem~\ref{thm:iv-limit}.
In the original proof of OpenAI, there are several variables $b$, $\ell$, and $3/4\le\kappa\le1$. 
Part II first varies $\ell$ while retaining both $b$ and
$3/4\le\kappa\le1$, obtaining
\[
 B_r=\frac{34999}{40000}
 =0.874975.
\]
Part III then varies both variables $b$ and $\ell$ while retaining
$3/4\le\kappa\le1$, obtaining
\[
 B_{\mathrm{alg}}=\frac{1507-2\sqrt{921}}{1653}
 =0.874957069799\ldots.
\]
Part IV extends the unweighted fourth-moment estimate to smaller
values of $\kappa$, reoptimizes the lengths and iterates the resulting
bounds to obtain $B_{\mathrm{new}}$. 
We prove that $B_{\mathrm{new}}$ is optimal in the extended system, including
variation of its auxiliary parameters, in
Theorem~\ref{thm:iv-system-optimality}.

\paragraph{\it The choices of $b$ and $\ell$.}
In the original $7/8$ argument the lengths are chosen with
$b=1/8$ and $\ell=1/6$. Here $b=l_y-l_x$ is the difference
between the two summation exponents, and $\ell$ is the total exponent
of the selected prime lengths. Let $Z\ge2$ be a real variable tending to infinity.
With balanced averaging lengths,
\[
 X=Z^{l_x},\qquad Y=Z^{l_y},\qquad
 l_x=\frac{1-b-\ell}{2},\qquad
 l_y=\frac{1+b-\ell}{2}.
\]
The prime-norm parameters $P_i$ satisfy $\prod_iP_i=Z^\ell$.
For the admissible parameter choices used below and every $\eps>0$, the
direct bound and the exponent of the principal Mellin term are
respectively
\[
 |J_\eta(Z)|\ll_\eta Z^{l_x/2+b/12+\eps},\qquad
 C_b(s)=s-\frac{4+b}{6},
\]
where $J_\eta$ is the normalized cubic-theta average constructed in
Part I and generalized in Part II. Matching these exponents gives
\begin{equation}\label{eq:intro-length-boundary}
 C_b(B)=\frac{l_x}{2}+\frac b{12}
 \quad\Longleftrightarrow\quad
 B=B(\ell):=\frac{11}{12}-\frac\ell4.
\end{equation}
Thus increasing $\ell$ lowers the comparison boundary. The allowable
increase is constrained by the nonprincipal terms, whose exponents
also depend on $b$.
For the following parameter comparisons, write
$E(b,\ell;\delta,x)=E_{B(\ell)}(h)$ for the relative nonprincipal
exponent in \eqref{eq:balanced-extension-exponent}, with
$h=(1+b+3\ell)/2$ and the character-count exponent $R_*$ of
\eqref{eq:fixed-count-parameters}. Here $\delta=2a-1$ is the
zero-class parameter, and $x=\bar g/\delta$ measures the normalized
prime-sum size, as defined in Section~\ref{sec:counts}; the polynomial
comparison uses $(\delta,x)\in[0,5/6]\times[0,1/2]$.
These comparisons use the fixed moment parameter $\kappa=3/4$.
The notation $\partial_\ell E$ denotes partial differentiation of
$E(b,\ell;\delta,x)$ with respect to $\ell$, holding $b,\delta,x$
fixed; the dependence of $B(\ell)$ and $h=(1+b+3\ell)/2$ on $\ell$
is included in this differentiation.

For the first improvement we keep the original length difference and
increase the total prime length by $1/10000$:
\[
 b_{\mathrm r}=\frac18,\qquad
 \ell_{\mathrm r}=\frac16+\frac1{10000},\qquad
 B_{\mathrm r}=B(\ell_{\mathrm r})
 =\frac78-\frac1{40000}=\frac{34999}{40000}.
\]
The rational choice is deliberately conservative: $b_{\mathrm r}$
retains the original difference between the summation exponents,
while $1/10000$ is a convenient perturbation rather than a distinguished
or maximal value. Lemma~\ref{lem:rational-saving} establishes the uniform
bounds
\[
 E(1/8,1/6;\delta,x)\le-\frac1{5000},\qquad
 \partial_\ell E(1/8,\ell;\delta,x)\le\frac{45}{32}.
\]
Consequently, at $\ell=\ell_{\mathrm r}$,
\[
 E(1/8,\ell_{\mathrm r};\delta,x)
 \le-\frac1{5000}+\frac{45}{32\cdot10000}
 =-\frac{19}{320000}<0.
\]
Thus the improvement retains a positive saving in the error exponent.
Part II proves these bounds by an exact calculation in the Bernstein basis polynomials
(defined in \eqref{eq:bernstein-basis},
in the proof of Lemma~\ref{lem:rational-saving}; see also
\cite{Bernstein1912,Farouki2012}) and verifies the remaining contour
and length conditions.

For the second improvement we vary both parameters:
\begin{align*}
 b_{\mathrm{alg}}&=-\frac4{29}+\frac{230\sqrt{921}}{26709},&
 \ell_{\mathrm{alg}}&=\frac{33+8\sqrt{921}}{1653},\\
 B_{\mathrm{alg}}&=B(\ell_{\mathrm{alg}})
 =\frac{1507-2\sqrt{921}}{1653}.&&
\end{align*}
The algebraic values arise from the critical error class and a tangency
condition. At $x=1/2$, the count exponent is
\[
 R_*(\delta,1/2)
 =\frac{185-288\delta+96\delta^2}{185-138\delta},
 \qquad \partial_b E=\frac{R_*}{2}-\frac13.
\]
Hence a class with $R_*=2/3$ imposes a constraint that cannot be
relaxed by changing $b$. Its admissible parameter is the smaller root
\[
 288\delta^2-588\delta+185=0,\qquad
 \delta_{\rm c}=\frac{49-\sqrt{921}}{48}.
\]
In particular, $\sqrt{921}$ comes from the discriminant
$588^2-4\cdot288\cdot185=144\cdot921$, rather than from an arbitrary
choice. At this class, $E=0$ determines
\[
 \ell=\frac{5-6\delta_{\rm c}}{9+18\delta_{\rm c}}
 =\ell_{\mathrm{alg}},
\]
and the stationary condition
\[
 \partial_\delta E(b,\ell_{\mathrm{alg}};\delta_{\rm c},1/2)=0
\]
determines $b=b_{\mathrm{alg}}$.
These values balance the direct comparison with the largest
nonprincipal exponent. The balancing equations locate the candidate;
their solution alone does not establish the estimate for other classes.
Part III derives the critical zero class,
proves the endpoint inequality for every permitted zero and amplitude
class by an exact quadratic identity, and verifies all auxiliary
strict auxiliary inequalities. The positive linear combination of inequalities in
Proposition~\ref{prop:scalar-optimality} explains the algebraic
obstruction suggested by these parameter choices.
The notation $\mathrm{alg}$ distinguishes this bound from the rational one.

\paragraph{\it Varying the moment parameter.}
We vary the parameter $\kappa$ governing the admissible lengths in the
mixed fourth-moment estimate for unweighted character polynomials.
Part IV extends the unweighted moment from $3/4\le\kappa\le1$
to $13/18\le\kappa\le1$, retaining the indispensable hypothesis
\[
 \beta_*\le(1+\kappa)/2\qquad(\kappa<1).
\]
The bound from Part III therefore permits
$\kappa=2B_{\mathrm{alg}}-1<3/4$.
Decreasing $\kappa$ increases the available selected-prime length in
the unweighted moment, improves the nonprincipal count, and permits
a further increase of $\ell$. The direct boundary $B(\ell)$ would
remain $B_{\mathrm{alg}}$ if $\ell_{\mathrm{alg}}$ were kept fixed.
Thus the improvement uses changed length parameters as well as
a changed moment parameter.

For a critical class $\delta=\xi$ and $x=1/2$, the count equation
$R_\kappa=2/3$ becomes
\[
 (216\kappa-18)\xi^2-(468\kappa-57)\xi+150\kappa-20=0.
\]
The equations $E=0$ and $\partial_\delta E=0$ then give
\[
 \ell(\xi)=\frac{5-6\xi}{9+18\xi},\qquad
 b(\xi)=\frac{8(3\xi-1)(12\xi^2-40\xi+15)}
                   {9(2\xi+1)(24\xi^2-20\xi+5)},\qquad
 B(\xi)=\frac{18\xi+7}{18\xi+9}.
\]
Part IV proves the nonprincipal inequality uniformly in every
zero and amplitude class, checks the remaining strict analytic inequalities,
and obtains a new bound from each previously established
moment hypothesis. Iteration reaches the limiting equality
$\kappa_{\mathrm{new}}=2B_{\mathrm{new}}-1$; eliminating $\kappa$ gives the cubic
in the main theorem. Only the resulting real upper bounds are
passed to their limit, so uniform analytic constants over the
iteration are not required.

The limiting root has the explicit form
\[
 \xi_{\mathrm{new}}=\frac{20}{33}+\frac{\sqrt{1391}}{33}
 \cos\!\left(\frac13\arccos\frac{28470}{1391^{3/2}}
                         -\frac{2\pi}{3}\right),
 \qquad \arccos\in[0,\pi].
\]
Substitution in the balanced length relations gives the exact limiting parameters
\[
\begin{aligned}
 b_{\mathrm{new}}
 &=\frac{8(3\xi_{\mathrm{new}}-1)(12\xi_{\mathrm{new}}^2-40\xi_{\mathrm{new}}+15)}
          {9(2\xi_{\mathrm{new}}+1)(24\xi_{\mathrm{new}}^2-20\xi_{\mathrm{new}}+5)}
   =0.123405685419541002\ldots,\\
 \ell_{\mathrm{new}}
 &=\frac{5-6\xi_{\mathrm{new}}}{9+18\xi_{\mathrm{new}}}
   =0.166838588986270882\ldots,\\
 \kappa_{\mathrm{new}}
 &=\frac{18\xi_{\mathrm{new}}+5}{18\xi_{\mathrm{new}}+9}
   =2B_{\mathrm{new}}-1=0.749914038840197892\ldots.
\end{aligned}
\]
The root formula is proved in Lemma~\ref{lem:iv-explicit-root};
the parameter identities follow from \eqref{eq:iv-geometry} and
\eqref{eq:iv-fixed-point-identity}.
In particular,
\[
 B_{\mathrm{alg}}-B_{\mathrm{new}}
 =5.037906979388974\ldots\times10^{-8}.
\]
The main structural contribution of Part IV is
Theorem~\ref{thm:iv-system-optimality}: it determines the exact optimum
of the explicitly specified exponent comparison system, allowing
free total length, every permitted detector cutoff and contour choice,
and the full moment range $13/18\le\kappa\le1$.
The proof isolates a critical class for which every permitted detector
cutoff gives an upper-bound exponent at least $2/3$ in the counting estimate, forcing three
affine inequalities for any putative boundary $B<B_{\mathrm{new}}$.
An explicit positive weighted combination eliminates the length
parameters from these inequalities; the moment compatibility condition
$B\le(1+\kappa)/2$ then forces $B\ge B_{\mathrm{new}}$, and the limiting
construction shows that this lower bound is sharp.
This assertion concerns the supplied uniform upper bounds and does
not exclude improvements from stronger analytic estimates.
The limiting pair $(b_{\mathrm{new}},\ell_{\mathrm{new}})$ consists of cubic
irrationals, whereas the rational parameter choice does not attain the limit.

All three improvements retain the exact arithmetic formulas and principal
residues of \cite{QRH}. Their Euler correction is shown to be
holomorphic below $7/8$. In Parts II and III, the original conclusion
$\beta_*\le7/8$ permits $\kappa=3/4$, removing the loss associated
with the original varying parameter. In Part IV each new moment
hypothesis is supplied by the preceding proved bound. The strict exponent comparisons then give the same Mellin-continuation contradiction as in
Part I. The analytic dependencies are indexed in
Appendix~\ref{app:inputs-index}.

Here is the structure of the paper. Part I recalls the $7/8$ argument and fixes the arithmetic and analytic
notation used throughout. Part II establishes the shared
changed-parameter estimates and gives the simpler rational extension $B_r$. 
Its calculation in the Bernstein basis polynomials is retained as a separate rational example;
this bound is superseded by the stronger one in Part III.
Part III varies both length parameters, proves the error-exponent
inequality at the algebraic parameter choices,
and discusses the resulting algebraic lower bound $B_{alg}$. Part IV proves
the smaller-$\kappa$ moment estimate, derives the changed count and
length parameters, and iterates the analytic comparison to $B_{\mathrm{new}}$.
All three extensions use Proposition~\ref{prop:criterion-summary}
and the Dirichlet transfer of Corollary~\ref{cor:dirichlet-conclusion}.

\subsection*{Acknowledgments}
OpenAI Codex (GPT 6.1 Sol Medium-Ultra) assisted with drafting, checking references, and the search
for improved bounds. Anthropic Claude (Opus 5.5 Medium) assisted with auditing. The
zero-free conclusions for $B_{\mathrm r}$, $B_{\mathrm{alg}}$ and $B_{\mathrm{new}}$
have Lean~4 formalizations for the finite-order Hecke family over $F$
and its Dirichlet and zeta consequences.
The Lean formalization, exact-arithmetic certificates and verification instructions are
available in the accompanying
\href{https://github.com/liubaiying101/Slightly-improved-zero-free-half-planes-for-the-quasi-Riemann-hypothesis}{GitHub repository} \cite{LeanExtension}.

\part{Summary of the \texorpdfstring{$7/8$}{7/8} proof}
\section{The original argument and its two-stage reduction}
\label{sec:reduction}

In this section, we briefly describe the two representations of the cubic-theta average, i.e., the smoothed arithmetic average \(I_\eta(X,Y,Z)\) defined in \eqref{eq:base}, formed from the Fourier coefficients of the cubic theta function, 
and the two-stage reduction leading to the $7/8$ bound.
The first stage proves $11/12$ using the unmodified average;
the second introduces selected prime factors and stronger moment
estimates to reach $7/8$.

The argument compares two descriptions of one cubic-theta average:
the divisibility-restricted form of the cubic-theta functional equation
\eqref{eq:theta-reflection} (\Src{Corollary}{14.1}) gives the theta
mean-square estimate \eqref{eq:BJ-bound}. Together with the additive
mean-square estimate \eqref{eq:Gram}, the Cauchy--Schwarz inequality
\eqref{eq:separated-cauchy-schwarz} gives the direct bound \eqref{eq:direct}.
Poisson summation \eqref{eq:Poisson}, applied to the modification
\eqref{eq:modified}, isolates the principal Mellin integral \eqref{eq:f}
involving $1/L_F(s,\eta)$.
The simultaneous large values \eqref{eq:large-values} forced by a zero,
together with the moment bounds \eqref{eq:inverse-moment} and
\eqref{eq:plain-moment}, control the remaining Poisson terms.
Part I makes those dependencies, the common-character hypotheses and
the numerical exponent comparisons explicit. Its statements record
the original argument using the proved results indexed in
Appendix~\ref{app:inputs-index}. The accompanying explanations
retain the decisive identities, exponent comparisons and contradiction,
including calculations used in Parts II--IV. Complete proofs of the
imported results are given at the cited locations in \cite{QRH}.
Numbered statements have this paper's numbering;
locators inside citations always refer to the cited work.

\begin{theorem}[Main statement of \Src{Theorem}{1.1}]
\label{thm:main-summary}
Every finite-order Hecke $L$-function over
$F=\mathbb Q(\sqrt{-3})$ has no zeros in
\begin{equation}\label{eq:goal}
 \Rea s>7/8.
\end{equation}
The same holds for every Dirichlet $L$-function, including
$\zeta$. Principal poles at $s=1$ are allowed; the boundary
$\Rea s=7/8$ is not included.
\end{theorem}
The Hecke and Dirichlet conclusions are proved below in
Theorem~\ref{thm:hecke-conclusion} and Corollary~\ref{cor:dirichlet-conclusion}.
Here $L_F(s,\eta)$ denotes the Hecke
$L$-function over $F$, and $\zeta(s)$ the Riemann zeta function;
$\eta$ denotes the character fixed in each application.

For the primitive Hecke family define
\begin{equation}\label{eq:beta}
 \beta_*=\sup\left(\{1/2\}\cup
 \left\{\Rea\rho:
 \begin{array}{l}
 \Lfun(\rho,\eta)=0,\quad 1/2\le\Rea\rho\le1,\\
 \eta\text{ primitive and finite-order}
 \end{array}\right\}\right).
\end{equation}
Poles are excluded. The inserted value $1/2$ gives $\beta_*\ge1/2$;
absolute convergence of the Euler product gives $\beta_*\le1$.
The family in \eqref{eq:beta} comprises \emph{all primitive
finite-order Hecke characters over $F$, of arbitrary conductor}.
Indeed, Poisson summation produces denominators
$L_F^S(s,\eta\overline{\chi_\bullet(u)})$ in
\eqref{eq:Poisson}. The primitive inducing characters of these
twists belong to that family, so the same $\beta_*$ controls their zeros.

\paragraph{\it The two stages.}
Stage 1 proves $\beta_*\le11/12$.
Stage 2 assumes, for contradiction, $\beta_*>7/8$ and sets
\begin{equation}\label{eq:two-stages}
 \Delta=\beta_*-7/8\in(0,1/24],\qquad
 \kappa=2\beta_*-1=3/4+2\Delta\le5/6.
\end{equation}
Group the nonprincipal terms by the location of the zeros of
their associated twists; these groups are the zero classes
of Section~\ref{sec:detector}. In the second stage they satisfy
\[
 \delta=2a-1\le\kappa\le\alpha:=5/6.
\]
This restriction is essential in the original paper's character-count
argument: the long-polynomial bound contains
$(\alpha-\delta)r$, and its comparison with $r\le t$ uses
$\alpha-\delta\ge0$. Thus the preliminary theorem cannot be
omitted from this proof without replacing that estimate.

\paragraph{\it The parameter choice giving $7/8$.}
The normalized modified average has a direct bound $Z^{L+\eps}$,
while its principal residue has Mellin exponent $C(s)$, where
\[
 L=\frac{l_x}{2}+\frac{b}{12},\qquad
 b=l_y-l_x,\qquad
 C(s)=s+\frac{l_x}{3}-\frac56+\frac\ell6.
\]
Under $l_x+l_y=1-\ell$, matching these exponents gives
\[
 C(B)=L\quad\Longleftrightarrow\quad B=\frac{11}{12}-\frac\ell4.
\]
The mean-square calculation using the functional equation in
\cite[Lemma~15.1, proof]{QRH} contains the algebraic excess
\[
 \left(\frac{1+3\ell'-2M'}4\right)_+
 =\left(\frac{d_J+5\ell-1}4\right)_+,
 \qquad \ell'=\ell-d_J,\quad M'=1-\ell-2d_J.
\]
For $0\le d_J\le\ell$, this term vanishes for every subset
exactly when $6\ell\le1$. The convenient choice $\ell=1/6$
therefore gives $B=7/8$ in the exponent matching above.
With the paper's parameter choices \eqref{eq:geometry},
\[
 B=7/8,\qquad C(s)=s-11/16,\qquad C(7/8)=L=3/16.
\]
This exponent matching motivates the parameters; proving a strict
saving for every nonprincipal contribution requires both large
values from the zero detector and the two moment estimates.
That saving continues every reciprocal $1/L_F(s,\eta)$ to the
left of $\beta_*$, producing the contradiction.

Sections~\ref{sec:arithmetic}--\ref{sec:average} introduce the
characters, continuation criterion and theta average.
Sections~\ref{sec:detector}--\ref{sec:first} give the detector and
the $11/12$ stage. Sections~\ref{sec:modified}--\ref{sec:parameters}
develop the modified average, moments, counts and exponent savings.
Sections~\ref{sec:contradiction}--\ref{sec:transfer} complete the
continuation and Dirichlet transfer.

\section{Arithmetic conventions}
\label{sec:arithmetic}

In this section, we fix the arithmetic conventions used in every subsequent construction,
including character presentations, zero extensions and reciprocity phases.
The family zero supremum $\beta_*$ used in the continuation argument
is defined in \eqref{eq:zeros}.
These choices determine the local factors of the theta average in
Section~\ref{sec:average} and must be retained when its parameters vary.

Set
\[
 \omega=e^{2\pi i/3},\quad
 \OO=\mathbb Z[\omega],\quad \lambda=\sqrt{-3},\quad
 q_a=|a|^2,\quad q_{\mathfrak a}=N\mathfrak a,\quad
 \varpi(a)=a/|a|.
\]
$\OO$ is norm-Euclidean. Ideals coprime to $3$ have unique
primary generators $a\equiv1\pmod3$. Ideal sums use these
generators; sums over elements may also include units.
Unless stated otherwise, ideal sums run over nonzero integral ideals.
A sum over $x\bmod c$ runs over residue classes in $\OO/(c)$.
For an integral ideal $\mathfrak a$, $N\mathfrak a=|\OO/\mathfrak a|$.
Write $v_p(a)$ for its prime valuation, $\rad(a)=\prod_{p\mid a}p$,
and $\mu$ for the ideal M\"obius function. For a nonzero integral ideal,
\[
 \mu(\mathfrak a)=
 \begin{cases}
  1,&\mathfrak a=\OO,\\
  (-1)^k,&\mathfrak a=\mathfrak p_1\cdots\mathfrak p_k
          \text{ with distinct prime ideals }\mathfrak p_j,\\
  0,&\mathfrak p^2\mid\mathfrak a\text{ for some prime ideal }\mathfrak p.
 \end{cases}
\]
When an ideal is represented by a primary generator $n$, we write
$\mu(n)=\mu((n))$. Thus M\"obius coefficients restrict a sum to
squarefree ideals and assign the sign $(-1)^k$ according to the
number of their prime factors. The notation $\sfq$ in a sum
means squarefree, and $1_E$ is the indicator of a condition $E$.
Thus $(a,b)=1$ means that the corresponding ideals are coprime.
A prime outside $S$ is called good.

Unless specified otherwise, $Z\ge2$ is a real variable tending to infinity;
$\eps,\eps_1>0$ are arbitrarily small losses. For $a\in\mathbb C$
and $b\ge0$, the notation $a\ll b$, read as ``$a$ is bounded by
a constant times $b$'', means $|a|\le Cb$ for some $C>0$.
The absolute value is unnecessary when $a$ is already nonnegative.
A subscript specifies allowed dependence of the constant: for example,
\[
 a\ll_\eps b\quad\Longleftrightarrow\quad
 |a|\le C_\eps b,
\]
where $C_\eps>0$ may depend on $\eps$ but is independent of the
other varying quantities; the notation $\ll_S$ has the analogous
meaning for the fixed prime set $S$.
For nonnegative $a,b$, the notation $a\gg b$, read as ``$a$ is
bounded below by a positive constant times $b$'', means $b\ll a$.
The notation $a\asymp b$, read as ``$a$ is comparable to $b$'', means
\[
 cb\le a\le Cb\qquad\text{for some }c,C>0.
\]
These constants are independent of the varying quantities, subject to
any indicated dependence; comparability does not assert that $a/b$
is constant. We also write $x_+=\max(x,0)$.

Throughout the paper, we often express an actual summation length
$N>0$ as a power of the common real variable $Z$, and use its
\emph{length exponent} $r$:
\begin{equation}\label{eq:length-exponent-convention}
 N=Z^r,\qquad r=\frac{\log N}{\log Z}.
\end{equation}
For instance, a sum supported on norms $q_n\asymp N$ is then
written as supported on $q_n\asymp Z^r$. Products of lengths
correspond to sums of their exponents:
\[
 Z^{r_1}Z^{r_2}=Z^{r_1+r_2}.
\]
In estimates stated in this notation, ``length $r$'' means the
length exponent; the actual norm parameter is $Z^r$.

Auxiliary letters are used independently in separately stated
lemmas: in particular, $K$ in
\eqref{eq:hybrid} is a norm bound, whereas $K$ in
Section~\ref{sec:modified} counts prime intervals.
For $c\in\mathbb R$, the notation $\int_{(c)}F(s)\,ds$ denotes
integration upwards along the vertical line $\Rea s=c$, interpreted
as the limit of the integrals from $c-iT$ to $c+iT$ as $T\to\infty$.
The test functions used to restrict norms are smooth functions supported
on fixed compact subintervals of $(0,\infty)$; their logarithmic
derivatives are bounded. Constants $A,J$ in height estimates denote finite
nonnegative exponents or derivative orders fixed before the height cutoff;
their values may differ from one estimate to another.

Fix a finite set $S$ containing the prime ideals above $6$, all prime
ideals dividing the defining moduli of the fixed characters, and all
prime ideals of norm at most $P_0$, where $P_0 \in \mathbb{R}_{>0}$ is the fixed prime-norm cutoff. Unless stated otherwise, ideal sums are restricted to ideals coprime
to every prime ideal in $S$. These arithmetic data are fixed independently
of $Z$ and the summation variables; they may depend on an already fixed
character $\eta$. Denote these fixed moduli, characters, ray group,
exclusions and normalizations collectively by $\mathcal A$.

\begin{center}
\small
\begin{tabular}{@{}p{.24\textwidth}p{.69\textwidth}@{}}
\toprule Symbol & Meaning and scope\\\midrule
$\beta_*,\Delta,\kappa$ & Family zero supremum; $\Delta=\beta_*-7/8$;
$\kappa=2\beta_*-1$.\\
$\alpha=5/6,\ \varpi(a)$ & Fixed amplification slope;
unit phase $a/|a|$ (written $\alpha(a)$ in \cite{QRH}).\\
$\eb,\ \eps$ & Fixed zero-class width (written $e$ in \cite{QRH});
arbitrarily small exponent loss.\\
$\mathfrak a(r)$ & Amplified moment exponent
$\max\{1,(1+5r)/6\}$ (written $e(r)$ in \cite{QRH}).\\
$Z,U=Z^d,T_1$ & Real variable tending to infinity; dyadic norm parameter;
zero-detection height.\\
$l_x,l_y,b,h,\ell$ & Averaging exponents; $b=l_y-l_x$;
combined frequency exponent; total auxiliary prime length.\\
$z_{\rm cap},z_0$ & Upper bound for total prime length in \eqref{eq:canonical};
real contour value $z_0=17/50$ in Section~\ref{sec:exponents}.\\
$\mathfrak r$ & Coprimality modulus in \eqref{eq:canonical}
(written $R$ in \cite{QRH}); distinct from a counting exponent $R$.\\
$\mathcal R,\mathcal M$ & Reciprocity phase; Mellin transform
of the Fourier-transformed weight.\\
$H_\eta,J_\eta,f_\eta$ & Nonvanishing correction; normalized average;
principal-residue Mellin integral.\\
$\Gamma_1(3),\Gamma_2;\ \Gamma(s)$ & Matrix groups; Euler gamma function.\\
\bottomrule
\end{tabular}
\end{center}
\noindent Other indices and length parameters are local to the
statement in which they are introduced. A \emph{zero class} is a
set of elements satisfying the same pair of zero-location inequalities
\eqref{eq:zero-class}.

\paragraph{\it Hecke characters, moduli and conductors.}
For a nonzero integral ideal $\mathfrak f$, called the modulus, the ray class group is
$I(\mathfrak f)/P_1(\mathfrak f)$: $I(\mathfrak f)$ consists of
fractional ideals coprime to $\mathfrak f$, and $P_1(\mathfrak f)$
of principal ideals generated by elements congruent to $1$
modulo $\mathfrak f$. A finite-order Hecke character is a
homomorphism from this finite group to roots of unity.
Its conductor $\mathfrak f_*$ is the smallest modulus through which
it factors. A \emph{character presentation} specifies a defining modulus
$\mathfrak f$ divisible by $\mathfrak f_*$, together with extension
by zero on integral ideals not coprime to $\mathfrak f$. The presentation
is primitive when $\mathfrak f=\mathfrak f_*$ and imprimitive otherwise.
More precisely, a character $\psi$ with defining modulus $\mathfrak f$
is induced by a unique primitive character $\psi^*$ of conductor
$\mathfrak f_*\mid\mathfrak f$. If $\mathrm{Cl}_{\mathfrak f}(F)$
denotes the ray class group above, the natural surjection gives
\begin{equation}\label{eq:primitive-induction}
 \pi:\mathrm{Cl}_{\mathfrak f}(F)\longrightarrow
       \mathrm{Cl}_{\mathfrak f_*}(F),\qquad
 \psi=\psi^*\circ\pi.
\end{equation}
Thus $\psi(\mathfrak a)=\psi^*(\mathfrak a)$ whenever
$(\mathfrak a,\mathfrak f)=1$. Primitivity means that $\psi^*$
does not factor through the ray class group of any proper divisor
of $\mathfrak f_*$. The zero extensions can nevertheless differ:
$\psi(\mathfrak a)=0$ when $(\mathfrak a,\mathfrak f)\ne1$,
including when $(\mathfrak a,\mathfrak f_*)=1$.
Thus the defining modulus records the coprimality restriction attached
to this presentation, including primes absent from the conductor; see \cite[Section 4]{QRH}.
Initially,
\[
 L_F(s,\eta)=\sum_{\mathfrak a\ne0}
       \eta(\mathfrak a)(N\mathfrak a)^{-s}
   =\prod_p(1-\eta(p)q_p^{-s})^{-1},
 \qquad \Rea s>1;
\]
outside this region the notation refers to its meromorphic
continuation. A norm twist multiplies a character by
$\mathfrak a\mapsto(N\mathfrak a)^{it}$, $t\in\mathbb R$.

For a prime $p\notin S$, the sextic residue character is
\[
 \chi_p(a)=(a/p)_6,\qquad
 \chi_p(a)\equiv a^{(q_p-1)/6}\pmod p\quad(p\nmid a),
 \qquad \chi_p(a)=0\quad(p\mid a).
\]
For a primary element $c$, meaning $c\equiv1\pmod3$, write
$c=\prod_p p^{v_p(c)}$ with primary prime generators $p$, and put
$\chi_c(a)=\prod_{p\mid c}\chi_p(a)^{v_p(c)}$.
Throughout, character powers are first taken on invertible residue
classes and then extended by zero on the remaining classes. This
convention applies to every integer exponent, including zero, negative
exponents and multiples of six:
\begin{equation}\label{eq:zeros}
 \chi_c(a)^6=1_{(a,c)=1}.
\end{equation}
Passing to a primitive inducing character therefore requires
retaining the coprimality indicator for the original defining modulus. Coefficients and
such conditions are independent of the variable being averaged,
except for the explicitly displayed residue symbols.

The additive normalization is
\[
 e(z)=\exp\!\left(2\pi i\operatorname{Tr}_{F/\mathbb Q}(z/\lambda)\right),
 \qquad d\mu(z)=\frac2{\sqrt3}\,d\xi_1\,d\xi_2.
\]
Here $z=\xi_1+i\xi_2$ with $\xi_1,\xi_2\in\mathbb R$;
for $z\in\mathbb C$ the trace in $e(z)$ means
$z/\lambda+\overline{z/\lambda}$. The pairing $(z_1,z_2)\mapsto e(z_1z_2)$
makes $\OO$ self-dual. For squarefree $c$ define
\[
 \gamma_j(c)=q_c^{-1/2}\sum_{v\bmod c}\chi_c(v)^j e(v/c).
\]
Fixed-numerator symbols $n\mapsto\chi_n(u)$ are finite-order
ray characters \Src{Lemma}{4.1}. Prime Gauss identities and
sextic reciprocity provide root-of-unity-valued functions $\mathcal R,G$
that depend only on fixed ray classes.
Precisely, $\mathcal R$ is the bicharacter on residue square
classes modulo $4$ defined in \Src{Lemma}{4.4}; on coprime
primary pairs it satisfies
\[
 \chi_b(a)=\mathcal R(a,b)\chi_a(b).
\]
The same lemma defines $G$ on all primary indices, with $|G|=1$.
Its value depends only on the ray class of its argument in a fixed
finite ray quotient; on squarefree $c$,
$G(c)=\overline{\chi_c(4)}\gamma_3(c)$. Its extension to
nonsquarefree indices is that finite-ray function, rather than a
Gauss sum, and satisfies
\[
 G(AB)=G(A)G(B)\mathcal R(A,B)
\]
on the admissible primary elements \Src{Lemmas}{4.2--4.4}.
These functions cancel the reciprocity factors linking distinct primes,
so that the coefficients can be expressed as an Euler product in
Section~\ref{sec:average}.

\section{A criterion for zero-free continuation}
\label{sec:criterion}

In this section, we isolate the analytic implication that converts the two estimates for
an average into a zero-free half-plane.
Proposition~\ref{prop:criterion-summary} compares the direct bound
\eqref{eq:criterion-direct} with the approximation by the principal
Mellin integral in \eqref{eq:criterion-comparison}.
All four parts use this criterion; their remaining task is to construct
the average and verify its hypotheses with a saving independent of the
target character.

For a finite excluded set,
\[
 \Lfun^S(s,\eta)=\Lfun(s,\eta)
       \prod_{p\in S}(1-\eta(p)q_p^{-s}).
\]
The deleted factors are nonzero for $\Rea s>0$.

The formulation below is taken from \Src{Proposition}{2.1}.
Its proof uses the classical principle that a growth bound for an
inverse Mellin transform yields holomorphic continuation of a reciprocal
$L$-function, and hence zero-freeness; compare Riesz's criterion
\cite[Sections~3--4]{Riesz1916}. Here the positive saving must be
uniform over the target characters, whereas the implied constants
may depend on the character.

\begin{proposition}[Continuation criterion, \Src{Proposition}{2.1}]
\label{prop:criterion-summary}
Fix $1/2<\sigma_0<1$ and $C(s)=s+c$, $c\in\mathbb R$.
Suppose that, whenever $\beta_*>\sigma_0$, the following
construction is available.
For every primitive $\eta$, suppose $H_\eta$ is holomorphic on
$\Rea s>\sigma_0$ and
\begin{equation}\label{eq:contraction}
 \sup_{\Rea s>\sigma_0}|H_\eta(s)-1|\le1/2.
\end{equation}
Define
\begin{equation}\label{eq:general-f}
 f_\eta(Z)=\frac1{2\pi i}\int_{(2)}
 Z^{C(s)}e^{(s-5/6)^2}\frac{H_\eta(s)}{\Lfun^S(s,\eta)}\,ds.
\end{equation}
Here $(v)$ denotes a vertical line of real part $v$.
Assume there is a normalized average $J_\eta$ satisfying
\begin{align}
 |J_\eta(Z)|&\ll_\eta Z^{C(\sigma_0)+\omega_0},
 &0<\omega_0<\beta_*-\sigma_0,\label{eq:criterion-direct}\\
 |J_\eta(Z)-f_\eta(Z)|&\ll_\eta Z^{C(\beta_*)-\sigma},
 &\sigma>0,\label{eq:criterion-comparison}
\end{align}
The positive numbers $\omega_0,\sigma$ are independent of $\eta$;
constants, finite excluded sets and lower thresholds may depend on $\eta$.
Then $\beta_*\le\sigma_0$.
\end{proposition}

\paragraph{\it Idea of the continuation argument.}
Assume $\beta_*>\sigma_0$ and put
\[
 \epsilon_c=\min\{\beta_*-\sigma_0-\omega_0,\,
                  \sigma,\,(\beta_*-\sigma_0)/2\}>0.
\]
The estimates \eqref{eq:criterion-direct} and \eqref{eq:criterion-comparison} give
$f_\eta(Z)\ll_\eta Z^{C(\beta_*)-\epsilon_c}$ for $Z\ge1$.
Moving \eqref{eq:general-f} to any fixed line $\Rea s=b_c>2$
gives $f_\eta(Z)\ll_{\eta,b_c}Z^{b_c+c}$ for $0<Z\le1$.
Consequently
\[
 F_\eta(s)=\int_0^\infty f_\eta(Z)Z^{-s-c}\,\frac{dZ}{Z}
\]
is holomorphic on $\Rea s>\beta_*-\epsilon_c>\sigma_0$.
Mellin inversion and analytic continuation identify
\[
 F_\eta(s)=e^{(s-5/6)^2}H_\eta(s)/L_F^S(s,\eta)
 \qquad(\Rea s>1)
.
\]
Since $|H_\eta|\ge1/2$, the function
\[
 e^{-(s-5/6)^2}F_\eta(s)/H_\eta(s)
\]
is a holomorphic continuation of $1/L_F^S(s,\eta)$ to
$\Rea s>\beta_*-\epsilon_c$, for every primitive $\eta$.
The supremum definition supplies some $\eta_0$ and a zero $\rho$ there;
attainment of $\beta_*$ is unnecessary. Locally,
\[
 L_F^S(s,\eta_0)=(s-\rho)^{k_\rho}g(s),\qquad
 k_\rho\ge1,\quad g(\rho)\ne0.
\]
Its reciprocal has a pole, contradicting this holomorphic continuation.
Neither the Gaussian nor $H_{\eta_0}$ can cancel it. Principal poles,
in contrast, give zeros of the reciprocal and cause no obstruction.

The two applications in Part I use
\[
 \begin{array}{c|c|c}
 \sigma_0&C(s)&C(\sigma_0)\\ \hline
 11/12&s-2/3&1/4\\
 7/8&s-11/16&3/16
 \end{array}
\]
and different normalized averages.

\section{Cubic theta series and the smoothed average}
\label{sec:average}

In this section, we construct the cubic-theta average whose two representations enter
Proposition~\ref{prop:criterion-summary}.
The cubic-theta functional equation \eqref{eq:theta-reflection} supplies
the mean-square estimate \eqref{eq:unmarked} for the direct bound,
while Poisson summation \eqref{eq:Poisson} exposes the reciprocal $L$-function and the
nonprincipal terms to be controlled in the following sections.
These representations first give the $11/12$ bound and then, after
the modification in Section~\ref{sec:modified}, the $7/8$ bound;
Parts II--IV retain this construction with changed parameters.

\subsection{The cubic theta function and its Fourier coefficients}
\label{sec:theta-definition}

In this subsection, we introduce the cubic theta function and the Fourier
coefficients used to construct the smoothed average.
The transformation law leads to the functional equation
\eqref{eq:theta-reflection} for the theta sums, while the explicit coefficients connect those sums to
the cubic Gauss sums in \eqref{eq:completed}.

The function used here is \emph{Kubota's cubic metaplectic theta
function} $\theta_{\rm cub}$ on
$\mathbb H^3=\mathbb C\times\mathbb R_{>0}$.
Kubota developed the underlying metaplectic construction
\cite{Kubota}; Patterson determined the explicit cubic theta
coefficients \cite[Theorem~8.1]{Patterson}.
We use the cubic metaplectic theta function in the normalization of
Dunn--Radziwi{\l}{\l} \cite[Section~5.1 and Appendix~A]{DR},
recording below the transformation law and Fourier coefficients needed
in our argument, in the form stated in \Src{Proposition}{5.1}.
Here \emph{cubic} refers to the cubic residue symbol appearing in the
transformation law of $\theta_{\rm cub}$ displayed below. Set
\[
 \Gamma_1(3)=\{g\in{\rm SL}_2(\OO):g\equiv I\pmod3\},
 \qquad
 \Gamma_2=\langle{\rm SL}_2(\mathbb Z),\Gamma_1(3)\rangle.
\]
For $g=\left(\begin{smallmatrix}a&b\\c&d\end{smallmatrix}\right)
\in{\rm SL}_2(\mathbb C)$,
its action on $w=(z,v)\in\mathbb H^3$ is
\[
 g(z,v)=
 \left(
 \frac{(az+b)\overline{(cz+d)}+a\bar c\,v^2}
      {|cz+d|^2+|c|^2v^2},
 \frac{v}{|cz+d|^2+|c|^2v^2}
 \right).
\]
The transformation law is
\[
 \theta_{\rm cub}(gw)=\kappa_{\rm cub}(g)\theta_{\rm cub}(w),
 \qquad
 \kappa_{\rm cub}(g)=
 \begin{cases}
 (c/a)_3,&g\in\Gamma_1(3),\ c\ne0,\\
 1,&g\in\Gamma_1(3),\ c=0,\ \text{or }g\in{\rm SL}_2(\mathbb Z).
 \end{cases}
\]
The factor $\kappa_{\rm cub}(g)$ extends to $\Gamma_2$ as a homomorphism
\cite[Section~2]{Patterson};
the action and transformation law above are
\cite[equations~(5.1)--(5.4)]{DR}. The symbol $(c/a)_3$ is the cubic residue
symbol, multiplicative in its primary denominator; at good
primes $(x/p)_3=\chi_p(x)^2$. The notation $\kappa_{\rm cub}$
distinguishes this transformation factor from the real zero-bound parameter
$\kappa$.

The Fourier expansions needed in the proof are those at the
three cusps represented by
\[
 H_j=\begin{pmatrix}1&0\\t_j&1\end{pmatrix},
 \qquad(t_0,t_1,t_2)=(0,\omega^2,\omega),\qquad j=0,1,2.
\]
Each $H_j$ sends $\infty$ to a boundary cusp in
$F\cup\{\infty\}$; expanding $\theta_{\rm cub}\circ H_j$ gives
its Fourier expansion there.
We take the complex conjugate so that the Fourier coefficients involve
the cubic Gauss sums $\widetilde g_3(c)$ used below, rather than their
complex conjugates; Appendix~\ref{app:theta-conventions} gives the precise
coefficient conversion.
With
$\check e(x)=\exp(2\pi i(x+\bar x))$, these expansions are
\[
 \overline{\theta_{\rm cub}(H_j(z,v))}
 =C_j(v)+
 \sum_{\substack{\mu\in\lambda^{-4}\OO\\\mu\ne0}}
 d_j(\mu)\,vK_{1/3}(4\pi|\mu|v)\check e(\mu z).
\]
Here
\[
 C_0(v)=\frac{3^{5/2}}2v^{2/3},\qquad C_1(v)=C_2(v)=0.
\]
The symbol $d_j$ denotes the fixed
arithmetic Fourier-coefficient function, and $K_{1/3}$ the
modified Bessel function of the second kind. The precise correspondence with the notation of
\cite{DR} is given in Appendix~\ref{app:theta-conventions}.
In particular, for primary squarefree $c$ and primary $n$,
\[
 d_0(\lambda^{-3}cn^3)=3^{5/2}|n|\,\widetilde g_3(c),\qquad
 \widetilde g_3(c)=q_c^{-1/2}\sum_{x\bmod c}
                  (x/c)_3\check e(x/c).
\]
When $(c,S)=1$ this becomes
\[
 \widetilde g_3(c)=\chi_c(\lambda)^{-2}\gamma_2(c).
\]
Thus the squarefree Gauss coefficients and the cube multiples
in the following Dirichlet series come directly from these
Fourier coefficients. The automorphy law supplies the functional
equation; its explicit coefficients supply the mean-square
estimates. These are the theta identities and estimates used below; Appendix~\ref{app:theta-conventions}
checks the constants directly from \cite[equations~(5.6), (5.7), (5.15), (5.16)]{DR}.

\subsection{Cubic theta sums, their functional equation and mean-square estimates}

In this subsection, we define the theta sums including all cube multiples and record their
functional equation and mean-square estimates for the direct bound.
The sum \eqref{eq:completed} supplies the theta factor in
\eqref{eq:separation}; its mean-square estimate is combined with
the additive estimate by \eqref{eq:separated-cauchy-schwarz}.
The estimate \eqref{eq:unmarked} is used for the preliminary
$11/12$ bound, while the divisibility-restricted estimates described
below are used for the modified average and the subsequent bounds.

For a ray character $\Psi$ with the conventions of
Section~\ref{sec:arithmetic}, let $V\in C^\infty(0,\infty)$
satisfy, for every $C>0$ and $j\ge0$,
\[
 (x\partial_x)^jV(x)=O_{C,j}(x^C)\quad(x\to0),\qquad
 (x\partial_x)^jV(x)=O_{C,j}(x^{-C})\quad(x\to\infty).
\]
The indices $c,n$ below represent nonzero integral ideals coprime to
$S$, using their primary generators; $c$ is squarefree and $n$ is
unrestricted. Put
\begin{align*}
 D(s,\Psi)&=\sum_{c\ \sfq}\gamma_2(c)\overline{\varpi(c)}
                    \Psi(c)q_c^{-s},\\
 L_0(s,\Psi)&=\sum_n\overline{\varpi(n)}^{\,3}
                    \Psi(n)^3q_n^{-3s+1/2},\\
 T(s,\Psi)&=D(s,\Psi)L_0(s,\Psi).
\end{align*}
The series $T(s,\Psi)$ is the character-twisted Dirichlet series
built from the theta coefficients after the fixed rescaling and
phase conversion in \Src{Proposition}{5.1}; it converges
absolutely for $\Rea s>1$. Its smoothed expression is
\begin{equation}\label{eq:completed}
 \mathcal C_V(X;\Psi)=
 \sum_{c\ \sfq,n}
 \frac{\gamma_2(c)\overline{\varpi(cn^3)}\Psi(cn^3)}
      {\sqrt{q_c}\,q_n}\,
 V(q_cq_n^3/X).
\end{equation}
No condition $(c,n)=1$ is imposed. The cube factor is included
because the natural theta support consists of the indices $cn^3$;
thus the sum includes all cube multiples of each squarefree index $c$. For a primary $A$,
\[
 A=cn^3
 \quad\Longleftrightarrow\quad
 v_p(A)\equiv0\ \text{or }1\pmod3\quad\text{for every }p,
\]
with $c$ squarefree and $c,n$ uniquely determined primary
generators. The gamma factors enter separately through the functional equation.

\begin{proposition}[Cubic theta functional equation, \Src{Proposition}{5.1}]
Fix a finite family of ray characters with conductors supported on
$S$, and a common periodicity modulus $L_{\rm per}$ supported on
$S$, divisible by $18$ and their defining moduli, as specified in
\cite[Proposition~5.1, hypotheses]{QRH}.
Write $\Psi=\Psi_0\prod_{p\in\mathcal P}\chi_p^{j_p}$ on
primary elements, with $\mathcal P$ a finite set of distinct primes
outside $S$ and $j_p\in\{0,\ldots,5\}$; zero extensions are retained.
The product $T(s,\Psi)$ continues to an entire function with polynomial
growth on each fixed vertical strip, with constants depending on
the fixed arithmetic data and the strip.
For $X>0$ and the test $V$ specified above,
the precise functional-equation identity is
\begin{equation}\label{eq:theta-reflection}
 \mathcal C_V(X;\Psi)=\sum_{\mathfrak t}\zeta_{\mathfrak t}
 \sum_{0\ne\mu\in\lambda^{-4}\OO}
 \frac{d_{\mathfrak t}(\mu)\varpi(\mu)
       \vartheta_{\mathfrak t}(\lambda^4\mu)}{\sqrt{q_\mu}}
 \prod_{p\ {\mathrm{active}}} B_p(\lambda^4\mu)
 V^\sharp(q_\mu X/q_{c_{\mathfrak t}}^2).
\end{equation}
Here $\mathfrak t$ runs over the terms indexed by
$h_0\bmod L_{\rm per}$ and an active subset of $\mathcal P$
\cite[Proposition~5.1, item~(1)]{QRH}. Every prime with $j_p\ne0$
is active. The denominator is $c_{\mathfrak t}=c_F\prod_{p\ {\mathrm{active}}}p$,
where $c_F$ is the reduced denominator of $\lambda^2h_0/L_{\rm per}$,
with the normalizing unit fixed as in that item.
The function $d_{\mathfrak t}$ is one of $d_0,d_1,d_2$;
$\vartheta_{\mathfrak t}$ is an additive character modulo a fixed
modulus depending on $L_{\rm per}$; the scalar $\zeta_{\mathfrak t}$
is independent of $\mu$ and has $|\zeta_{\mathfrak t}|\le1/81$.
Their exact phases are specified in
\cite[Proposition~5.1, items~(2)--(3)]{QRH}.
In the following local factors, every character value is extended by
zero on noninvertible residue classes:
\[
 B_p(x)=\begin{cases}
 \chi_p(x)^{-j_p-2},&j_p\ne0,4,\\
 q_p^{-1/2}(-1+q_p1_{p\mid x}),&j_p=4,\\
 q_p^{-1/2}\chi_p(x)^{-2},&j_p=0\text{ and }p\text{ active}.
 \end{cases}
\]
The function $V^\sharp$ appearing in the functional equation is obtained
from $V$ by the integral transform
\[
 V^\sharp(x)=\frac1{2\pi i}\int_{(v)}\widehat V(-t)
 \prod_{\pm}\frac{\Gamma(1+t\pm1/6)}{\Gamma(1-t\pm1/6)}
 \left(\frac{(2\pi)^4x}{27}\right)^{-t}dt,\qquad v\ge0,
\]
where $\Gamma$ is the Euler gamma function and
$\widehat V(t)=\int_0^\infty V(x)x^t\,dx/x$.
The dual coefficients have support
$u\lambda^k c n^3$, where $u\in\OO^\times$ is a unit,
$k\in\mathbb Z$, $k\ge-4$, $c$ is squarefree, and $c,n$ are
primary. At such an index each of the three fixed cusp coefficient
functions $d_j$ of Section~\ref{sec:theta-definition}
has absolute value $\le27\,3^{k/6}|n|$. The formula retains the local character
zeros, including the coprimality conditions of imprimitive
presentations.
\end{proposition}
The quadratic large sieve over number fields is
\cite[Definition~1 and Theorem~1.1]{GL}.

\begin{lemma}[Cubic theta mean square, \Src{Lemma}{5.8}]
Fix the arithmetic data, bounded ranges for $M,N$ and dyadic
support comparisons. For every $\eps>0$, there exist finite
seminorm orders $A,J$ such that, for all sufficiently large $Z$
and every nonempty powerful-part dyadic range $O$, one has
\begin{equation}\label{eq:unmarked}
 \sum_{\substack{0<q_m\ll Z^M\\q_{m_{\rm pow}}\asymp Z^O}}
 |\mathcal C_V(Z^N;\nu\chi_\bullet(m))|^2
 \ll Z^{O/2+\max(M-O,\,2M-O-N)+\eps}
      \mathfrak M(V)^2.
\end{equation}
Here $m$ ranges over nonzero elements of $\OO$, including all unit
associates. The parameters $M,O\ge0$ and $N\in\mathbb R$ are norm-range exponents, $\nu$ is a fixed
finite ray character, and $\chi_\bullet(m)$ means
$n\mapsto\chi_n(m)$. The ideal
$m_{\rm pow}=\prod_{v_p(m)\ge2}p^{v_p(m)}$ is the powerful part.
The seminorm is $\mathfrak M=\mathfrak M_{A,J}$ for sufficiently
large fixed $A,J$, where \Src{Proposition}{5.1} defines
\[
 \mathfrak M_{A,J}(V)=
 \sup_{-A\le\sigma\le1/4}\int_{\mathbb R}
 (1+|t|)^J|\widehat V(\sigma+it)|\,dt.
\]
Decomposition into the fixed ray classes and decay of the integral
transform defining $V^\sharp$ justify the specialization to the Gaussian test \Src{Lemmas}{5.2--5.7}.
\end{lemma}

For the later divisibility restrictions, the corresponding
mean-square estimate is \Src{Lemma}{14.3}. Its arithmetic
ingredient \Src{Lemma}{14.2} combines the quadratic sieve
\cite[Theorem~1.1]{GL} and the cubic sieve \cite[Theorem~2]{HB}.
\begin{lemma}[Quadratic--cubic sieve, \Src{Lemma}{14.2}]
Let $K,N,B,L$ be independent positive real norm parameters:
\[
 K,N,B,L\ge1,\qquad
 q_k\ll K,\quad q_n\asymp N,\quad q_b\asymp B,\quad q_P\asymp L.
\]
Here $k,n,P$ are squarefree primary ideals and $b$ is any
primary ideal. Thus $KNBL$ is the product of the four norm
parameters, not a further variable. The indices $k,P$ avoid $S$;
$n,b$ may meet $S$ except at $\lambda$, as specified in the cited lemma.
For every $\eps>0$,
\begin{equation}\label{eq:hybrid}
\begin{split}
 &\sum_{k\ \sfq,\ q_k\ll K}
 \left|\sum_{P\ \sfq,\ q_P\asymp L}
 \frac{a(P)}{\sqrt{q_P}}1_{(P,k)=1}
 \sum_{\substack{n\ \sfq,\ q_n\asymp N\\q_b\asymp B}}
 \beta(n,b)\chi_k(nb)^3\chi_P(n)^{-2}1_{(P,b)=1}\right|^2\\
 &\hspace{15mm}\ll(KNBL)^\eps(K+NB)B\bigl[N+L+(NL)^{2/3}\bigr].
\end{split}
\end{equation}
The coefficient functions satisfy $|a(P)|,|\beta(n,b)|\le1$;
$a(P)$ is independent of $k,n,b$, and $\beta(n,b)$ of $k,P$.
\end{lemma}
The ideal $P$ need not be prime. Inclusion-exclusion removes
the moving coprimality condition before applying the sieves.
The full estimate obtained from the functional equation follows by dyadic decomposition,
the local theta coefficients, and decay of $V^\sharp$.

\subsection{The smoothed average and its additive and theta factors}
\label{sec:average-separation}

In this subsection, we define the smoothed average $I_\eta(X,Y,Z)$
in \eqref{eq:base} and decompose it into the two factors in
\eqref{eq:separation}. The Cauchy--Schwarz inequality
\eqref{eq:separated-cauchy-schwarz} then reduces its direct bound
to their mean-square estimates; the next
subsection gives its Poisson representation, which contains the
reciprocal $L$-function needed in Proposition~\ref{prop:criterion-summary}.

Choose a fixed finite ray class group $\mathcal T$ through which the
functions $\mathcal R$, $G$, and the fixed root-of-unity factors arising
from the primary-generator convention and sextic reciprocity factor
\Src{Lemma}{4.4}. Thus these factors depend only on the corresponding
ray classes; the fixed congruence conditions involved are supported at
primes above $2$ and $3$. Here $\widehat{\mathcal T}$ is its
character group, and $1_{\mathcal T}$ its identity ray class.
Membership $p\in1_{\mathcal T}$ means that the primary generator
of $p$ has trivial image in $\mathcal T$. Write
$\Theta=\langle\eta,\widehat{\mathcal T}\rangle$ for the finite
group of Hecke characters generated by these characters; all
its conductor primes are in $S$ \Src{Section}{8}.
Let $b_*$ generate the squarefree product of primes in $S$ and
choose a primitive character $\xi\bmod b_*$ of order dividing
six, nonprincipal at each local prime. For a residue character
$\psi$ modulo a primary modulus $c$ and $h\in\OO$, set
\[
 g_\psi(c,h)=\sum_{d\bmod c}\psi(d)e(hd/c),\quad
 \tau_*=q_{b_*}^{-1/2}g_\xi(b_*,1),\quad |\tau_*|=1,\quad
 \Xi(A)=\xi(A)\chi_A(b_*).
\]
Here $A\in\OO$ ranges over primary elements, so
$A\equiv1\pmod{3\OO}$.
For the primary averaging variable $v$ let
\[
 \Psi_{m,v,\eta}(A)=\eta(A)\Xi(A)^{-1}\mathcal R(A,v)\chi_A(m)
\]
on primary elements coprime to $S$, extended by zero elsewhere, and define
\[
 T_{m,v,\eta}(t)=\sum_{c\ \sfq,n}
 \gamma_2(c)\overline{\varpi(cn^3)}\Psi_{m,v,\eta}(cn^3)
 \overline{G(cn^3)}q_c^{-1/2-t}q_n^{-1-3t}.
\]
For fixed nonnegative nonzero test functions $W_0,W_1$ and
$\Phi(t)=e^{t^2}$, define the \emph{smoothed average}
$I_\eta(X,Y,Z)$ used in the original paper by
\begin{equation}\label{eq:base}
\begin{split}
 I_\eta(X,Y,Z)={}&\frac1Y\sum_{\substack{v\ \mathrm{primary}\\(v,S)=1}}
 \frac{W_1(q_v/Y)\chi_v(b_*)}{\tau_*\xi(v)\sqrt{q_{b_*}q_vX}}
 \sum_{m\in\OO\setminus\{0\}}\xi(m)q_v^{-1/2}g_{\chi_v}(v,-m)\\
 &\times W_0(q_m/(q_{b_*}q_vX))
 \frac1{2\pi i}\int_{(4)}Z^t\Phi(t)T_{m,v,\eta}(t)dt.
\end{split}
\end{equation}
In \eqref{eq:base}, $v$ runs over the primary generators of nonzero
integral ideals coprime to $S$, with no squarefree restriction, and
$m$ runs over nonzero elements of $\OO$, including all unit associates.
The factor $\xi(m)$ vanishes when $(m,b_*)\ne1$; the compact supports
of $W_0,W_1$ make both sums finite and exclude $m=0$.
These same domains are used in the factors $A_{m,\sigma,t_0}$ and
$B_{m,\sigma}$ below; $v\in\sigma$ additionally restricts $v$ to
the indicated ray class.
Here \emph{smoothed} means that the summation ranges are restricted
by smooth functions rather than by characteristic functions of intervals:
$W_1(q_v/Y)$ restricts the $v$-sum and
$W_0(q_m/(q_{b_*}q_vX))$ restricts the $m$-sum.
The function $V_\Phi$ defined below similarly smooths the theta sum.
The term \emph{average} refers to the weighted summation over $m,v$,
with the normalization $1/Y$ and the other factors in \eqref{eq:base}.
The smooth functions permit Mellin inversion and Poisson summation
with decay estimates for their transforms.
The $m,v$ sums are finite, and the theta series including all cube multiples converges
absolutely on the initial line.

Expand $\overline{G(A)}=\sum_{\theta\in\widehat{\mathcal T}}
a_\theta\theta(A)$, where $\sum|a_\theta|\le|\mathcal T|^{1/2}$.
To relate \eqref{eq:base} directly to the theta sum
\eqref{eq:completed}, define
\[
 V_\Phi(x)=\frac1{2\pi i}\int_{(4)}\Phi(t)x^{-t}\,dt
 =\frac1{2\sqrt\pi}\exp\!\left(-\frac{(\log x)^2}{4}\right),
 \qquad x>0.
\]
The coefficient definitions and Mellin inversion give
\[
 \begin{aligned}
 T_{m,v,\eta}(t)
 &=\sum_{\theta\in\widehat{\mathcal T}}a_\theta
 T(t+1/2,\Psi_{m,v,\eta}\theta),\\
 \frac1{2\pi i}\int_{(4)}Z^t\Phi(t)T_{m,v,\eta}(t)\,dt
 &=\sum_{\theta\in\widehat{\mathcal T}}a_\theta
 \mathcal C_{V_\Phi}(Z;\Psi_{m,v,\eta}\theta).
 \end{aligned}
\]
Thus \eqref{eq:base} is a weighted average over $m$ and $v$ of
finite linear combinations of the cubic theta sums
\eqref{eq:completed}, with the character depending on $m,v,\eta$.
For a class $\sigma\in\mathcal T$ and representative $v_\sigma$,
put $\nu_\sigma(A)=\eta(A)\Xi(A)^{-1}\mathcal R(A,v_\sigma)$ and
\begin{align*}
 B_{m,\sigma}(Z)&=\sum_\theta a_\theta\frac1{2\pi i}
 \int_{(4)}Z^t\Phi(t)T(t+1/2,\nu_\sigma\theta\chi_\bullet(m))dt,\\
 A_{m,\sigma,t_0}(Y)&=\frac1Y\sum_{v\in\sigma}
 \frac{W_1(q_v/Y)\chi_v(b_*)}{\tau_*\xi(v)}
 (q_v/Y)^{-1/2+it_0}q_v^{-1/2}g_{\chi_v}(v,-m).
\end{align*}
Writing $Q=q_{b_*}XY$, Mellin inversion separates the average:
\begin{equation}\label{eq:separation}
 I_\eta=\frac{Q^{-1/2}}{2\pi}\sum_{\sigma\in\mathcal T}
 \int_{\mathbb R}\widehat W_0(it_0)
 \sum_{m\ne0}\Omega(q_m/Q)\xi(m)(q_m/Q)^{-it_0}
 A_{m,\sigma,t_0}(Y)B_{m,\sigma}(Z)\,dt_0.
\end{equation}
Here $t_0\in\mathbb R$ is the Fourier integration variable,
$\widehat W_0(it_0)=\int_0^\infty W_0(r)r^{it_0}\,dr/r$,
and $\Omega$ is a fixed smooth cutoff equal to one on the
support of $W_0$ \Src{Section}{6}.
Applying the Cauchy--Schwarz inequality to \eqref{eq:separation} gives
\begin{equation}\label{eq:separated-cauchy-schwarz}
\begin{split}
 |I_\eta|\ll Q^{-1/2}\sum_{\sigma\in\mathcal T}
 \int_{\mathbb R}|\widehat W_0(it_0)|
 &\left(\sum_{0<q_m\ll Q}|A_{m,\sigma,t_0}(Y)|^2\right)^{1/2}\\
 {}\times&\left(\sum_{0<q_m\ll Q}|B_{m,\sigma}(Z)|^2\right)^{1/2}dt_0.
\end{split}
\end{equation}
The fixed bound for $\Omega$ is absorbed into the implied constant.
Thus the direct estimate reduces to the two displayed mean squares.

\subsection{Poisson summation and the scalar Euler factors}
\label{sec:average-poisson}

In this subsection, we apply Poisson summation to the smoothed average
and express the resulting terms through their scalar Euler factors.
The term $u=1$ yields the principal Mellin integral involving
$H_\eta(s)/L_F^S(s,\eta)$ in \eqref{eq:general-f}; bounding the remaining
terms gives the comparison estimate \eqref{eq:criterion-comparison}.
Together with the direct mean-square bound from Section~\ref{sec:average},
this comparison allows Proposition~\ref{prop:criterion-summary} to be
applied at $11/12$ in Section~\ref{sec:first}.
The $7/8$ argument uses the same Poisson representation after inserting
the selected prime factors of Section~\ref{sec:modified}, with the refined
counts of Section~\ref{sec:counts} controlling the remaining terms.

Let $\widetilde W_0(|y|^2)$ be the Fourier transform on
$\mathbb C\cong\mathbb R^2$ of the function $x\mapsto W_0(|x|^2)$,
with the pairing and self-dual measure of Section~\ref{sec:arithmetic}:
\[
 \widetilde W_0(|y|^2)
 =\int_{\mathbb C}W_0(|x|^2)e(-xy)\,d\mu(x).
\]
For a fixed theta index $A=cn^3$ and primary $v$ coprime to $S$,
the arithmetic coefficient of the $m$-sum in \eqref{eq:base} is
periodic modulo $vb_*A$, even when $(v,A)\ne1$.
Define its finite Fourier coefficient by
\[
 \mathcal F_{v,A}(H)
 =\sum_{\substack{d\bmod v\\H\equiv b_*Ad\pmod v}}
 \chi_v(d)\,
 g_{\xi\chi_A}\!\left(b_*A,\frac{H-b_*Ad}{v}\right),
 \qquad H\in\OO.
\]
Expanding $g_{\chi_v}(v,-m)$ gives
\[
 \sum_{r\bmod vb_*A}\xi(r)\chi_A(r)q_v^{-1/2}
 g_{\chi_v}(v,-r)e\!\left(\frac{rH}{vb_*A}\right)
 =q_v^{1/2}\mathcal F_{v,A}(H).
\]
Poisson summation on $\OO$ therefore gives the explicit identity
\begin{equation}\label{eq:average-poisson-step}
\begin{split}
 &\sum_{m\in\OO\setminus\{0\}}
 \xi(m)\chi_A(m)q_v^{-1/2}g_{\chi_v}(v,-m)
 W_0\!\left(\frac{q_m}{q_{b_*}q_vX}\right)\\
 &\qquad=\frac{Xq_v^{1/2}}{q_A}
 \sum_{H\in\OO}\mathcal F_{v,A}(H)
 \widetilde W_0\!\left(\frac{Xq_H}{q_A}\right).
\end{split}
\end{equation}
Here $W_0$ is extended by zero at $0$. The normalization follows from
$q_{vb_*A}=q_vq_{b_*}q_A$ and the dilation factor
$q_{b_*}q_vX$; see \cite[Section~7, Poisson summation]{QRH}.
Primitivity of $\xi$ at the primes in $S$ gives
$\mathcal F_{v,A}(H)=0$ unless $(H,S)=1$, and in particular removes
$H=0$. Every remaining frequency is uniquely $H=ua^6$, where
$u\in\OO\setminus\{0\}$ includes its unit factor and has prime
valuations at most five, while $a$ is a primary generator.
Mellin inversion of the two test functions, followed by the local
calculation below, converts \eqref{eq:average-poisson-step} into
\eqref{eq:Poisson}. Set
\[
 \mathcal M(z)=\int_0^\infty\widetilde W_0(r)r^{z-1}dr,\quad
 \widehat W_1(w)=\int_0^\infty W_1(r)r^{w-1}dr,
\]
\[
 \mathcal W=X^{1/2-z}Z^{s+z-1}Y^{w-1}
               \Phi(s+z-1)\mathcal M(z)\widehat W_1(w).
\]
The complete local calculation \Src{Lemma}{7.1} gives
\begin{equation}\label{eq:Poisson}
\begin{split}
 I_\eta={}&\frac1{(2\pi i)^3}\int_{(3)}\int_{(3)}\int_{(2)}
 \mathcal W
 \sum_u^{(6)}q_u^{-z}\overline{\xi(u)}
 \frac{\zeta_F^S(6z)\Lfun^S(w,\chi_\bullet(u))}
      {\Lfun^S(s,\eta\overline{\chi_\bullet(u)})}
 \mathcal H_{\eta,u}(s,w,z)\,dz\,dw\,ds.
\end{split}
\end{equation}
Here $\chi_\bullet(u):n\mapsto\chi_n(u)$ and $\sum^{(6)}$
means sixth-power-free nonzero elements coprime to $S$.

Let $\zeta_F^S(s)=L_F^S(s,1)$ be the Dedekind zeta function
with the primes of $S$ deleted. The local coefficient
$P_p=P_p(s,w,z;\eta,u)$ is the complete valuation series
computed in \Src{Lemma}{7.1}, immediately before and in its
local identity: $v_p(c)\in\{0,1\}$ and $v_p(n),v_p(v),v_p(a)\ge0$
are summed, with $v_p(u)\in\{0,\ldots,5\}$ fixed. This summary's
$s$ is the variable called $x$ in that lemma. For its scalar
Euler factors write
\begin{equation}\label{eq:local-scalar-factors}
 V_p=q_p^{-6z},\quad W_p=\chi_p(u)q_p^{-w},\quad
 D_p=\eta(p)\overline{\chi_p(u)}q_p^{-s},\qquad
 H_p=P_p\frac{(1-V_p)(1-W_p)}{1-D_p}.
\end{equation}
Then $\mathcal H_{\eta,u}=\prod_{p\notin S}H_p$.
Gauss identities, reciprocity and the inserted phases factor the
coefficient into these local terms. The six cases $v_p(u)=0,\ldots,5$
are summed before extracting the scalar quotient.
The correction is holomorphic in the two regions
\begin{align}
 \mathcal D_1:\quad &s_r\ge51/100,\quad z_r\ge17/50,\quad
 w_r\ge-1/100,\quad s_r+w_r\ge1+\eps_0,\label{eq:D1}\\
 \mathcal D_2:\quad &s_r\ge7/8,\quad z_r\ge33/200,\quad
 w_r\ge19/20,\label{eq:D2}
\end{align}
where $s_r=\Rea s$, $w_r=\Rea w$, $z_r=\Rea z$,
and $\eps_0>0$ is the fixed gap above $1$ in the condition
$s_r+w_r\ge1+\eps_0$, with
$\mathcal H_{\eta,u}\ll_\eps q_u^\eps$ uniformly in heights.

The numerator character $\chi_\bullet(u)$ in \eqref{eq:Poisson} is
principal only when $u=1$: a good prime dividing $u$
forces nontrivial ramification, and a nontrivial unit gives a
nontrivial Kummer character. At $u=1$ the poles are
$w=1,z=1/6$. Define
\begin{equation}\label{eq:H-c}
 H_\eta(s)=\mathcal H_{\eta,1}(s,1,1/6),\qquad
 c_S=\widehat W_1(1)\mathcal M(1/6)
       \frac{(\Res_{v=1}\zeta_F^S(v))^2}{6}>0.
\end{equation}
The factor $1/6$ comes from taking the residue of $\zeta_F^S(6z)$
at $z=1/6$:
\[
 \operatorname{Res}_{z=1/6}\zeta_F^S(6z)
 =\frac16\operatorname{Res}_{s=1}\zeta_F^S(s).
\] A positive Gaussian-integral
representation proves $\mathcal M(1/6)>0$.
Since $(s,1,1/6)\in\mathcal D_2$ for $\Rea s>7/8$,
\eqref{eq:D2} and \eqref{eq:H-c} give holomorphy of $H_\eta$ there.
The uniform local estimate $H_\eta(s)=1+O(P_0^{-4/5})$
\Src{Lemma}{7.1} allows $P_0$ to be chosen sufficiently large that
\begin{equation}\label{eq:H-uniform}
 \sup_{\Rea s>7/8}|H_\eta(s)-1|\le1/2.
\end{equation}
The choice of $P_0$ is uniform in the fixed target characters and
root-of-unity factors; see also
\cite[Section~11, principal correction]{QRH}.

\section{Zero detection by two character polynomials}
\label{sec:detector}

In this section, we associate the nonprincipal terms in the Poisson
representation \eqref{eq:Poisson} of Section~\ref{sec:average-poisson}
with zero classes and detecting character polynomials.
The zero classes are defined in \eqref{eq:zero-class}.
A selected zero forces the two detecting sums $M_r,S_m$, with
the same character presentation and norm-twist height, to satisfy
the simultaneous large-value bounds \eqref{eq:large-values}. The zero classes and the large value of $M_r$
are first used with the sextic sieve in Section~\ref{sec:first}
to obtain the preliminary $11/12$ bound. For the sharper $7/8$
bound, Sections~\ref{sec:moments}--\ref{sec:counts} combine
both large values with the two moment estimates to obtain a
stronger count of elements in each class.

For the zero-detection method, see
\cite[Appendix~C, Zero detecting setup]{MP}.
The simultaneous Hecke-character version used here is
\Src{Proposition}{8.3}.

For nonzero sixth-power-free elements $u\in\OO$ coprime to $S$,
including all unit associates, with $q_u\asymp U=Z^d$, consider
\[
 \mathcal X_u=\{\nu(n)\chi_n(u)^{\pm1}:\nu\in\Theta\}.
\]
For a presentation $\psi$, let $\psi^*$ denote its primitive
inducing character as in \eqref{eq:primitive-induction}: it has
the minimal defining modulus and agrees with $\psi$ on ideals
coprime to the defining modulus of $\psi$.
Let $L(s,\psi)$ be the Hecke series of $\psi$ with the
original zero extension; thus $L(s,\psi^*)=L_F(s,\psi^*)$.
Primitive conductors have norm $\ll_\mathcal A U$; any
Euler factors omitted by the original presentations occur only at prime
ideals in $S$ or dividing $(u)$.
Choose $0<\eb<10^{-3}$ and $T_1=Z^\tau\le U^{1/100}$.
Set $I_{\eb}=\lceil1/\eb\rceil+2$. With $a_0=51/100$ and
$j\in\{1,\ldots,I_{\eb}\}$, let
\[
 M_j(u)=\max\bigl(\{a_0\}\cup
 \{\Rea\rho:\psi\in\mathcal X_u,\ L(\rho,\psi^*)=0,\
             \Rea\rho\ge a_0,\ |\Ima\rho|\le3jT_1\}\bigr).
\]
A pigeonhole argument gives $1\le i<I_{\eb}$ and
$a\in(a_0+\eb\mathbb Z_{\ge0})\cap[a_0,1]$ such that
\begin{equation}\label{eq:zero-class}
 a\le M_i(u)<a+\eb,\qquad M_{i+1}(u)<a+2\eb
 \qquad\Src{Lemma}{8.1}.
\end{equation}
Hence the associated functions have a zero-free rectangle to
the right of $a+2\eb$ at the needed height. Concretely,
\[
 \mathcal U_{a,i}=
 \{s\in\mathbb C:a+2\eb\le\Rea s\le2,\ 
                 |\Ima s|\le3(i+1)T_1\}
\]
contains no zero of $L_F(s,\psi^*)$ for $\psi\in\mathcal X_u$,
by \eqref{eq:zero-class} and the definition of $M_{i+1}$.
If $a>a_0$, a selected
zero $\rho=\sigma+i\gamma$ satisfies
$a\le\sigma<a+\eb$, $|\gamma|\le3iT_1$.
Since these are actual family zeros, $a\le\beta_*$.

Let $\mathcal B$ denote a subset of sixth-power-free elements
$q_u\asymp U$ in one fixed $(i,a)$ class; later it may also fix
prime-sum size ranges. All class counts are finite for fixed
parameters \Src{Lemma}{8.1}.
For one presentation $\psi\in\mathcal X_u$, real length exponents
$r,m\ge0$, and test functions $W_M,W_S$, define
\[
 M_r=U^{-r/2}\sum_n\mu(n)\psi(n)W_M(q_n/U^r),\qquad
 S_m=U^{-m/2}\sum_n\psi(n)W_S(q_n/U^m).
\]
Here $n$ runs over nonzero integral ideals, represented by their
primary generators where a residue symbol is used; the character
values retain their stated zero extensions.
The test functions may include factors $q_n^{it}$ with $t\in\mathbb R$;
these are the norm twists specified in Section~\ref{sec:arithmetic}.
Polynomial bounds on fixed vertical strips, logarithmic bounds
in zero-free disks and deleted
Euler estimates \Src{Lemmas}{4.8--4.10} yield
\begin{equation}\label{eq:pointwise}
 |M_r|^2\ll U^{\delta r+\eps},\qquad
 |S_m|^2\ll U^{\delta\min(m,1-m)+\eps},
 \qquad\delta=2a-1,
\end{equation}
when $(1+T_1)^A\le U^{\eps/10}$ \Src{Lemma}{8.2}.
Here $T_1$ is the height cutoff used in \eqref{eq:zero-class}, and
$A\ge0$ is a fixed finite exponent for its polynomial dependence,
as in Section~\ref{sec:arithmetic}.
The second estimate takes the better of the direct bound and the bound
obtained by applying the Hecke functional equation. For the classical functional equation
in this setting, see \cite[equation~(1.1)]{GZ}; the zero-class bounds
with deleted Euler factors are those of \Src{Lemma}{8.2}.

\begin{proposition}[Simultaneous large values, \Src{Proposition}{8.3}]
Fix $\eps>0$, a class $(i,a)$ from \eqref{eq:zero-class} with
$a>a_0$, and a sixth-power-free $u$ of norm $\asymp U$ in that
class. Assume $T_1\le U^{1/100}$ and the height bound
$(1+T_1)^A\le U^{\eps/10}$ from \Src{Lemma}{8.2}, after
choosing its preliminary losses sufficiently small.
For every $t\in[1,3/2]$ and all $U\ge U_0(\mathcal A,\eps)$,
there exist a character presentation $\psi\in\mathcal X_u$, a
zero $\rho$ of $L_F(s,\psi^*)$, where $\psi^*$ is the primitive
character inducing $\psi$, and two sums $M_r,S_m$ of the form
defined above, both using this same presentation (including its
defining modulus and zero extension) and the same norm-twist
height, such that
\begin{align}
 r&\le t+O(1/\log U),&
 r+m&\ge t-O(1/\log U),\label{eq:lengths}\\
 t-\tfrac12-O(\eps)&\le r,&
 0\le m&\le\tfrac12+O(\eps),\label{eq:short-plain}\\
 |M_r|^2&\gg U^{\delta r-\eps},&
 |S_m|^2&\gg U^{\delta m-\eps}
.\label{eq:large-values}
\end{align}
\end{proposition}
\paragraph{\it Idea of zero detection.}
To obtain them, put
\[
 C_\psi(s)=\sum_l\mu(l)\psi(l)
 V_{\le}(q_l/U^t)q_l^{-s},\qquad
 V_{\le}=1\text{ on }[0,1],\quad
 V_{\le}=0\text{ on }[2,\infty),
\]
where $V_{\le}$ is a fixed smooth cutoff. Here $l$ ranges over
nonzero integral ideals and $t$ is the real truncation exponent.
The polynomial $C_\psi$ is a truncated version of the M\"obius series
for $1/L(s,\psi)$ in $\Rea s>1$. Insert it into
\[
 \frac1{2\pi i}\int_{(2)}U^{20z}\Gamma(z)
 L(\rho+z,\psi)C_\psi(\rho+z)\,dz.
\]
Moving the contour to $\Rea z=-1/4$ gives a small integral; the residue at $z=0$
vanishes because $L(\rho,\psi)=0$. Expansion of
$LC_\psi$ cancels the ideal coefficients up to length $U^t$
except the term of the unit ideal. Splitting the remaining ideal norms
into dyadic intervals therefore gives a product of detecting polynomials with
$|M_rS_m|^2\gg U^{\delta(r+m)-\eps}$.
Comparison with \eqref{eq:pointwise} forces
$m\le1/2+O(\eps)$ and then both individual lower bounds.
At $a=a_0$, the class definition \eqref{eq:zero-class} need not
provide a selected zero, so this lowest class is bounded by counting
all elements in its norm range.

\section{The preliminary \texorpdfstring{$11/12$}{11/12} bound}
\label{sec:first}

In this section, we first apply Proposition~\ref{prop:criterion-summary}
at $11/12$, using the unmodified average \eqref{eq:base} and the
preliminary character count \eqref{eq:first-count}.
This bound restricts the possible zero classes and supplies the starting
point for the refined $7/8$ argument in the remaining sections of Part I.

Take $X=Y=Z^{1/2}$ and $Q=q_{b_*}Z$.
For the number-field additive large sieve, see
\cite{Huxley}; its exact planar form and application here are
\Src{Lemmas}{6.1--6.2}. It gives
\[
 \sum_{q_m\ll Q}|A_{m,\sigma,t_0}(Y)|^2\ll(Q+Y^2)/Y\ll Q/Y
.
\]
Equation~\eqref{eq:unmarked}, at $M=N=1$, gives
$\sum|B_{m,\sigma}(Z)|^2\ll Z^{1+\eps}$.
Thus \eqref{eq:separation} implies
\begin{equation}\label{eq:first-direct}
 |I_\eta(Z^{1/2},Z^{1/2},Z)|
 \ll Q^{-1/2}(Q/Y)^{1/2}Z^{1/2+\eps}
 =Z^{1/4+\eps}\qquad\Src{Proposition}{6.3}.
\end{equation}

The underlying estimates are the cubic sieve \cite[Theorem~2]{HB},
the quadratic sieve \cite[Theorem~1.1]{GL}, and the higher-order
norm recursion \cite[Theorem~1.3 and Section~3]{BGL}.
An Eisenstein-integer sextic formulation is also recorded in
\cite[Lemma~2.9]{GZSextic}; \Src{Lemma}{9.1} gives the
primary-generator convention used here.
\begin{lemma}[Sextic sieve, \Src{Lemma}{9.1}]
Let $K,N_{\rm s}\ge1$ bound the norms of squarefree primary elements
$k,n$ coprime to $S$, and let $c_n$ be finitely supported and
independent of $k$. For every $\eps>0$,
\[
 \sum_{k\ \sfq,\ q_k\le K}
 \left|\sum_{n\ \sfq,\ q_n\le N_{\rm s}}c_n\chi_n(k)^{\pm1}\right|^2
 \ll(KN_{\rm s})^\eps\bigl[K+N_{\rm s}+(KN_{\rm s})^{2/3}\bigr]\sum_n|c_n|^2.
\]
\end{lemma}

\begin{proposition}[Preliminary character count, \Src{Proposition}{9.2}]
Let $\mathcal B$ be a zero class of sixth-power-free elements
$q_u\asymp U$, defined by \eqref{eq:zero-class} in
Section~\ref{sec:detector}, with the detector
height assumptions and $\delta=2a-1$.
For every $\eps>0$ and sufficiently large $U$,
\begin{equation}\label{eq:first-count}
 \#\mathcal B\ll U^{R_{\rm I}(\delta)+\eps}(1+T_1)^A,\quad
 R_{\rm I}(\delta)=\min\{1,\max(1-\delta/2,4/3-\delta)\}.
\end{equation}
Here $T_1$ is the zero-detection height cutoff of
Section~\ref{sec:detector}; $A\ge0$ is a fixed finite exponent
controlling the height dependence, as in Section~\ref{sec:arithmetic}.
\end{proposition}

\begin{proposition}[Preliminary zero-free bound, \Src{Theorem}{3.1}]
\label{prop:preliminary-zero-free}
The family supremum in \eqref{eq:beta} satisfies
\begin{equation}\label{eq:first-conclusion}
 \beta_*\le11/12.
\end{equation}
\end{proposition}
\paragraph{\it The preliminary exponent comparison.}
Assume $\beta_*>11/12$ and put $\Delta_1=\beta_*-11/12>0$.
The principal residue in \eqref{eq:Poisson} has exponent
$X^{1/3}Z^{s-5/6}=Z^{s-2/3}$. Let
\[
 C_{\rm I}(s)=s-2/3,\qquad
 J_{{\rm I},\eta}(Z)=I_\eta(Z^{1/2},Z^{1/2},Z)/c_S,
\]
with $f_{{\rm I},\eta}$ defined by \eqref{eq:general-f}.
The preliminary character count \eqref{eq:first-count}, the contour estimates
\Src{Lemmas}{10.2--10.6}, and the height choice in
\Src{Proposition}{11.2} give
\[
 |J_{{\rm I},\eta}-f_{{\rm I},\eta}|
 \ll_\eta Z^{C_{\rm I}(\beta_*)-1/4800}.
\]
Apply Proposition~\ref{prop:criterion-summary} with
\[
 \sigma_0=11/12,\qquad c=-2/3,\qquad
 \omega_0=\Delta_1/2,\qquad \sigma=1/4800.
\]
The correction hypotheses follow from \eqref{eq:D2}, \eqref{eq:H-c}
and \eqref{eq:H-uniform}. The displayed comparison bound and
\eqref{eq:first-direct}, with its loss chosen below $\Delta_1/2$,
verify the two estimates, since $C_{\rm I}(11/12)=1/4$.
Thus the preliminary argument reaches the contradiction through the
same continuation criterion used for the $7/8$ bound.

\section{The modified average and its direct estimate}
\label{sec:modified}

In this section, we modify the average used for the preliminary
$11/12$ bound in Section~\ref{sec:first}. There the lengths are
$X=Y=Z^{1/2}$ and the direct estimate is
$Z^{1/4+\eps}$, as in \eqref{eq:first-direct}. Here we use unequal
lengths and apply the selected-prime subtractions in
\eqref{eq:modified} to obtain the direct estimate
$Z^{3/16+\eps}$ in \eqref{eq:direct}.
The corresponding principal exponent changes from
$C_{\rm I}(s)=s-2/3$ to $C(s)=s-11/16$, so that
$C_{\rm I}(11/12)=1/4$ and $C(7/8)=3/16$.
The next section computes this principal term and the effect of the
subtractions on the nonprincipal terms.

Parts II--IV retain this modified average and vary its parameters.
Theorem~\ref{thm:iv-system-optimality} proves that further variation
within the explicitly specified system of estimates cannot give a
bound below $B_{\mathrm{new}}$. A different modification of the average or a
stronger estimate for its terms lies outside that optimality statement.

Assume $\beta_*>7/8$ and use \eqref{eq:two-stages}.
Choose the averaging lengths and their exponents as follows:
\begin{equation}\label{eq:geometry}
 \begin{gathered}
 b=1/8,\quad h=13/16,\quad\ell=1/6,\quad
 l_x=17/48,\quad l_y=23/48,\\
 X=Z^{l_x},\quad Y=Z^{l_y},\quad
 l_x+l_y=5/6,\quad b=l_y-l_x=1/8,\quad h=1-l_x+\ell.
 \end{gathered}
\end{equation}
Choose a positive integer $K$ and $K$ disjoint prime intervals with smooth nonnegative
weights $W_i$, norms $q_p\asymp P_i=Z^{\ell_i}$,
and $\sum_i\ell_i=\ell$. The allowed set $\mathcal P_i$
also imposes $p\notin S$ and $p\in1_{\mathcal T}$.
The underlying intervals are disjoint before these restrictions.

For a squarefree ideal $D$ outside $S$, let $I_{\eta;D}$ denote \eqref{eq:base} with $1_{D\mid cn^3}$
inserted. Write $J^c=\{1,\ldots,K\}\setminus J$ and
$p_J=\prod_{i\in J}p_i$ (empty product $1$). All tuple sums
have $p_i\in\mathcal P_i$. Define
\begin{equation}\label{eq:modified}
\begin{split}
 I_{\eta,\rm mod}(Z)=\sum_{(p_i)}\prod_iW_i(q_{p_i}/P_i)
 \sum_{J\subseteq\{1,\ldots,K\}}
 (-1)^{|J|}q_{p_J}^{-3/2}\overline{\eta(p_{J^c})}\\
 {}\times I_{\eta;p_{J^c}}
       (X/q_{p_J},Y/q_{p_J},Zq_{p_{J^c}}).
\end{split}
\end{equation}
The one-prime operation is
$\overline{\eta(p)}I_{\eta;p}(X,Y,Zq_p)
-q_p^{-3/2}I_\eta(X/q_p,Y/q_p,Z)$.
Its effect on the local Euler factors is computed in
\eqref{eq:selected-local}: the subtraction cancels an unwanted
scalar Euler contribution and produces sextic-character prime sums,
together with the local errors bounded in \eqref{eq:joint-error}.
These prime sums enter the estimates for the nonprincipal terms,
whereas their principal contribution is nonzero by
\eqref{eq:prime-scalar} and gives the normalization
\eqref{eq:normalization}. These substitutions change the averaging lengths $X,Y$, while
each selected prime $p_i$ remains in its original set $\mathcal P_i$,
with norm $q_{p_i}\asymp P_i=Z^{\ell_i}$ and test function
$W_i(q_{p_i}/P_i)$, using the original value of $Z$.

\begin{proposition}[Modified direct bound, \Src{Proposition}{15.3}]
\label{prop:modified-direct}
For the fixed arithmetic data, prime intervals and smooth weights
specified above, with the parameter choices \eqref{eq:geometry}, one has
\begin{equation}\label{eq:direct}
 |I_{\eta,\rm mod}(Z)|\ll_{\eta,\eps} Z^{3/16+\eps}.
\end{equation}
\end{proposition}
For a fixed $J$ put
\[
 d_J=\sum_{i\in J}\ell_i,\quad
 X'=X/q_{p_J},\quad Y'=Y/q_{p_J},\quad
 M'=5/6-2d_J,\quad Q=q_{b_*}X'Y'.
\]
Let $B^J_{m,\sigma}$ be the theta factor including all cube multiples in the separated
expression, with the primes $i\notin J$ already summed and
their divisibility restrictions imposed. The theta
functional equation with the divisibility restrictions \Src{Corollary}{14.1} and the resulting
mean-square bound give the following estimate.
\begin{lemma}[Modified theta mean square, \Src{Lemma}{15.1}]
For every $\eps>0$, fixed subset $J$ and ray class $\sigma$,
uniformly in the fixed tuple $(p_i)_{i\in J}$, with $J$ and the
prime sets as in \eqref{eq:modified},
\begin{equation}\label{eq:BJ-bound}
 \sum_{q_m\ll Q}|B^J_{m,\sigma}|^2
 \ll Z^{M'+((d_J-1/6)/4)_++\eps}
 =Z^{M'+\eps}.
\end{equation}
\end{lemma}
The extra term vanishes because $0\le d_J\le\ell=1/6$.
The remaining prime sums are included in $B^J_{m,\sigma}$.

\begin{proposition}[Additive mean square, \Src{Proposition}{15.2}]
\label{prop:additive-mean-square}
Fix $S,\mathcal T,b_*,\xi,W_1$ and a ray class $\sigma$ in the
finite ray group $\mathcal T$ of
Section~\ref{sec:average-separation}.
Let $Q,Y'\ge1$ range over fixed polynomial ranges in $Z$,
and suppose $P_a=Y'^2/Q\ge1$. For every $\eps>0$ and
real $t_0$, with the coefficient $A$ defined in Section~\ref{sec:average},
\begin{equation}\label{eq:Gram}
 \sum_{q_m\ll Q}|A_{m,\sigma,t_0}(Y')|^2
 \ll(1+|t_0|)^{J_{\rm Gr}}\frac Q{Y'}
       \left(1+P_a^{1/6}+\frac{P_a^2}{Y'}\right)Z^\eps
.
\end{equation}
Here $J_{\rm Gr}$ is a fixed seminorm order, independent of the
subset $J$, varying norm parameters $Q,Y'$ and $t_0$.
\end{proposition}
\paragraph{\it Idea of the additive estimate.}
Expansion of the square and Poisson summation reduce
\eqref{eq:Gram} to complete finite correlations. The common-factor
identities and local prime-power formulas are
\Src{Lemmas}{13.2--13.4}; the final mean-square estimate is
\Src{Proposition}{15.2}. Both error terms $P_a^{1/6}$ and $P_a^2/Y'$
in \eqref{eq:Gram} are retained. Their comparison for the chosen parameter values is the
calculation needed below.

\paragraph{\it The direct exponent calculation.}
With the fixed parameter choices \eqref{eq:geometry},
\[
 P_a=q_{b_*}^{-1}Z^{1/8},\qquad
 l_y-d_J-11b/6\ge1/12,\qquad P_a^2/Y'\ll P_a^{1/6}.
\]
Thus $P_a\ge1$ once $Z^{1/8}\ge q_{b_*}$, which holds
for all sufficiently large $Z$ with the arithmetic data fixed.
Equations~\eqref{eq:BJ-bound}--\eqref{eq:Gram} and
the Cauchy--Schwarz inequality \eqref{eq:separated-cauchy-schwarz},
applied after fixing $J$ and $(p_i)_{i\in J}$ and summing the
complementary prime variables in \eqref{eq:modified}, therefore bound
each separated term by
$(X')^{1/2}P_a^{1/12}Z^\eps$.
The tuples $(p_i)_{i\in J}$ have total weight
\[
 Z^{d_J+\eps}O(Z^{-3d_J/2}),\qquad
 (X')^{1/2}\asymp X^{1/2}Z^{-d_J/2}.
\]
Consequently the exponent in \eqref{eq:direct} is
$l_x/2+b/12=3/16$, since
$d_J-3d_J/2-d_J/2=-d_J\le0$.
The asymmetry in \eqref{eq:geometry} was chosen to obtain this
exponent while keeping the additive error under control.

\section{Local cancellation and principal residues}
\label{sec:principal}

In this section, we compute the effect of the modification in Section~\ref{sec:modified}
on the local Euler factors \eqref{eq:selected-local} and the principal
coefficient obtained from \eqref{eq:prime-scalar}.
The normalized average $J_\eta(Z)$ is defined in
\eqref{eq:normalization}. Its principal contribution is the integral
$f_\eta(Z)$ in \eqref{eq:f}, whose numerator contains the correction
$H_\eta(s)$ defined in \eqref{eq:H-c}; the estimate
\eqref{eq:H-uniform} gives $|H_\eta(s)|\ge1/2$ for
$\Rea s>7/8$. We also prove the joint local error estimate
\eqref{eq:joint-error}, used in Section~\ref{sec:exponents}.
The same local cancellation and normalization are retained in
Parts II--IV, with the extended domain proved in
Section~\ref{sec:extension-euler}.

\subsection{Selected-prime Euler factors and joint local errors}

In this subsection, we compute the local cancellation caused by the
selected prime factors: \eqref{eq:selected-local} gives the local
factor and \eqref{eq:modified-H} defines the resulting holomorphic
correction. The joint local error is bounded in \eqref{eq:joint-error}.
These bounds ensure that the correction remains holomorphic on the
contours used in Section~\ref{sec:exponents} and that the modification
does not introduce an error as large as the principal term.

Let $P_p$ be the local valuation series introduced in
Section~\ref{sec:average-poisson} immediately before
\eqref{eq:local-scalar-factors}, and let $P_p^*$ consist of its terms
satisfying $v_p(c)+3v_p(n)>0$, equivalently $p\mid cn^3$.
The selected-prime operation
replaces $P_p$ by
\[
 q_p^{z-1}\bigl[\overline{\eta(p)}q_p^sP_p^*-q_p^{-w}P_p\bigr].
\]
After extracting the same scalar Euler factors as in
\eqref{eq:Poisson}, define
\begin{equation}\label{eq:selected-local}
 G_p=\bigl[\overline{\eta(p)}q_p^sP_p^*-q_p^{-w}P_p\bigr]
          \frac{(1-V_p)(1-W_p)}{1-D_p}.
\end{equation}
The local identity of \Src{Lemma}{7.1} shows that $G_p$
continues holomorphically in $\mathcal D_1,\mathcal D_2$.
It does not require division by $P_p$. The correction is
\begin{equation}\label{eq:modified-H}
 \mathfrak H_{\eta,u,Z}=
 \sum_{(p_i)}\prod_i\bigl[W_i(q_{p_i}/P_i)q_{p_i}^{z-1}G_{p_i}\bigr]
           \prod_{\substack{p\notin S\\p\notin\{p_i\}}}H_p.
\end{equation}
Replacing $\mathcal H$ by $\mathfrak H$ in
\eqref{eq:Poisson} gives the exact modified average.

On the contours used for the nonprincipal contributions,
$s_r=a+16\eb,w_r=1-a-6\eb,z_r=17/50$, \Src{Section}{16} decomposes
\[
 \mathfrak H=\mathcal H\prod_i(\mathcal Q_i+\mathcal E_i),\qquad
 \mathcal Q_i=-\sum_{p\in\mathcal P_i}
       W_i(q_p/P_i)q_p^{z-1}\overline{\chi_p(u)}.
\]
For each selected prime set $\mathcal P_i$, define its exact sum
$\mathcal B_i$ and error sum $\mathcal E_i$ on these contours by
\[
 \mathcal B_i=\sum_{p\in\mathcal P_i}
 W_i(q_p/P_i)q_p^{z-1}\frac{G_p}{H_p},\qquad
 \mathcal E_i=\mathcal B_i-\mathcal Q_i.
\]
\begin{proposition}[Joint local error, \Src{Proposition}{16.1}]
Let $a_0\le a\le1$, $0<\eb\le10^{-3}$, $q_u\asymp U$,
with $u$ sixth-power-free and $(u,S)=1$. Put $\delta=2a-1$.
Suppose the primitive
character $\psi_u^*$ inducing $\chi_\bullet(u)$ is nonprincipal.
On the displayed contours restrict to $|\Im w|\le T_1$,
and assume $|L_F(1-w,\overline{\psi_u^*})|\ll U^\eps$
uniformly on the points under consideration.
For sufficiently large prime-norm cutoff $P_0$, as defined in
Section~\ref{sec:arithmetic}, the local $H_p$ are nonzero on these
contours. Moreover, for every subset $I\subseteq\{1,\ldots,K\}$,
one has
\begin{equation}\label{eq:joint-error}
 \left|\Lfun^S(w,\chi_\bullet(u))\prod_{i\in I}\mathcal E_i\right|
 \ll U^{\delta/2+O(\eb)+\eps}(3+T_1)^A
             \prod_{i\in I}P_i^{17/50-1/2+O(\eb)+\eps}.
\end{equation}
Here $A\ge0$ is a fixed finite exponent controlling the dependence
on the height cutoff $T_1$, as in Section~\ref{sec:arithmetic}.
Recall that $P_i=Z^{\ell_i}$ is the norm parameter of the
selected prime set $\mathcal P_i$ in Section~\ref{sec:modified},
and $\eb$ is the zero-class width in \eqref{eq:zero-class}.
\end{proposition}
Let $\mathfrak f_u$ be the conductor ideal of $\psi_u^*$, and
$I_0$ any set of distinct selected primes dividing $u$. With
$j_p=v_p(u)\in\{1,\ldots,5\}$,
\[
 q_{\mathfrak f_u}\ll_S q_{\rad(u)}
       \le q_u\prod_{p\in I_0}q_p^{-(j_p-1)}.
\]
The displayed conductor bound saves the product
$\prod_{p\in I_0}q_p^{j_p-1}$ relative to $q_u$.
Applying this saving once in the functional-equation estimate for
$L_F(w,\psi_u^*)$ controls the local errors at all selected ramified
primes simultaneously, as required in \eqref{eq:joint-error}. Absolute bounds on the other contours use
\Src{Lemma}{16.2}.

\subsection{The principal coefficient and the normalized average}

In this subsection, we evaluate the principal contribution using the
prime-sum asymptotic \eqref{eq:prime-scalar} and use \eqref{eq:normalization} to normalize
the modified average for Proposition~\ref{prop:criterion-summary}.
The normalized average is $J_\eta(Z)$, and its principal integral
is $f_\eta(Z)$ in \eqref{eq:f}, with exponent $C(s)$ in \eqref{eq:C}.

At $u=1$ in $\mathcal D_2$,
$G_p/H_p=-1+O(q_p^{-7/8})$ uniformly in heights.
\begin{lemma}[Prime asymptotic in a fixed ray class, \Src{Lemma}{13.1}]
Fix the ray quotient $\mathcal T$ of
Section~\ref{sec:average-separation}, the excluded set $S$ of
Section~\ref{sec:arithmetic}, and a
nonnegative nonzero test function $W_i$.
As $P_i\to\infty$, with $\mathcal P_i$ restricted to the identity
class of $\mathcal T$ and primes outside $S$, one has
\begin{equation}\label{eq:prime-scalar}
 S_i(Z):=\sum_{p\in\mathcal P_i}W_i(q_p/P_i)q_p^{-5/6}
 \sim\frac{P_i^{1/6}}{|\mathcal T|\log P_i}
                    \int_0^\infty W_i(y)y^{-5/6}dy>0.
\end{equation}
\end{lemma}
Fixed-ray prime counting follows by applying Chebotarev to
the fixed ray-class field; see \cite[Theorem~1.1]{TZ}.
Smooth partial summation gives \eqref{eq:prime-scalar}.

For large $Z$ set
\begin{equation}\label{eq:normalization}
 A_{\mathcal T}(Z)=(-1)^KZ^{-\ell/6}\prod_iS_i(Z),\qquad
 J_\eta(Z)=\frac{I_{\eta,\rm mod}(Z)}{c_SA_{\mathcal T}(Z)}.
\end{equation}
Then $|A_{\mathcal T}|\asymp(\log Z)^{-K}$ and
$|A_{\mathcal T}|^{-1}\ll_\eps Z^\eps$, so
\begin{equation}\label{eq:normalized-direct}
 |J_\eta(Z)|\ll_{\eta,\eps}Z^{3/16+\eps}.
\end{equation}

At $w=1,z=1/6$ the correction equals
\[
 H_\eta(s)Z^{\ell/6}A_{\mathcal T}(Z)
       \bigl[1+O(Z^{-\kappa_P})\bigr],\qquad
 0<\kappa_P<(7/8)\min_i\ell_i.
\]
The residue exponent is
\begin{equation}\label{eq:C}
 C(s)=s+l_x/3-5/6+\ell/6=s-11/16,\qquad C(7/8)=3/16,
\end{equation}
and the resulting integral is
\begin{equation}\label{eq:f}
 f_\eta(Z)=\frac1{2\pi i}\int_{(2)}
       Z^{s-11/16}e^{(s-5/6)^2}\frac{H_\eta(s)}{\Lfun^S(s,\eta)}ds.
\end{equation}
Crossing $w=1$ on the way to $19/20$, then $z=1/6$ on
the way to $33/200$, gives principal remainder exponents
\[
 C(\beta_*)+(1+h)\eb-m_w+\eps,\quad
 C(\beta_*)+\eb-m_z+\eps,\quad C(\beta_*)+\eb-\kappa_P+\eps,
\]
where
\[
 m_w=23/960,\qquad m_z=13/9600
 \qquad\Src{Lemma}{10.5}.
\]
Throughout these contour shifts, the denominator $L_F^S(s,\eta)$
is evaluated on $\Rea s=\beta_*+\eb$, to the right of all family zeros.

\section{Moment estimates for the two detecting polynomials}
\label{sec:moments}

In this section, we record the two moment estimates for the detecting polynomials of
Section~\ref{sec:detector}, one with M\"obius coefficients and one without.
The length restrictions in the hypotheses of the moment estimates
\eqref{eq:inverse-moment} and \eqref{eq:plain-moment} are combined in
Section~\ref{sec:counts} to bound each zero and prime-sum size class
for the $7/8$ argument. Parts II and III reuse these moments at
$\kappa=3/4$; Part IV extends the unweighted fourth moment to
smaller $\kappa$ while retaining the M\"obius-weighted estimate.

The lengths in this section are expressed by their exponents in
base $Z$, as in \eqref{eq:length-exponent-convention}: length $r$
means a norm parameter $Z^r$. Products of auxiliary prime polynomials use
disjoint prime intervals and one common character presentation.
A mean square is a sum $\sum_u|F_u|^2$; the fourth moment
arises from $\sum_u|S_uS_u|^2=\sum_u|S_u|^4$.

\subsection{M\"obius-weighted second moment}

In this subsection, we state the second-moment estimate for the
M\"obius-weighted detecting polynomial in \eqref{eq:inverse-moment},
including its admissible length ranges. Its version with auxiliary
prime factors is \eqref{eq:amplification}.
Together with the large-value bound for $M_r$ in \eqref{eq:large-values},
this supplies one of the two counting estimates in
\eqref{eq:two-counts}; the unweighted moment supplies the other.

Fix a ray character $\nu$ and a sign $\varepsilon_\chi\in\{-1,1\}$,
and put $\psi_u(n)=\nu(n)\chi_n(u)^{\varepsilon_\chi}$.
Here $\chi_n(u)$ is the sextic residue character defined in
Section~\ref{sec:arithmetic}, with the primary generator $n$
as modulus and $u$ as argument.
The estimates below hold for either choice, with the same sign
retained for every $n$, $u$ and prime factor in the polynomials.
Character powers retain the zero-extension convention of
\eqref{eq:zeros}, including for the negative sign.
Here $n$ runs over nonzero integral ideals, represented by
their primary generators, with the fixed coprimality restrictions
and zero extensions retained. The factor $\mu(n)$ makes
nonsquarefree ideals contribute zero, and $W(q_n/Z^r)$ restricts
the norm to $q_n\asymp Z^r$.
Let $I$ be a fixed finite index set. For each $i\in I$, let
$\mathcal P_i$ be a set of prime ideals outside $S$, represented
by primary generators, with norms $q_p\asymp Z^{z_i}$.
The sets $\mathcal P_i$ are pairwise disjoint and independent of
$u$; a product over the empty index set equals $1$. Define
\[
 M_u(Z^r)=Z^{-r/2}\sum_n\mu(n)\psi_u(n)W(q_n/Z^r),\qquad
 Q_u=\prod_{i\in I} Z^{-z_i/2}\sum_{p\in\mathcal P_i}
                  a_i(p)\psi_u(p)W_i(q_p/Z^{z_i}).
\]
\begin{lemma}[M\"obius-weighted second moment, \Src{Lemma}{17.1}]
\label{lem:mobius-moment}
Let $r,m_0,z_i\ge0$ be norm exponents: the ideal variable $n$
has norm $\asymp Z^r$, $u\in\OO\setminus\{0\}$ ranges over
$0<q_u\ll Z^{m_0}$, including all unit associates, and the
$i$th prime set has norms $\asymp Z^{z_i}$. The coefficients
$a_i(p)$ are bounded and independent of $u$.
Fix bounded ranges for these lengths and the number of prime
factors. Writing $z_{\rm p}=\sum_{i\in I} z_i$, assume for a fixed $c_m>0$ that
\[
 r+2z_{\rm p}\le m_0-c_m,\qquad
 2r+8z_{\rm p}\le3m_0-c_m.
\]
Then for every $\eps>0$ and all sufficiently large $Z$,
\begin{equation}\label{eq:inverse-moment}
 \sum_{0<q_u\ll Z^{m_0}}|M_u(Z^r)Q_u|^2\ll Z^{m_0+\eps}
.
\end{equation}
The implied constant may depend on $c_m$, the bounded ranges
and finitely many smooth seminorms, but is uniform in prime supports
and bounded coefficients. There is no restriction on individual
prime lengths in this lemma.
\end{lemma}

\paragraph{\it Structure of the inverse-moment induction.}
For the recursive Poisson large-sieve method, compare
\cite{HBQuadratic}, \cite[Section~4]{HB}, and
\cite[Sections~4--7]{GL}, as acknowledged in
\cite[Section~17, canonical estimate]{QRH}.
The precise moment estimates with the displayed auxiliary restrictions
are those of \cite[Sections~17--18]{QRH}.
One initial Poisson transform reduces the assertion to
\Src{Lemma}{17.2}:
\begin{equation}\label{eq:canonical}
\begin{split}
 &Z^{-V}\sum_{f\ \sfq,\ q_f\asymp Z^V}w(f)
 \sum_{0<q_k\ll Z^M}
 \left|Z^{-N/2}\sum_{n\ \sfq}
 \overline{\varpi(n)}\gamma_2(n)\nu(n)1_{(n,\mathfrak r)=1}
 \chi_n(kf^4)c_{\mathrm{sel}}(n)W(q_n/Z^N)\right|^2\\
 &\hspace{15mm}\ll Z^{F_0+\eps},\qquad F_0=N+V,
\end{split}
\end{equation}
under
\[
 q_{\mathfrak r}\le Z^{F_0-M-z_{\rm cap}-c_*},\qquad
 3M+6z_{\rm cap}+c_*\le4F_0,\qquad c_*>0.
\]
The parameters $M,N,V,z_{\rm cap}\ge0$ in this lemma are local norm
exponents: $f$ ranges over squarefree integral ideals coprime to $S$
with norm $\asymp Z^V$, $k$ over nonzero elements of $\OO$ with
norm $\ll Z^M$, and $n$ over squarefree integral ideals with norm
$\asymp Z^N$, using primary generators and the stated zero extensions. The squarefree ideal $\mathfrak r$ is the
coprimality modulus, independent of $k,f$; $c_*>0$ is a fixed
positive gap in the two displayed length inequalities, and
$z_{\rm cap}$ an upper bound for total prime length.
Here $w(f)$ depends on $f$ alone and is nonnegative and
bounded by a fixed power of the divisor function. Also,
\[
 c_{\mathrm{sel}}(n)=\sum_{(p_i)}\prod_i a_i(p_i)1_{p_i\mid n}
\]
has total prime length at most $z_{\rm cap}$: its tuples run over disjoint
fixed prime sets of norms $\asymp Z^{z_i}$, with
$\sum_i z_i\le z_{\rm cap}$. The coefficients $a_i$ are bounded,
with data independent of
$k,f$. No additional arbitrary coefficient depending on $k$ is admitted.

The strict length inequalities in Lemma~\ref{lem:mobius-moment} give the fixed
positive gaps required for the induction on the outer range exponent $M$
in \eqref{eq:canonical}. Two finite Poisson transforms then decrease $M$
while preserving the common character and coprimality condition,
and keeping $4F_0-3M-6z_{\rm cap}$ bounded below by a fixed
positive constant \Src{Lemma}{17.3}.
The mean-square bound obtained from the functional equation \Src{Lemma}{14.3} handles the final
range of the induction, and \Src{Corollary}{17.4} bounds the seminorm and height orders.
These strict length inequalities are needed for the later counts.

\begin{lemma}[Sixth-power amplification, \Src{Lemma}{17.6}]
Fix a ray character $\nu$, a common sign $\varepsilon_\chi\in\{-1,1\}$,
and a test function $W$, with the zero extensions of
Lemma~\ref{lem:mobius-moment}. Put $\psi_u(n)=\nu(n)\chi_n(u)^{\varepsilon_\chi}$
and use $M_u$ as defined above, without auxiliary prime factors.
Let $r\ge0$ range over a fixed bounded interval and $q_u\asymp U$,
with $u$ sixth-power-free. For every $\eps>0$ and sufficiently large $U$,
\begin{equation}\label{eq:amplification}
 \sum_{q_u\asymp U}^{(6)}|M_u(U^r)|^2\ll
 U^{\mathfrak a(r)+\eps}(1+T_1)^A,\qquad
 \mathfrak a(r)=\max\{1,(1+5r)/6\}.
\end{equation}
The estimate permits a smooth weight depending on the element $u$
$W(y)y^{-\sigma_u+it_u}$, with $\sigma_u$ in a fixed compact interval
and $|t_u|\le T_1$. The factor $(1+T_1)^A$ accounts for its smooth seminorms; $A$ depends only on the fixed ranges and arithmetic data.
\end{lemma}

\paragraph{\it Idea of sixth-power amplification.}
Here ``amplification'' refers to the additional averaging over $a$
after replacing $u$ by $ua^6$. Since the sextic residue character
satisfies $\chi_n(ua^6)=\chi_n(u)$ when $(n,a)=1$, this
replacement preserves the character values on those arguments.
Inclusion--exclusion accounts for the remaining coprimality
restrictions, and averaging over $a$ allows a mean-square estimate
over a larger norm range to bound the original sum over $u$.
For a polynomial length $D_{\rm amp}\ge1$ and primary $a$,
inclusion-exclusion gives
\[
 M_u(D_{\rm amp})=\sum_{d\mid\rad a}\mu(d)\psi_u(d)q_d^{-1/2}
                                     M_{ua^6}(D_{\rm amp}/q_d).
\]
Average over $q_a\le P_{\rm amp}=(H_{\rm amp}/U)^{1/6}$ with
$H_{\rm amp}=\max(2U,D_{\rm amp}^{1+c})$. The map $(u,a)\mapsto ua^6$ is injective
on sixth-power-free $u$, and the enlarged mean-square bound contributes the factor
$H_{\rm amp}/P_{\rm amp}=U^{1/6}H_{\rm amp}^{5/6}$. Here $c>0$ is a small fixed
enlargement parameter. Choosing $c$ small proves
\eqref{eq:amplification}, including $r=1$ without letting the positive gap in the
length restriction tend to zero.

\subsection{Fourth moment of unweighted character polynomials}

In this subsection, we state the fourth-moment estimate for products of
unweighted character polynomials in \eqref{eq:plain-moment}
and explain the cancellation in its proof.
Applied to the unweighted detecting polynomial of Section~\ref{sec:detector},
this supplies the second bound in \eqref{eq:two-counts}.
Its dependence on $\kappa$ is the estimate extended in
Section~\ref{sec:iv-moment} to obtain the final improvement.

Here ``unweighted'' means that the character sums contain no
M\"obius coefficients $\mu(\mathfrak a)$, in contrast to the
sum $M_u$ in Section~\ref{sec:moments}. The smooth test functions
$W_j$ are still present.

\begin{lemma}[Unweighted product moment, \Src{Lemma}{18.1}]
Let $\nu_{\rm fix}$ be a finite-order Hecke character presentation
with the fixed ray and residue-symbol factors allowed in
\Src{Lemma}{18.1}, independent of the summation variable $k$.
Put $\psi_k(n)=\nu_{\rm fix}(n)\chi_n(k)$.
Let $m_0\ge0$ specify the range $0<q_k\ll Z^{m_0}$.
Let $q_0\ge0$ be an exponent such that the product of the distinct
prime ideals supporting the displayed variable residue-symbol
factors has norm at most $Z^{q_0}$.
These supporting primes are counted before any cancellation of
factors or reduction of exponents modulo six. Put $M=m_0+q_0$. Allow a common
coprimality condition with modulus of norm $O(Z^B)$ for a fixed
$B$, independent of $k$ and common to every factor. Prime coefficients are fixed finite combinations of
characters in $\Theta$.
For $\psi=\psi_k$, define
\[
 S_\psi(n_j)=Z^{-n_j/2}\sum_a\psi(a)W_j(q_a/Z^{n_j}),\qquad
 Q_k=\prod_i Z^{-z_i/2}\sum_{p\in\mathcal P_i}
 a_i(p)\psi_k(p)W_i(q_p/Z^{z_i}),
\]
where $a$ runs over nonzero integral ideals, using primary generators
and the stated zero extensions, and $n_1,n_2\ge0$ are length exponents, not ideals,
and $z_{\rm p}=\sum_i z_i$. Every factor includes the common
coprimality indicator just specified. All real length parameters
range over prescribed bounded sets. Let $3/4\le\kappa\le1$;
if $z_{\rm p}>0$ and $\kappa<1$, assume
$\beta_*\le(1+\kappa)/2$.
For every $\eps>0$ there exists $\ell_{\rm mesh}>0$, depending
only on these ranges and $\eps$, such that $z_i\le\ell_{\rm mesh}$
for every $i$ implies
\begin{equation}\label{eq:plain-moment}
 \sum_{k\in\mathcal R_{z_{\rm p}}}
 |S_{\psi_k}(n_1)S_{\psi_k}(n_2)Q_k|^2\ll Z^{M+\eps}.
\end{equation}
Here
$\mathcal R_{z_{\rm p}}$ consists of the elements $0<q_k\ll Z^{m_0}$
including all unit associates, whose primitive inducing characters
satisfy the exclusions below:
\[
 \begin{array}{c|c|c}
 &\text{length condition}&\text{excluded inducing characters}\\ \hline
 z_{\rm p}=0& n_1,n_2\text{ arbitrary bounded}&\text{principal}\\
 z_{\rm p}>0&n_1+n_2+6\kappa z_{\rm p}\le M&\Theta
 \end{array}
\]
The permitted $\ell_{\rm mesh}$ is independent of the number
of prime factors and uniform for $\kappa\in[3/4,1]$; constants and
finite seminorm orders may depend on the fixed number of factors.
\end{lemma}
In this application, the choice $\kappa=2\beta_*-1$ in
\eqref{eq:two-stages} gives $\beta_*=(1+\kappa)/2$, which
satisfies the lemma's hypothesis $\beta_*\le(1+\kappa)/2$.

\paragraph{\it Cancellation in the unweighted-moment argument.}
Applying the ordinary Hecke functional equation first reduces the length ranges. Two finite transforms then
reduce the combined range exponent $M$. For transformed characters in $\Theta$, the sums over the two ideal
variables may have main terms proportional to the product of their
summation lengths. For larger values of $n_1+n_2+z_{\rm p}$ the
original paper subtracts two sums whose length products agree,
using fixed ideals
$\mathbf b=(b_1,b_2)$, primary ideal variables $l_1,l_2$, and
positive real length parameters $X_i,Y_i$:
\begin{equation}\label{eq:rectangle}
 D_{\mathbf b}(l_1,l_2)=
 \prod_{i=1}^2V_i(q_{b_i}q_{l_i}/X_i)
 -\prod_{i=1}^2V_i(q_{b_i}q_{l_i}/Y_i),\qquad
 X_1X_2=Y_1Y_2.
\end{equation}
The two transforms preserve, on each side of the squared norm,
one character, one coprimality condition and one norm twist
on the full product of prime variables and unrestricted ideal variables
\Src{Lemma}{18.2}. For an inducing character in $\Theta$,
fixing the prime variables leaves
$\vartheta(l_1)\vartheta(l_2)1_{(l_1l_2,\mathfrak r_*)=1}
q_{l_1l_2}^{i\tau_{\rm tw}}D_{\mathbf b}(l_1,l_2)$.

Here $\vartheta\in\Theta$, $\mathfrak r_*$ is a squarefree
coprimality modulus with $q_{\mathfrak r_*}\le Z^{B_{\rm cop}}$ for
fixed $B_{\rm cop}$, and $\tau_{\rm tw}\in\mathbb R$ is the common norm-twist
height. The $V_i$ are the two test functions of ideal norms; the prime-interval
weights remain $W_i$.
By \Src{Lemma}{18.3}, the two product main terms have a common
coefficient, independent of $X_i,Y_i$, multiplying respectively
$(X_1X_2)^{1+i\tau_{\rm tw}}$ and $(Y_1Y_2)^{1+i\tau_{\rm tw}}$. Hence
\[
 (X_1X_2)^{1+i\tau_{\rm tw}}-(Y_1Y_2)^{1+i\tau_{\rm tw}}=0,
\]
which is the decisive cancellation in \eqref{eq:rectangle}.
The cited lemma supplies the uniform Poisson remainder.
Direct volume bounds suffice first in the uncentered range
$n_1+n_2+z_{\rm p}\le5M/6$; the cancellation handles the
remaining ranges. Induction on the combined range exponent proves all
cases with $z_{\rm p}=0$ before those with $z_{\rm p}>0$.
Bounding the combined support errors uniformly by $O_K(1/\log Z)$
keeps the maximum-length restriction independent of $K$.
This cancellation is what completes the inductive estimate
without discarding those transformed characters.

\section{Counting elements from simultaneous large values}
\label{sec:counts}

In this section, we combine the simultaneous large values of Section~\ref{sec:detector}
with the two moments \eqref{eq:inverse-moment} and \eqref{eq:plain-moment}.
The two counting estimates are displayed in \eqref{eq:two-counts};
optimizing the detecting lengths gives the count
\eqref{eq:count-final}, with exponent
$\max(R_{\rm short}(t),R_{\rm long}(t))$ for each zero and prime-sum
size class. This exponent enters the contour estimate
in Section~\ref{sec:exponents}.
The fixed-parameter count in Part II and the variable-parameter
count in Part IV repeat this combination with their respective
moment hypotheses.

\subsection{Prime-sum size classes and admissible detecting lengths}
\label{sec:prime-sum-classes}

In this subsection, we partition the elements according to the sizes of
the auxiliary prime sums, using the size parameter in \eqref{eq:gbar},
and determine the lengths permitted by \eqref{eq:prime-lengths}.
This partition records both the available savings from large prime sums
and their contribution to the nonprincipal integrals, so that the
counting bounds below and the contour exponent \eqref{eq:E} use the
same size classes.

For a fixed Mellin parameter $z$ on those contours, put
\[
 Q_i(u;z)=P_i^{-1/2}\sum_{p\in\mathcal P_i}
 \overline{\chi_p(u)}W_i(q_p/P_i)(q_p/P_i)^{z-1},\qquad
 \mathcal Q_i=-P_i^{z-1/2}Q_i.
\]
The logarithmic-derivative estimate in the zero-free rectangle
gives $|Q_i|\ll U^{\eps_1}P_i^{\delta/2+O(\eb)}$
\Src{Lemma}{19.1}. Choose a fixed small width for the exponent intervals $\theta_{\rm bin}>0$.
Partition by real exponents $g_i$ so
\[
 |Q_i|\le P_i^{g_i+\theta_{\rm bin}},\qquad
 g_i>0\Longrightarrow |Q_i|\ge P_i^{g_i},\qquad
 0\le g_i\le\delta/2.
\]
For an error factor set $g_i=0$. Write
\begin{equation}\label{eq:gbar}
 \bar g=\frac{\sum_i\ell_i g_i}{\ell},\qquad
 0\le\bar g\le\delta/2,\qquad x=\bar g/\delta\in[0,1/2].
\end{equation}
Ordering positive exponents decreasingly and selecting an initial
set of prime sums of total length at most $z_{\rm p}$ gives
\begin{equation}\label{eq:greedy}
 \left|\prod_{i\ {\rm selected}}Q_i\right|^2
 \ge U^{2\bar g z_{\rm p}-\delta\max_i(\ell_i/d)},
\end{equation}
up to a fixed small length decrement. Allowing a fraction of the final prime interval would give gain at
least $\bar g z_{\rm p}$; retaining only whole intervals loses at
most the contribution of one interval, giving the error in
\eqref{eq:greedy}.

After selecting a presentation, take
$\psi(n)=\nu(n)\overline{\chi_n(u)}$. The prime coefficient is
$\overline{\nu(p)}1_{p\in1_{\mathcal T}}$, a fixed finite
combination of $\Theta$-characters. Conjugate the whole
unweighted product when using \eqref{eq:plain-moment}.
If $u$ has a prime divisor outside $S$, the primitive character
induced by $\psi$ is ramified there and hence cannot belong to
$\Theta$, whose conductors are supported on $S$.
The required moment hypotheses therefore hold; smooth
parameter estimates handle the allowed element-dependent
smooth weights with a fixed polynomial dependence on the height.

Measure lengths in base $U$: the variable $u$ has norm
$\asymp U$, so its range exponent is one. The largest total prime
lengths allowed by the two moment inequalities are
\begin{equation}\label{eq:prime-lengths}
 z_M(r)=\frac{1-r}{2},\qquad
 z_S(m)=\frac{1-2m}{6\kappa}
       =\frac{1-2m}{9/2+12\Delta}.
\end{equation}
Using \eqref{eq:large-values}, \eqref{eq:greedy} and the moments,
\begin{equation}\label{eq:two-counts}
 \#\mathcal B\ll
 \min\left\{
 U^{1-\delta r-2\bar g z_M+\eps},
 U^{1-2\delta m-2\bar g z_S+\eps}\right\}(1+T_1)^A
\end{equation}
where the respective lengths are admissible. The second bound
uses two copies of $S_m$, hence $|S_m|^4$.

\subsection{Optimization of the two counting estimates}

In this subsection, we optimize the two bounds in \eqref{eq:two-counts}
to obtain a uniform count for each zero and prime-sum size class.
The resulting exponent \eqref{eq:count-final} is substituted into the
nonprincipal contour estimate in Section~\ref{sec:exponents}.

\begin{proposition}[Refined character count, \Src{Proposition}{19.2}]
Assume the Stage-2 setting \eqref{eq:two-stages}, so
$\Delta=\beta_*-7/8\in(0,1/24]$ and $\kappa=3/4+2\Delta$.
Let $\mathcal B$ consist of elements in one fixed zero class of
Section~\ref{sec:detector}, defined by \eqref{eq:zero-class},
and one fixed prime-sum size class of
Section~\ref{sec:prime-sum-classes}, specified by the exponents $g_i$,
with $q_u\asymp U=Z^d$, $a>a_0$, $0<\delta\le\alpha$,
$0\le\bar g\le\delta/2$ and $x=\bar g/\delta$.
Assume the detector height hypotheses and that the available
prime length $\ell/d$ exceeds $7/37$ by a fixed positive amount.
Put
\[
 D_x=3-17x/9,\qquad P_x=(2-8x/9)(1-x).
\]
For $t\in[1,3/2]$ define
\begin{equation}\label{eq:Rshort}
 R_{\rm short}(t)=1-\delta+\frac{\delta P_x}{D_x}(3/2-t).
\end{equation}
\begin{equation}\label{eq:Rlong}
 R_{\rm long}(t)=1-\delta+(\alpha-\delta)(t-1),\qquad\alpha=5/6.
\end{equation}
Then, for every $\eps>0$, the length decrements and maximum
individual prime length can be chosen so that
\begin{equation}\label{eq:count-final}
 \#\mathcal B\ll
 U^{\max(R_{\rm short}(t),R_{\rm long}(t))+\Delta/4+\eps}
 (1+T_1)^A.
\end{equation}
Here $T_1$ is the zero-detection height cutoff, and $A\ge0$ is
a fixed finite exponent measuring polynomial growth in that height.
The constants and exponent $A$ are uniform in $u$; their
arithmetic dependence is on the fixed data.
With no selected prime sums, $t=1$ instead gives
$\#\mathcal B\ll U^{1-2\delta/3+\eps}(1+T_1)^A$ without
any hypothesis on the available prime length.
\end{proposition}
\paragraph{\it The character-count exponent calculation.}
At $\delta=\alpha$, choose the detector parameter $t=3/2$,
so $r\ge1-O(\eps)$. Sixth-power amplification then gives
$\#\mathcal B\ll U^{1-\delta+\eps}(1+T_1)^A$, which is at least
as strong as \eqref{eq:count-final}; hence assume $\delta<\alpha$.
For short inverse length $r<1$ and $m<1/2$, compare the
unweighted estimate first with its $\Delta=0$ prime length.
Using $m\ge t-r-O(\eps)$ gives
\[
 A_I(r)=1-\delta\bigl[x+(1-x)r\bigr],\qquad
 A_S(r)=1-\delta\bigl[4x/9+(2-8x/9)(t-r)\bigr].
\]
Let
\[
 r_*(t)=\frac{(2-8x/9)t-5x/9}{D_x}.
\]
Then $r_*(t)\ge23/37$ and
$1/3\le t-r_*(t)\le1/2$.
The optimal common upper exponent is \eqref{eq:Rshort}.

On the inverse side a fixed decrement $\nu_0>0$ ensures
\[
 1-r-2z_{\rm p}\ge2\nu_0,\qquad
 3-2r-8z_{\rm p}=4(1-r-2z_{\rm p})+(2r-1),
 \quad 2r-1\ge9/37.
\]
Its required prime length is at most $7/37$, whereas
$\ell/h=8/39>7/37$. On the unweighted side, restoring the
actual length in \eqref{eq:prime-lengths} increases the counting exponent by at most
\[
\begin{aligned}
 2\bar g(1-2m)\left(\frac29-\frac1{9/2+12\Delta}\right)
 &\le \frac56\left(\frac13+O(\eps)\right)\frac{16}{27}\Delta\\
 &\le \frac{40}{243}\Delta+O(\eps)
 <\Delta/4+O(\eps).
\end{aligned}
\]
Here $\bar g\le\delta/2\le5/12$,
$m\ge t-r_*(t)-O(\eps)\ge1/3-O(\eps)$, and
$\frac29-\frac1{9/2+12\Delta}\le\frac{16}{27}\Delta$.
For $r\ge1$, \eqref{eq:amplification} instead gives
\eqref{eq:Rlong}.

If an admissible auxiliary prime length is zero or very small,
use the moment estimates without auxiliary prime-polynomial factors.
These comparisons give the count exponent in \eqref{eq:count-final}.
With no selected prime sums and $t=1$, balance
$1-\delta r$ and $1-2\delta(1-r)$ at $r=2/3$ to obtain
the stated unweighted count.

\section{Contour estimates and the endpoint inequality}
\label{sec:exponents}

In this section, we insert the count \eqref{eq:count-final} into the
nonprincipal terms $u\ne1$ of the Poisson integral \eqref{eq:Poisson},
with the modified correction \eqref{eq:modified-H} and normalization
\eqref{eq:normalization}. We compare their powers of $Z$ with the
principal Mellin integral $f_\eta(Z)$ in \eqref{eq:f}, whose exponent
$C(s)$ is given in \eqref{eq:C}.
The endpoint inequality \eqref{eq:endpoint} and the bounds for the remaining frequency
ranges provide the positive exponent savings assembled in
Section~\ref{sec:parameters}.

\subsection{The exponent of a nonprincipal contribution}

In this subsection, we insert the character count into the contour
estimates and compute the exponent of each nonprincipal contribution.
Subtracting the principal-term exponent gives $E(d)$ in \eqref{eq:E};
a negative value supplies the saving required by
\eqref{eq:criterion-comparison}.

Recall that $q_u\asymp U=Z^d$, $\delta=2a-1$ is the
zero-class parameter of \eqref{eq:zero-class}, and $\bar g$ is the
weighted mean of the prime-sum size exponents in \eqref{eq:gbar}.
The length parameters $l_y,\ell,h$ are those of \eqref{eq:geometry}.
Move the original zero-class integral with the holomorphic
correction \eqref{eq:modified-H} to
$s_r=a+16\eb,w_r=1-a-6\eb,z_r=z_0:=17/50$.
Only afterward subdivide pointwise by prime-sum sizes.
For a subset with cardinality $\ll U^{R+\eps}(1+T_1)^A$,
the factors, before small losses, are
\[
 \underbrace{X^{1/2-z_0}Z^{a+z_0-1}Y^{-a}}_{\text{Mellin powers}}
 \underbrace{U^{R+\delta/2-z_0}}_{\text{count and numerator}}
 \underbrace{Z^{\ell(z_0-1/2)+\bar g\ell}}_{\text{prime sums}}.
\]
The numerator and all error factors are estimated jointly by
\eqref{eq:joint-error}. Subtracting $C(7/8)=3/16$ gives
\begin{equation}\label{eq:E}
\begin{split}
 E(d)&=a-\frac78+h(z_0-1/6)-a l_y-\ell/2+\bar g\ell
                 +d(R+\delta/2-z_0)\\
 &=-\frac1{48}+\frac23\delta+\frac{\bar g}{6}-h(1-R)
                 +(d-h)(R+\delta/2-z_0)
 \qquad\Src{Lemma}{20.1}.
\end{split}
\end{equation}
The integral is bounded by the supremum of this exponent over
subsets defined by the sizes of the prime sums at each fixed
Mellin parameter. The contour shift is made before this subdivision,
so no holomorphic continuation of a subset is required.

\subsection{Choice of the detecting cutoff and proof of the endpoint saving}

In this subsection, we choose the zero-detector truncation exponent
as in \eqref{eq:optimal-t} and prove the endpoint inequality
\eqref{eq:endpoint}. The explicit lower bound
\eqref{eq:positive-bound} supplies a positive saving.
The endpoint $d=h$ controls the central frequency ranges by the
monotonicity of \eqref{eq:E}; the next subsection handles the smaller
and larger ranges needed for the full comparison estimate.

\begin{lemma}[Endpoint saving, \Src{Lemma}{20.2}]
Let $1/50<\delta<\alpha=5/6$ and $0\le x\le1/2$,
with $D_x,P_x$ as above and the fixed parameter choices \eqref{eq:geometry}.
Choose
\begin{equation}\label{eq:optimal-t}
 \mathcal J=(\alpha-\delta)D_x+\delta P_x,\qquad
 t=1+\frac{\delta P_x}{2\mathcal J},\qquad
 R_*=R_{\rm short}(t)=R_{\rm long}(t).
\end{equation}
Let $E_*=E(h)$ in \eqref{eq:E} with $\bar g=x\delta$ and
$R=R_*$. Then $-E_*\ge49/440640$.
If the count loss $\Delta/4$ is included and all remaining
exponent losses have sum $\eps_{\rm total}$, then
\begin{equation}\label{eq:endpoint}
 E_{\rm actual}(h)-\Delta
 \le-\frac{49}{440640}-\frac{51}{64}\Delta+\eps_{\rm total}.
\end{equation}
Here $E_{\rm actual}$ denotes the resulting exponent relative
to $C(7/8)$.
\end{lemma}
\paragraph{\it The endpoint exponent calculation.}
Here $1<t<3/2$ and $35/54\le\mathcal J\le5/2$.
Put
\[
\begin{gathered}
 y=1/2-x\in[0,1/2],\quad v_y=51+41y,\quad p_y=7+18y+8y^2,\\
 j_y=185+170y+(-138+12y+96y^2)\delta.
\end{gathered}
\]
Then $P_x=p_y/9$ and $\mathcal J=j_y/108$. The original paper's
polynomial identity is
\begin{equation}\label{eq:positive-identity}
\begin{split}
 10368v_y\mathcal J(-E_*)={}&
 (3+5y)\bigl[(4v_y\delta-79)^2+49\bigr]\\
 &+4y\Bigl[4v_y\delta((1+3y)(15+32y)\delta+9-13y)
                       +265+3485y\Bigr].
\end{split}
\end{equation}
Since $9-13y\ge5/2$ and $v_y\le17(3+5y)$,
\begin{equation}\label{eq:positive-bound}
 -E_*\ge\frac{49}{176256\mathcal J}
       \ge\frac{49}{440640}>\frac1{10000}
.
\end{equation}
Restoring the $\Delta/4$ in \eqref{eq:count-final}, the
coefficient of $R$ in \eqref{eq:E} is $h$. Since
$C(\beta_*)-C(7/8)=\Delta$, this gives \eqref{eq:endpoint}.
Here $E_{\rm actual}$ includes the $\Delta/4$ count loss,
and $\eps_{\rm total}$ denotes the sum of the small contour,
length, size-partition and smoothing losses. The subtraction
of $\Delta$ is essential.

At $\delta=\alpha$, use $t=3/2$ and unweighted amplification:
$R=1-\delta$, whence $E(h)\le-1/48-\delta/16$.
At $a=a_0$ use the direct count $R=1$, giving
$E(h)\le-7/1200$. No selected zero or moment estimate with a length restriction
approaching equality is needed in these respective boundary cases.

\subsection{Estimates over all nonprincipal norm ranges}

In this subsection, we combine the endpoint saving with the bounds for
small and large norms to control all nonprincipal frequency ranges.
This completes the estimates that Section~\ref{sec:parameters}
combines into a single saving uniform over all classes.

Let $d_{\min}=1/100$ and $0<\zeta<1/48$.
The estimates covering $U=Z^d$ are:
\begin{enumerate}
\item $d<d_{\min}$: absolute tuple bounds on
$(s_r,w_r,z_r)=(\beta_*+\eb,1/2,17/50)$ give saving
$79/800-2d_{\min}=63/800$ relative to $C(\beta_*)$.
Here $2d_{\min}$ bounds the contribution $63d/50$ to the exponent for $d<d_{\min}$.
\item $d_{\min}\le d\le1/2$: using
$R=76/75-2\delta/3$ (also valid at $a=a_0$),
\[
 E(d)\le E(1/2)\le-529/2400+(25/96)\delta\le-49/14400.
\]
The slope in $d$ is
$R+\delta/2-z_0=101/150-\delta/6>0$, since $\delta\le5/6$;
this justifies $E(d)\le E(1/2)$.
\item $1/2\le d\le h$: the total available prime length is
$\ell/d\ge8/39>7/37$, and
$R+\delta/2-z_0\ge33/50-\delta/2\ge4/25>0$.
Thus \eqref{eq:endpoint} controls these ranges.
\item $h<d\le h+\zeta$: the total available prime length remains $>1/5>7/37$.
The derivative of the exponent with respect to $d$ is below two,
so its increase is at most $2\zeta$.
\item $d>h+\zeta$: the contours $(\Rea s,\Rea w,\Rea z)=(2,2,z_{\mathrm{new}})$,
where the series converge absolutely, give
\[
 \sum_{U>Z^{h+\zeta}}|\mathcal R_\eta(U;Z)|
 \ll Z^{53/32+(h+\zeta)(1+\eps)-\zeta z_{\mathrm{new}}+\eps}.
\]
Here $U$ runs over the dyadic parameters indexing the ranges $q_u\asymp U$, and
$\mathcal R_\eta(U;Z)$ denotes the modified nonprincipal
contribution from $q_u\asymp U$ in \eqref{eq:Poisson}, before
normalization, on these contours of absolute convergence \Src{Lemma}{16.2}.
Choose the fixed $z_{\mathrm{new}}>2$ sufficiently large.
\end{enumerate}
The estimates for the small- and large-norm ranges are
\Src{Lemmas}{10.6, 16.2}.
Only the ranges $d\ge1/2$ use selected prime sums in the
character-count moments.

\section{Auxiliary parameter choices and the uniform comparison estimate}
\label{sec:parameters}

In this section, we choose the auxiliary parameters in an order that preserves all
strict inequalities and exponent savings obtained in Sections~\ref{sec:principal}--\ref{sec:exponents}.
The resulting direct and comparison estimates are displayed together
in \eqref{eq:final-estimates} and have precisely the
uniformity required by Proposition~\ref{prop:criterion-summary}, whose
application completes the Hecke argument in the next section.

Recall that $\Delta=\beta_*-7/8>0$ is the assumed gap in the
contradiction argument, and $C(s)=s-11/16$ is the principal
exponent in \eqref{eq:C}. The normalized average $J_\eta$ and
principal integral $f_\eta$ are given in \eqref{eq:normalization}
and \eqref{eq:f}.

The lowest zero class $a=a_0=51/100$ of
\eqref{eq:zero-class} and the intermediate norm range
$d_{\min}\le d\le1/2$ considered in Section~\ref{sec:exponents} satisfy
$E\le-7/1200$ and $E\le-49/14400$, respectively, relative to
$C(7/8)$. Relative to $C(\beta_*)=C(7/8)+\Delta$, they therefore
save at least $\Delta$. Since $51\Delta/64<\Delta$, both are
already covered by $m_{\rm high}$ below. The common exponent savings are
\[
 m_{\rm high}=\min\{51\Delta/64,63/800\},\quad
 m_w=23/960,\quad m_z=13/9600.
\]
Choose small moment errors and strict length decrements.
Recall that $\ell_{\rm mesh}$ bounds each individual prime-length
exponent in the moment estimate \eqref{eq:plain-moment}.
Obtain the permissible value $\ell_{\rm mesh}>0$ from \Src{Lemma}{18.1} before choosing $K$.
Then choose a fixed even $K$ with equal $\ell_i=\ell/K$, where
$2\ell/K$ is smaller than $\ell_{\rm mesh}$ and than the allowed exponent loss from discarding a partial
prime interval. Smooth weights in disjoint subintervals of $(1,2)$
give disjoint prime intervals with the common norm parameter $Z^{\ell/K}$.
Thus $\ell_i/d\le2\ell/K$ throughout the selected ranges.
Choose $0<\kappa_P<(7/8)\ell/K$, then the width of the
prime-sum size intervals, $\eb$, $\zeta$, $z_{\mathrm{new}}$ and $P_0$
so that the resulting losses are smaller than the remaining
positive exponent savings. All these parameters are chosen before fixing $\eta$.

One may retain
\begin{equation}\label{eq:margin}
 \mu_0=\frac14\min\{m_{\rm high},m_w,m_z,\kappa_P\}>0,\qquad
 \omega_0=\Delta/2.
\end{equation}
After fixing $\eta$ and its arithmetic data, choose all internal
seminorm orders. \Src{Lemmas}{10.2--10.4} and the preceding
estimates give
\begin{equation}\label{eq:height}
 |J_\eta-f_\eta|
 \ll_{\eta,N}Z^{C(\beta_*)-\mu_0}(1+T_1)^{A_\eta}
                   +Z^{B_\eta}T_1^{-N}.
\end{equation}
$A_\eta,B_\eta$ are finite nonnegative exponents fixed before
the external decay order $N\in\mathbb Z_{\ge0}$. The later
$N_\eta$ is a sufficiently large choice of this order. There is a fixed $\tau_{0,\eta}>0$ ensuring
all zero-detection height conditions for $T_1=Z^\tau$,
$0<\tau\le\min(d_{\min}/100,\tau_{0,\eta})$.

Choose
\[
 0<\tau_\eta\le
 \min\left\{d_{\min}/100,\tau_{0,\eta},
                 \frac{\mu_0}{4(A_\eta+1)}\right\},\qquad
 B_\eta-N_\eta\tau_\eta<C(\beta_*)-\mu_0/2.
\]
The first term in \eqref{eq:height} saves at least
$3\mu_0/4$, and the second saves $\mu_0/2$.
Increasing the derivative orders used to obtain decay in the
truncated contour tails changes only the seminorm constants and
the lower threshold for $Z$. This is the height choice of
\Src{Lemma}{11.1}; the full order is \Src{Proposition}{20.3}.
Together with \eqref{eq:normalized-direct},
\begin{equation}\label{eq:final-estimates}
 |J_\eta(Z)|\ll_\eta Z^{3/16+\Delta/2},\qquad
 |J_\eta(Z)-f_\eta(Z)|\ll_\eta Z^{C(\beta_*)-\mu_0/2}.
\end{equation}
The exponents are independent of $\eta$. Neither $J_\eta$ nor
$f_\eta$ contains $T_1$.

\section{Mellin continuation and proof of the Hecke bound}
\label{sec:contradiction}

In this section, we apply Proposition~\ref{prop:criterion-summary} to the estimates
assembled in \eqref{eq:final-estimates} of
Section~\ref{sec:parameters} and conclude the $7/8$ bound
for the Hecke family.
We also extend the conclusion to characters presented modulo a larger
modulus than their conductor, retaining their prescribed zero extensions; the
Dirichlet consequence follows in Section~\ref{sec:transfer}.

\begin{theorem}[Hecke conclusion, \Src{Theorem}{1.1}]
\label{thm:hecke-conclusion}
The family supremum in \eqref{eq:beta} satisfies $\beta_*\le7/8$.
Consequently every finite-order Hecke $L$-function over $F$ has
no zeros in $\Rea s>7/8$, with principal poles allowed.
\end{theorem}
\paragraph{\it The final continuation contradiction.}
Assume $\beta_*>7/8$. Apply Proposition~\ref{prop:criterion-summary} with
\[
 \sigma_0=7/8,\qquad c=-11/16,\qquad
 \omega_0=\Delta/2,\qquad \sigma=\mu_0/2.
\]
Holomorphy of $H_\eta$ follows from \eqref{eq:D2} and
\eqref{eq:H-c}; the chosen cutoff gives \eqref{eq:H-uniform}.
Together with \eqref{eq:final-estimates}, these verify the criterion's
hypotheses; $C(7/8)=3/16$ and
$0<\omega_0<\beta_*-7/8$. Its contradiction therefore gives
\begin{equation}\label{eq:conclusion}
 \boxed{\beta_*\le7/8}.
\end{equation}
For an imprimitive character $\eta$ induced by $\eta^*$,
\[
 L_F(s,\eta)=L_F(s,\eta^*)
 \prod_{\substack{\mathfrak p\mid\mathfrak f_\eta\\
                   \mathfrak p\nmid\mathfrak f_{\eta^*}}}
 (1-\eta^*(\mathfrak p)q_{\mathfrak p}^{-s}).
\]
Here $\mathfrak f_\eta$ is the defining modulus and
$\mathfrak f_{\eta^*}$ the primitive conductor. Each finite factor is
nonzero for $\Rea s>0$, giving the same conclusion for every presentation.

\section{Transfer to Dirichlet \texorpdfstring{$L$}{L}-functions}
\label{sec:transfer}

In this section, we transfer the Hecke bound to Dirichlet $L$-functions by pulling a
Dirichlet character back along the ideal norm and factoring its
$L$-function over $F$.
The factorization \eqref{eq:transfer} gives the transfer stated in
Corollary~\ref{cor:dirichlet-conclusion}.
The same argument will transfer each improved Hecke bound in
Parts II--IV, so it suffices to give the factorization and pole
check once.

\begin{corollary}[Dirichlet conclusion, \Src{Proposition}{11.3}]
\label{cor:dirichlet-conclusion}
Every Dirichlet $L$-function, including $\zeta$, has no zeros in
$\Rea s>7/8$, with principal poles at $s=1$ allowed.
\end{corollary}
\paragraph{\it The Dirichlet factorization.}
The identity below is the quadratic case of Artin formalism
(compatibility with induction and direct sums); see
\cite[Chapter~VIII, Section~10]{MilneCFT}.
For a primitive Dirichlet character $\chi$ of conductor $q$,
define $\eta(\mathfrak a)=\chi(N\mathfrak a)$ on ideals coprime
to $3q$, and let $\chi_{-3}$ be the primitive quadratic Dirichlet
character associated with $F/\mathbb Q$, given by the Kronecker
symbol $\chi_{-3}(n)=(-3/n)$.
Deleting primes dividing $3q$ and their prime ideals,
\begin{equation}\label{eq:transfer}
 \Lfun^{S_F}(s,\eta)=
 L^{S_{\mathbb Q}}(s,\chi)L^{S_{\mathbb Q}}(s,\chi\chi_{-3})
.
\end{equation}
The split and inert Euler factors give \eqref{eq:transfer} in
absolute convergence; meromorphic continuation extends the identity
\Src{Proposition}{11.3}. Both Dirichlet factors are holomorphic on
$7/8<\Rea s\le1$, $s\ne1$. A zero of either is therefore a
zero of their product, contradicting the Hecke conclusion.
Euler convergence handles $\Rea s>1$.
At $s=1$ a possible principal pole pairs only with
$\chi_{-3}$, whose value is
\[
 L(1,\chi_{-3})=\int_0^1\frac{dx}{1+x+x^2}
               =\frac{\pi}{3\sqrt3}>0.
\]
Finally, finite Euler deletions transfer the result to all
imprimitive Dirichlet presentations. The character
modulo one gives $\zeta$.

\part{A rational extension of the bound}
\label{part:rational}

\section{The rational zero-free bound and its parameters}
\label{sec:rational-parameters}

In this section, we state the rational improvement and choose its length parameters,
as stated in Theorem~\ref{thm:rational-extension} and
\eqref{eq:rational-geometry}. These choices will be verified in
Sections~\ref{sec:balanced-direct}--\ref{sec:extension-assembly}.
The estimates established in this part also supply the common analytic
construction for Parts III and IV. We vary the lengths in the modified average of
Section~\ref{sec:modified}; all arithmetic conventions, zero extensions,
smooth tests and character families remain those of Part I.
Recall that $F=\mathbb Q(\sqrt{-3})$, $\OO=\mathcal O_F$,
and $q_{\mathfrak a}=N\mathfrak a$. The finite-order target
character is $\eta$, and $\beta_*$ is the supremum of real parts
of family zeros defined in \eqref{eq:zeros}.
Throughout Parts II--IV we use the proved intermediate results
of \cite{QRH}. In particular, we apply
Theorem~\ref{thm:hecke-conclusion} before making the improved comparison.

\begin{theorem}[Rational extension]\label{thm:rational-extension}
Every finite-order Hecke $L$-function over
$F$ is zero-free in
\[
 \Rea s>B_{\mathrm r},\qquad B_{\mathrm r}=34999/40000.
\]
The same assertion holds for every Dirichlet $L$-function. Principal
poles are allowed, and the boundary line is excluded.
\end{theorem}

For clarity, the parameters that change are displayed together:
\begin{equation}\label{eq:rational-geometry}
\begin{gathered}
 \vartheta_{\mathrm r}=\frac1{10000},\quad b_{\mathrm r}=\frac18,
 \quad \ell_{\mathrm r}=\frac16+\vartheta_{\mathrm r}=\frac{5003}{30000},\\
 b=b_{\mathrm r},\qquad\ell=\ell_{\mathrm r},\\
 l_x=\frac{1-b-\ell}{2}=\frac{21247}{60000},\quad
 l_y=\frac{1+b-\ell}{2}=\frac{28747}{60000},\quad
 h=1-l_x+\ell=\frac{16253}{20000},\\
 X=Z^{l_x},\quad Y=Z^{l_y},\quad P_i=Z^{\ell_i},
 \quad \sum_{i=1}^K\ell_i=\ell.
\end{gathered}
\end{equation}
The pair $(b_{\mathrm r},\ell_{\mathrm r})$ specifies the rational
extension. We use $b=b_{\mathrm r}$ and $\ell=\ell_{\mathrm r}$ for these
parameter values, and retain $b,\ell$ for the general parameter formulas
shared with Part III.
Here $b=l_y-l_x$, while $\ell=\sum_i\ell_i$ is the total
selected prime-length exponent, and $h=1-l_x+\ell$ is the
endpoint of the central norm ranges. The integer $K$ is fixed and even. Its disjoint prime
sets $\mathcal P_i$ lie in the class $1_{\mathcal T}$ and outside $S$,
as in Section~\ref{sec:modified}; $1_{\mathcal T}$ is the identity
class of the fixed ray group and $S$ is the excluded prime set of
Section~\ref{sec:arithmetic}. The zero-class width $\eb$ of
\eqref{eq:zero-class} is independent of
$\vartheta_{\mathrm r}$. We retain
\[
 a_0=51/100,\quad \delta=2a-1,\quad
 x=\bar g/\delta\in[0,1/2],\quad z_0=17/50,\quad\alpha=5/6.
\]
Here $a$ is the real-part parameter of the zero class in
\eqref{eq:zero-class}, and $\bar g$ is the weighted mean defined
in \eqref{eq:gbar}. Thus $x$ denotes an amplitude ratio, whereas
$l_x$ denotes a length exponent.
In all applications $\delta\ge1/50$, so the ratio is defined.
Its extension to $\delta=0$ is used only in polynomial inequalities.

The balanced condition $l_x+l_y+\ell=1$ gives
\begin{equation}\label{eq:balanced-signal}
 C_b(s)=s-\frac{4+b}{6},\quad
 B(\ell)=\frac{11}{12}-\frac{\ell}{4},\quad
 C_b(B(\ell))=\frac{l_x}{2}+\frac b{12}.
\end{equation}
The shared balanced hypotheses for the extensions are
\begin{equation}\label{eq:balanced-hypotheses}
\begin{gathered}
 l_x+l_y+\ell=1,\qquad b=l_y-l_x>0,\qquad h=1-l_x+\ell,\\
 l_x>\ell>0,\qquad l_y>\ell,\qquad
 l_y-\ell-11b/6>0,\qquad1-3\ell>0,\qquad\ell<1/5,\\
 \ell/h\ge7/37+\rho_{\rm sup},\qquad\rho_{\rm sup}>0,
 \qquad5\ell>h,\\
 B=B(\ell)=11/12-\ell/4\in[437/500,7/8].
\end{gathered}
\end{equation}
The positive gap $\rho_{\rm sup}$ above $7/37$ in the
prime-length ratio $\ell/h$ is fixed independently of the target. The conditions on rescaled lengths and the additive estimate
are separated from the numerical choice of the length parameters.
For \eqref{eq:rational-geometry}, $C_b(s)=s-11/16$ and
$B(\ell)=B_{\mathrm r}$. The hypotheses are verified by
\begin{equation}\label{eq:rational-margins}
\begin{gathered}
 l_x-\ell=\frac3{16}-\frac32\vartheta_{\mathrm r}>0,\quad
 l_y-\ell=\frac5{16}-\frac32\vartheta_{\mathrm r}>0,\\
 l_y-\ell-\frac{11b}{6}=\frac1{12}-\frac32\vartheta_{\mathrm r}>0,
 \quad 1-3\ell>0,\quad \ell<\frac15,\\
 \frac\ell h>\frac8{39}>\frac7{37},\qquad
 5\ell-h=\frac1{48}+\frac72\vartheta_{\mathrm r}>0.
\end{gathered}
\end{equation}
They ensure that the same finite modification defines an average with
the new length parameters, that all rescaled lengths remain positive, and that
there are enough prime factors for the moments with selected prime factors.

\section{The direct estimate for balanced lengths}
\label{sec:balanced-direct}

In this section, we derive the direct estimate at variable balanced lengths from the
general estimate obtained from the functional equation \Src{Lemma}{14.3} and the additive estimate
\eqref{eq:Gram}. The resulting bound is \eqref{eq:balanced-low} in
Proposition~\ref{prop:balanced-low}; it supplies the direct hypothesis of
Proposition~\ref{prop:extension-assembly}; the specialized bound
\eqref{eq:BJ-bound} alone does not justify changing $\ell$.

If $\mathcal I\subset\{1,\ldots,K\}$ is the subset of rescaled factors,
put
\[
 d_{\mathcal I}=\sum_{i\in\mathcal I}\ell_i,\quad
 \ell'=\ell-d_{\mathcal I},\quad
 M'=1-\ell-2d_{\mathcal I},\quad
 X'=Z^{l_x-d_{\mathcal I}},\quad Y'=Z^{l_y-d_{\mathcal I}}.
\]
The symbol $d_{\mathcal I}$ is unrelated to the norm exponent $d$
used later.
For a fixed ray class $\sigma$ and rescaled prime tuple, let
$B^{\mathcal I}_{a,\sigma}$ denote the theta sum
$B^J_{m,\sigma}$ of \eqref{eq:BJ-bound}, with $J=\mathcal I$,
the outer summation variable renamed $a$, and the summation lengths replaced by $X',Y'$
above. The primes in the sets $\mathcal P_i$ with $i\notin\mathcal I$
are already summed, with the divisibility restrictions of
Section~\ref{sec:modified} imposed. Put $Q=q_{b_*}X'Y'$.

\begin{lemma}[Functional-equation estimate for balanced lengths]\label{lem:balanced-reflection}
Under the coefficient and smooth-weight hypotheses of
\Src{Lemma}{14.3}, the theta sum including all cube multiples, with remaining prime-length exponent $\ell'$
and outer element-range exponent $M'$ has squared norm
\begin{equation}\label{eq:balanced-row-energy}
 \sum_{q_a\ll Q}|B^{\mathcal I}_{a,\sigma}|^2
 \ll Z^{M'+\pos{d_{\mathcal I}+5\ell-1}/4+\eps}.
\end{equation}
Here $\pos{t}=\max(t,0)$ for a real argument $t$;
the loss $\eps>0$ is arbitrary.
\end{lemma}
\begin{proof}
We record the exponent calculation, retaining the normalization of
\eqref{eq:BJ-bound}. Write $O$ and $H$ for the exponents for which the norms of the
powerful part and remaining squarefree part of the outer element
are respectively $\asymp Z^O$ and $\asymp Z^H$, and put
$\Delta_H=M'-O-H$; the powerful part is defined after \eqref{eq:unmarked}.
The local auxiliary length exponents in \Src{Lemma}{14.3} satisfy, up to losses
that can be made arbitrarily small,
\[
 O,\Delta_H,S_0,B_0\ge0,\quad 2A_0\le O,\quad
 0\le N_0\le A_0,\quad z_a\le\ell',
\]
\[
 T_d=2H+2A_0+2z_a-1-\ell'-N_0-3B_0,
 \qquad v+3\ell_b+e_\lambda\le T_d.
\]
Here $A_0,N_0,B_0,S_0$ describe the local divisor decompositions,
$z_a$ is their retained prime-length exponent before applying the functional equation, $v,\ell_b,e_\lambda$
are the retained dual length exponents, and $T_d$ is the available
dual length exponent.
These variables occur only in this functional-equation calculation.
Since $M'+\ell'-1=-3d_{\mathcal I}$,
\[
 T_d\le H-3d_{\mathcal I}\le H.
\]

The exponent in the squared-norm estimate supplied by the cited lemma is
\begin{align*}
 E_{\rm ref}={}&O/2+H-S_0-B_0+z_a-u_{\rm hyb}
 -\ell_b-2e_\lambda/3\notag\\
 &\hspace{13mm}-\pos{T_d-v-3\ell_b-e_\lambda}/2,
 \qquad u_{\rm hyb}=\min\{v,z_a,(v+z_a)/3\}.
\end{align*}
If $u_{\rm hyb}=z_a$, then $E_{\rm ref}\le M'$.
Otherwise $u_{\rm hyb}\ge v/2$, and
\[
 u_{\rm hyb}+\ell_b+2e_\lambda/3
 +\pos{T_d-v-3\ell_b-e_\lambda}/2\ge T_d/4.
\]
Substituting $T_d$, using $N_0\le A_0$ and $z_a\le\ell'$, gives
\[
 E_{\rm ref}\le
 M'+\frac{1+3\ell'-2M'-2\Delta_H}{4}
 \le M'+\frac{d_{\mathcal I}+5\ell-1}{4}.
\]
The maximum of the two bounds proves \eqref{eq:balanced-row-energy}.
Fixed offsets arising from the support intervals are $O(1/\log Z)$; the Gaussian partition,
the losses from making the prime-length restrictions strict and from
logarithmically many dyadic intervals are absorbed in $\eps$.
Separation is performed on the joint weight before positivity or
squaring, so the separated coefficients remain independent of the outer element.
\end{proof}

\begin{proposition}[Balanced direct bound]\label{prop:balanced-low}
Assume $l_x+l_y+\ell=1$, $b=l_y-l_x>0$, and
\[
 l_x>\ell>0,\quad l_y>\ell,\quad
 l_y-\ell-11b/6>0,\quad1-3\ell>0,\quad\ell<1/5.
\]
Then
\begin{equation}\label{eq:balanced-low}
 |I_{\eta,\mathrm{mod}}(Z)|\ll_{\eta,\eps}
 Z^{l_x/2+b/12+\eps}.
\end{equation}
\end{proposition}
\begin{proof}
Apply the separation of \eqref{eq:separation} to each term of
\eqref{eq:modified}. In \eqref{eq:Gram} the parameter in the additive estimate is
$P_a=q_{b_*}^{-1}Z^b$, and the inequality
$l_y-d_{\mathcal I}-11b/6>0$ implies
$P_a^2/Y'\ll P_a^{1/6}$.
The Cauchy--Schwarz inequality \eqref{eq:separated-cauchy-schwarz},
\eqref{eq:Gram} and Lemma~\ref{lem:balanced-reflection} bound the contribution of each fixed tuple, before multiplication
by its coefficient in \eqref{eq:modified}, by
\[
 (X')^{1/2}P_a^{1/12}
 Z^{\pos{d_{\mathcal I}+5\ell-1}/8+\eps}.
\]
The number of rescaled tuples is $\ll Z^{d_{\mathcal I}+\eps}$.
Their coefficient is $O(Z^{-3d_{\mathcal I}/2})$, and replacing
$X'$ by $X$ contributes $Z^{-d_{\mathcal I}/2}$. Thus their net
additional exponent is
\[
 -d_{\mathcal I}+\pos{d_{\mathcal I}+5\ell-1}/8
 \le-7d_{\mathcal I}/8\le0.
\]
Sum over the finitely many subsets. This proves the claim without
invoking the specialized direct bound with altered length parameters.
\end{proof}

\section{Euler products and extraction of the principal term}
\label{sec:extension-euler}

In this section, we extend the holomorphic Euler correction \eqref{eq:modified-H} and extract
the principal Mellin term \eqref{eq:general-extension-signal} at the changed lengths of Section~\ref{sec:balanced-direct}.
The domain is specified in \eqref{eq:extension-domain}, and
\eqref{eq:extended-H} ensures that the principal correction does not
vanish. Proposition~\ref{prop:extension-principal} bounds the difference
between the principal contribution and its Mellin integral in
\eqref{eq:extension-principal-error}. These estimates are used in the
continuation arguments of Parts II--IV.

Let
\begin{equation}\label{eq:extension-domain}
 \frac{437}{500}\le B\le\frac78,\qquad
 \mathcal D_2(B)=\{\Rea s\ge B,\ \Rea w\ge19/20,
                         \ \Rea z\ge33/200\}.
\end{equation}
The first domain \eqref{eq:D1} remains available, with $\Rea(s+w)$ bounded above $1$ by a fixed positive amount. All Euler products use the same set $S$
as the modified average. A fixed prime cutoff is chosen first;
the target's conductor primes may then be added to $S$.

\begin{lemma}[Extended principal domain]\label{lem:extension-euler}
The unselected and selected local expressions of Part I are
holomorphic on $\mathcal D_2(B)$. Their full correction
\eqref{eq:modified-H} is defined there by the same local holomorphic
expressions.
For sufficiently large fixed cutoff $P_0$, the principal correction
satisfies
\begin{equation}\label{eq:extended-H}
 H_\eta(s)=1+O(P_0^{-4/5}),\qquad
 \sup_{\Rea s\ge B}|H_\eta(s)-1|\le\tfrac12.
\end{equation}
At the principal residue,
\begin{equation}\label{eq:extended-selected-prime}
 G_p(s,1,1/6)/H_p(s,1,1/6)=-1+O(q_p^{-B}).
\end{equation}
The constants are uniform in the imaginary parts of $s,w,z$ and
in the arithmetic factors of absolute value one appearing below.
\end{lemma}
\begin{proof}
To check the changed boundary, we give the exact local formulas from
\Src{Lemma}{7.1}. At a good prime put $j=v_p(u)\in\{0,\ldots,5\}$,
$\rho_p=\chi_p(u/p^j)$, and
\[
 a_p=\overline{\varpi(p)}^{\,3}\eta(p)^3,\quad
 b_p=\overline{\varpi(p)}^{\,2}\eta(p)^2\overline{\chi_p(4)},
 \quad b_p^3=a_p^2.
\]
Here $\varpi(p)=p/|p|$ is defined in Section~\ref{sec:arithmetic}.
The quantities $a_p,b_p$ have absolute value one and depend on $p$;
they are unrelated to the zero-class parameter $a$ and the length
difference $b$. Retain $V_p,W_p,D_p$ from
\eqref{eq:local-scalar-factors} and put
\[
 R=a_p^2q_p^{4-6s-6z},\qquad W_{\rm loc}=\rho_pq_p^{-w}.
\]
At $p\mid u$, $W_p=D_p=0$ by zero extension. With $\mathbf1$ denoting an
indicator, the complete local formulas are
\begin{align*}
 P_p^*&=\frac1{1-R}\left[
 \frac{R(1-q_p^{-1})-\eta(p)(q_p-1)q_p^{-s-w}V_p^{\mathbf1_{j\le1}}}{1-V_p}
 +J_j\right],\\
 P_p&=\frac1{1-V_p}+\mathbf1_{j=0}\frac{W_{\rm loc}}{1-W_{\rm loc}}+P_p^*,
 \qquad H_p=P_p\frac{(1-V_p)(1-W_p)}{1-D_p},
\end{align*}
where
\[
\begin{array}{c|l}
 j&J_j\\\hline
 0&-\eta(p)\rho_p^{-1}q_p^{-s}+W_{\rm loc}R\\
 1&\eta(p)q_p^{-s-w}\\
 2&a_p\rho_p^{-3}q_p^{3/2-3s}\\
 3&-\eta(p)b_p\rho_p^{-2}q_p^{2-3s-w}+b_p^2\rho_p^{-4}q_p^{2-4s}\\
 4&-\eta(p)a_p\rho_p^{-3}q_p^{5/2-4s-w}\\
 5&-a_p^2q_p^{3-6s}.
\end{array}
\]

Writing $\mathcal F_p=P_p^*+D_p$, cancellation gives
\begin{equation}\label{eq:extended-local-defect}
 H_p-1=\frac{D_p(V_p+W_p-V_pW_p)-V_pW_p+(1-V_p)(1-W_p)\mathcal F_p}{1-D_p}.
\end{equation}
For $p\nmid u$, the possible exponents are
\[
 -s_r-w_r,\quad-s_r-6z_r,\quad-w_r-6z_r,\quad
 4-6s_r-6z_r,\quad1-s_r-w_r-6z_r,
\]
where $s_r=\Rea s$, and similarly for $w_r,z_r$.
Their maximum on \eqref{eq:extension-domain} is $-B-47/50<-9/5$.
For $p\mid u$, the table instead gives $H_p-1=O(q_p^{1/20-B})$.
The denominators $1-R,1-V_p,1-D_p$ are uniformly separated from zero;
\eqref{eq:extended-local-defect} has no denominator $1-W_p$.
The Euler product over $p\nmid u$ converges locally uniformly,
while the finite product over $p\mid u$ is $O_\eps(q_u^\eps)$. For $u=1$, the tail saves
$B-3/50>4/5$, proving \eqref{eq:extended-H}.

The selected factor is defined without dividing by $H_p$:
\begin{equation}\label{eq:extended-G}
 G_p=\frac{(\overline{\eta(p)}q_p^s-q_p^{-w})(1-V_p)(1-W_p)P_p^*
                       -q_p^{-w}(1-V_pW_p)}{1-D_p}.
\end{equation}
It agrees with \eqref{eq:selected-local} on contours where the
series converge absolutely, and hence
continues the finite modification by \eqref{eq:modified-H}.
For $u=1$, the unselected factors $H_p$ are uniformly close to one
on \eqref{eq:extension-domain}, so $G_p/H_p$ is well defined
on the contours used to extract the principal residue. The four error exponents are
\[
 -s_r,\quad-6z_r,\quad4-5s_r-6z_r,\quad1-w_r-6z_r,
\]
all at most $-B$ for the range \eqref{eq:extension-domain}.

This proves \eqref{eq:extended-selected-prime}. No nonvanishing of a
nonprincipal local correction is needed.
\end{proof}

Recall that $c_S$ is the positive residue constant in
\eqref{eq:H-c}, and define $A_{\mathcal T}(Z)$ by
\eqref{eq:normalization}, with the present $\ell_i$ and the same
Mellin transform $\mathcal M$ of
$\widetilde W_0$ defined in Section~\ref{sec:average-poisson}.
The prime asymptotic \eqref{eq:prime-scalar} gives, for even $K$,
\[
 A_{\mathcal T}(Z)>0\quad(Z\text{ sufficiently large}),\qquad
 |A_{\mathcal T}(Z)|\asymp(\log Z)^{-K}.
\]
The poles remain $w=1$ and $z=1/6$; the latter is independent of $\ell$.
For arbitrary positive $l_x,l_y,\ell$, put
\begin{equation}\label{eq:general-extension-signal}
\begin{gathered}
 h=1-l_x+\ell,\qquad C(s)=s+l_x/2-1+h/6,\\
 J_\eta(Z)=\frac{I_{\eta,\mathrm{mod}}(Z)}{c_SA_{\mathcal T}(Z)},\qquad
 f_\eta(Z)=\frac1{2\pi i}\int_{(2)}
 Z^{C(s)}e^{(s-5/6)^2}\frac{H_\eta(s)}{L_F^S(s,\eta)}\,ds.
\end{gathered}
\end{equation}
These are the same constructions as \eqref{eq:f}, with the changed
length parameters. Neither function depends on a later height truncation.

\begin{proposition}[Principal extraction]\label{prop:extension-principal}
Assume \eqref{eq:balanced-hypotheses} and $\beta_*\ge B$.
The individual prime-length exponents satisfy $\ell_i>0$ and
$\sum_{i=1}^K\ell_i=\ell$, as in \eqref{eq:rational-geometry}. For every
$0<\kappa_P<B\min_i\ell_i$, the principal term $\mathcal P_\eta$ satisfies
\begin{align}\label{eq:extension-principal-error}
 \left|\frac{\mathcal P_\eta(Z)}{c_SA_{\mathcal T}(Z)}-f_\eta(Z)\right|
 \ll{}&Z^{C(\beta_*)+(1+h)\eb-l_y/20+\eps}
 +Z^{C(\beta_*)+\eb-h/600+\eps}\notag\\
 &+Z^{C(\beta_*)+\eb-\kappa_P+\eps}.
\end{align}
\end{proposition}
\begin{proof}
The identification of the principal frequency in
Section~\ref{sec:average-poisson} and
\Src{Lemmas}{10.2--10.5} is arithmetic and unchanged: the only principal
numerator occurs at $u=1$. For this frequency the scalar quotient is
$\zeta_F^S(6z)\zeta_F^S(w)/L_F^S(s,\eta)$.
Start with $\Rea s=\beta_*+\eb$, $\Rea w=1+\eb$,
$\Rea z=1/6+\eb$. Move $w$ to $19/20$, then in its residue move
$z$ to $33/200$. Lemma~\ref{lem:extension-euler} justifies both moves.
The raw exponent is
\[
 l_x/2+\Rea s-1+h\Rea z+l_y(\Rea w-1).
\]
The two unresidued integrals therefore save $l_y/20$ and $h/600$.
At the double residue, \eqref{eq:extended-selected-prime} and the
nonnegative prime tests give
\[
 \mathfrak H_{\eta,1,Z}(s,1,1/6)
 =H_\eta(s)Z^{\ell/6}A_{\mathcal T}(Z)
       (1+O(Z^{-\kappa_P})).
\]

The scalar residue is exactly $c_S$, including the factor $1/6$
from $\zeta_F^S(6z)$. Bound the error on the original contour $\Rea s=\beta_*+\eb$;
move only the main Gaussian integral right to $\Rea s=2$.
The reciprocal has no poles to the right of $\beta_*$, and principal
$L$-poles become reciprocal zeros. Gaussian decay removes horizontal
sides. This proves \eqref{eq:extension-principal-error} without moving
an approximation error through possible target zeros.
\end{proof}

\section{Character counts and nonprincipal exponents at the fixed moment parameter}
\label{sec:fixed-count}

In this section, we use the already proved $7/8$ bound to keep $\kappa=3/4$ fixed while
varying the averaging lengths.
Proposition~\ref{prop:fixed-count} gives the count
\eqref{eq:fixed-count-result}, and
Proposition~\ref{prop:general-exponent} gives the nonprincipal exponent
\eqref{eq:general-extension-exponent}. These are used
by the exact rational inequalities in Section~\ref{sec:rational-certificate}
and the algebraic optimization in Part III.

Recall that $\delta=2a-1$, $x=\bar g/\delta$, and
$q_u\asymp U=Z^d$ specify the zero, prime-sum size and norm
classes; see \eqref{eq:zero-class} and \eqref{eq:gbar}.
Theorem~\ref{thm:hecke-conclusion} permits the fixed choice
\begin{equation}\label{eq:fixed-kappa}
 \kappa=3/4,\qquad\beta_*\le(1+\kappa)/2.
\end{equation}
It is important to fix this value rather than substitute
$\kappa=2\beta_*-1$ below the allowed moment range.
All moments below have exactly the common-character, coefficient and
smooth-weight hypotheses of Section~\ref{sec:moments}.

\begin{proposition}[Fixed-parameter count]\label{prop:fixed-count}
Let $U=Z^d$, and let $\mathcal B$ be a set of the sixth-power-free
summation indices $u$ in \eqref{eq:Poisson}, with $q_u\asymp U$,
in one fixed zero class of \eqref{eq:zero-class} and one fixed
prime-sum size class of Section~\ref{sec:prime-sum-classes}.
The latter is also called an amplitude class below. Suppose $0<\delta\le\alpha$,
$0\le\bar g\le\delta/2$, and
$\ell/d\ge7/37+\rho_{\rm sup}$ for a fixed $\rho_{\rm sup}>0$.
Write $x=\bar g/\delta$, and retain $D_x=3-17x/9$ and
$P_x=(2-8x/9)(1-x)$ from the refined character count in
Section~\ref{sec:counts}. The detector parameter $t_*$ balances
the bounds for short and long detecting polynomials; $R_*$ is
the resulting counting exponent. Set
\begin{equation}\label{eq:fixed-count-parameters}
 \mathcal J=(\alpha-\delta)D_x+\delta P_x,\quad
 t_*=1+\frac{\delta P_x}{2\mathcal J},\quad
 R_*=1-\delta+
 \frac{(\alpha-\delta)\delta P_x}{2\mathcal J}.
\end{equation}
For every $\eps>0$, subdividing the total prime-length exponent
$\ell$ into sufficiently small positive exponents $\ell_i$, fixed
independently of the target character, gives
\begin{equation}\label{eq:fixed-count-result}
 \#\mathcal B\ll U^{R_*+\eps}(1+T_1)^A.
\end{equation}
Here $T_1$ is the height cutoff of Section~\ref{sec:detector},
and $A\ge0$ is a fixed finite height exponent under the
conventions of Section~\ref{sec:arithmetic}.
Without an assumption on the available prime length, the count without selected prime factors is
$\#\mathcal B\ll U^{1-2\delta/3+\eps}(1+T_1)^A$.
\end{proposition}
\begin{proof}
The simultaneously large detecting polynomials of \eqref{eq:large-values} have lengths
$U^r,U^m$, $r+m\ge t-O(\eps)$, and satisfy
\[
 |M_r|^2\gg U^{\delta r-O(\eps)},\qquad
 |S_m|^4\gg U^{2\delta m-O(\eps)}.
\]
Here $M_r,S_m$ are the M\"obius-weighted and unweighted detecting polynomials defined in
Section~\ref{sec:detector}, with the same character $\psi\in\mathcal X_u$, including its chosen
defining modulus and zero extension.
For a selected product of prime sums whose norm lengths have
product $U^z$, the inverse moment with selected prime factors
\eqref{eq:inverse-moment} and the unweighted moment \eqref{eq:plain-moment}
give the same set two count bounds:
\[
 U^{1-\delta r-2\bar g z+\eps}(1+T_1)^A,
 \qquad U^{1-2\delta m-2\bar g z+\eps}(1+T_1)^A.
\]

For elements with $q_u\asymp U$, the admissible selected-prime exponents
satisfy $z\le z_M$ or $z\le z_S$, respectively, where
\[
 z_M=(1-r)/2,\qquad z_S=2(1-2m)/9.
\]
Ordering the prime sums by decreasing size exponent and taking
the first intervals, as in \eqref{eq:greedy}, achieves gain at
least $\bar g z$ before discarding a partial interval; discarding
the one partially selected prime-polynomial factor loses at most $\delta\max_i(\ell_i/d)$ in the
squared lower bound. Decrease the selected prime-length bound by a fixed $\nu>0$ before applying
a moment; when the remaining prime-length bound is nonpositive, use the version without
selected prime factors. These errors fit the requested $\eps$.

For $r<1,m<1/2$, the affine count bounds without an additional
prime-length restriction are
\[
 A_I(r)=1-\delta\bigl[x+(1-x)r\bigr],\qquad
 S_t(r)=1-\delta\bigl[4x/9+(2-8x/9)(t-r)\bigr].
\]
They cross at
\[
 r_*(t)=\frac{(2-8x/9)t-5x/9}{D_x},\quad
 r_*(1)\ge23/37,\quad r_*(3/2)=1.
\]
Use the inverse estimate on $r\ge r_*$ and the unweighted estimate on
$r\le r_*$. The prime-length bound for the inverse moment is at most $7/37$, the bound
for the unweighted moment
is at most $2/27+O(\eps)$; the hypothesis on the available prime length suffices.
On the inverse side the second length inequality has the positive gap
\[
 3-2r-8z=4(1-r-2z)+(2r-1)\ge9/37,
\]
while the first length inequality has the positive gap $1-r-2z\ge2\nu$.

For short detecting polynomials, the count exponent is consequently
\[
 R_{\rm short}(t)=1-\delta+\frac{\delta P_x}{D_x}(3/2-t).
\]
For $r\ge1$, use the amplification estimate without selected prime factors \eqref{eq:amplification}:
\[
 \mathfrak a(r)-\delta r\le
 R_{\rm long}(t)=1-\delta+(\alpha-\delta)(t-1).
\]
For $m\ge1/2$ the unweighted bound without selected prime factors is at most $1-\delta$.
Balancing the short and long bounds gives \eqref{eq:fixed-count-parameters}.
Since $D_x-P_x=1+x-8x^2/9>0$,
\[
 \mathcal J\ge35/54,\qquad 1\le t_*\le3/2.
\]
At $\delta=\alpha$, the balanced values are $t_*=3/2$ and
$R_*=1-\alpha=1/6$. At $r=1$ the inverse prime-length bound is $z_M=0$;
use the estimate without selected prime factors \eqref{eq:amplification}. For $r\ge1$,
$\mathfrak a(r)-\alpha r=1/6$. Thus the endpoint uses a moment without selected prime
factors, rather than a limit
of moment estimates with selected prime factors whose strict
length gaps tend to zero.

For each character presentation chosen by the zero detector, both
moment estimates are applied to the same pair of detecting polynomials.
The dyadic choices of their lengths contribute only logarithmic factors. Common
smooth-weight separation puts height dependence into $(1+T_1)^A$.
There is no $\Delta/4$ term from \eqref{eq:count-final}: that term
compares a varying $\kappa$ with $3/4$, whereas \eqref{eq:fixed-kappa}
is already exact. The common character can be written as
$\psi(n)=\nu(n)\overline{\chi_n(u)}$, and each selected prime
coefficient is $\overline{\nu(p)}1_{p\in1_{\mathcal T}}$; it is a fixed
finite combination of $\Theta$-characters. Conjugate the whole unweighted
product to match \eqref{eq:plain-moment}. Elements divisible by a good prime
divisor induce outside $\Theta$, as in Section~\ref{sec:counts};
unit elements are covered by the small-norm estimate below.

With no selected factors, $t=1$ and the crossing
$r=2/3$ give the last assertion.
\end{proof}

\begin{proposition}[Exponent for arbitrary length parameters]\label{prop:general-exponent}
Let $B\le\beta_*$, and assume \eqref{eq:extension-domain}, positive
lengths and the unchanged joint-error and moment hypotheses of Part I.
Here $C(s)=s-5/6+l_x/3+\ell/6$ is the principal exponent in
\eqref{eq:general-extension-signal}, and $h=1-l_x+\ell$.
For a class with element-norm exponent $d$, count exponent $R$, zero-class parameter
$a$ of \eqref{eq:zero-class} and prime-sum size mean
$\bar g$ of \eqref{eq:gbar}, the normalized nonprincipal
contribution is, for each prescribed loss $\eps>0$ and sufficiently
small fixed contour and zero-class widths, bounded by
\[
 Z^{C(B)+E_B(d)+\eps}(1+T_1)^A
 =Z^{C(\beta_*)+E_B(d)-(\beta_*-B)+\eps}(1+T_1)^A,
\]
where
\begin{equation}\label{eq:general-extension-exponent}
 E_B(d)=a-B+h(z_0-1/6)-a l_y-\ell/2+\bar g\ell
             +d(R+\delta/2-z_0),\quad z_0=17/50.
\end{equation}
The factor $(1+T_1)^A$ records polynomial growth in the height
cutoff $T_1$ of Section~\ref{sec:detector}; $A\ge0$ is fixed and
finite, as in Section~\ref{sec:arithmetic}.
Thus $E_B(d)$ is measured relative to $C(B)$; relative to
$C(\beta_*)$, the exponent is $E_B(d)-(\beta_*-B)$.
When $l_x+l_y+\ell=1$ and $B=B(\ell)$,
\begin{equation}\label{eq:balanced-extension-exponent}
 E_B(h)=-\frac14+\frac b6+\frac{5\ell}{4}
 +(1/2+\ell)\delta+\ell\bar g-h(1-R).
\end{equation}
\end{proposition}
\begin{proof}
Use the exact Poisson identity \eqref{eq:Poisson}, which holds for
arbitrary positive summation lengths. Move the full holomorphic
correction before partitioning it into pointwise amplitude classes.
Choose the contour
\[
 (\Rea s,\Rea w,\Rea z)=(a+16\eb,1-a-6\eb,z_0).
\]
For example, first retain the global $s$ line and move $w$ so that
$\Rea(s+w)\ge1+14\eb$; after external height truncation move $s$
to the displayed line, where $\Rea(s+w)=1+10\eb$.
The zero-free separation in Section~\ref{sec:detector} and the Euler
domain \eqref{eq:D1} used in Part I apply.

The joint bound \eqref{eq:joint-error} applies the conductor-deficit estimate
once to the entire product and yields
\[
 \sum_{u\in\mathcal B}|L_F^S(w,\chi_\bullet(u))
                        \mathfrak H_{\eta,u,Z}|
 \ll U^{R+\delta/2+\eps}
 Z^{\ell(z_0-1/2)+\bar g\ell+\eps}(1+T_1)^A.
\]
The remaining factors contribute
\[
 Z^{l_x(1/2-z_0)+a+z_0-1-a l_y}\,U^{-z_0}.
\]
Subtracting $C(B)=B+l_x/2-1+h/6$ gives
\eqref{eq:general-extension-exponent}. Algebra using
$l_x+l_y+\ell=1$ gives \eqref{eq:balanced-extension-exponent}.
Horizontal joins and discarded heights are bounded by
$Z^{B_\eta}T_1^{-N}$: the fixed internal seminorm orders are chosen
before the arbitrary external integration-by-parts order $N$.
The class sets need no holomorphic dependence on the contour variables.
\end{proof}

\section{Polynomial verification of the rational saving and the remaining norm ranges}
\label{sec:rational-certificate}

In this section, we verify that the nonprincipal exponent
\eqref{eq:balanced-extension-exponent}, with the count
\eqref{eq:fixed-count-parameters}, has a strict saving at the rational parameters, uniformly over all zero
and amplitude classes. Here $E_B(d)$ is measured relative to the
principal exponent $C(B)$, as defined in
\eqref{eq:general-extension-exponent}, and $h=1-l_x+\ell$.
The polynomial inequalities and the remaining norm estimates supply the
positive exponent savings needed to prove Theorem~\ref{thm:rational-extension} in the next
section. The remaining-range estimates in Lemma~\ref{lem:extension-ranges}
are stated under the general balanced hypotheses, so they are also
used at the changed parameters in Parts III and IV.

\begin{lemma}[Uniform rational saving]\label{lem:rational-saving}
For \eqref{eq:rational-geometry}, $R=R_*$ and
$(\delta,x)\in[0,5/6]\times[0,1/2]$, one has
\begin{equation}\label{eq:rational-saving}
 E_{B_{\mathrm r}}(h)\le-19/320000.
\end{equation}
\end{lemma}
\begin{proof}
At the Part I parameter values $\ell=1/6,b=1/8,h=13/16$, denote the
balanced exponent by $E_0$. Put $y=1/2-x$ and
\[
 j=185+170y+(-138+12y+96y^2)\delta=108\mathcal J,
\]
\[
\begin{aligned}
 Q_0={}&10(37+34y)-8(237+377y+26y^2)\delta\\
       &+48(51+131y+94y^2+32y^3)\delta^2,\qquad
 Q=Q_0-96j/5000.
\end{aligned}
\]
Direct expansion gives
\begin{equation}\label{eq:rational-polynomial}
 Q=10368\mathcal J(-E_0-1/5000).
\end{equation}
To prove its sign on the whole rectangle, use the Bernstein basis,
named after Sergei Natanovich Bernstein, who introduced it in his
constructive proof of Weierstrass's approximation theorem
\cite{Bernstein1912}; see \cite{Farouki2012} for a modern account.
The Bernstein basis polynomials of degree $n$ on $[0,1]$ are
\begin{equation}\label{eq:bernstein-basis}
 \mathsf B_{i,n}(t)=\binom ni t^i(1-t)^{n-i},
 \qquad 0\le i\le n.
\end{equation}
They satisfy
\[
 \mathsf B_{i,n}(t)\ge0,\qquad
 \sum_{i=0}^n\mathsf B_{i,n}(t)=1\qquad(0\le t\le1).
\]
On $[a_1,a_2]\times[b_1,b_2]$, substitute
$\delta=a_1+(a_2-a_1)v$, $y=b_1+(b_2-b_1)w$ and write
$Q=\sum_{p=0}^2\sum_{q=0}^3 A_{pq}v^pw^q$.
Its representation in the tensor-product Bernstein basis is
\[
 Q(\delta,y)=\sum_{i=0}^2\sum_{k=0}^3
 B_{ik}\mathsf B_{i,2}(v)\mathsf B_{k,3}(w).
\]

The bidegree $(2,3)$ Bernstein coefficients are
\[
 B_{ik}=\sum_{p\le i,\ q\le k}A_{pq}
 \frac{\binom ip}{\binom2p}\frac{\binom kq}{\binom3q},
 \quad 0\le i\le2,\quad0\le k\le3.
\]
Nonnegativity and the partition-of-unity identity give
\[
 \min_{i,k}B_{ik}\le Q(\delta,y)\le\max_{i,k}B_{ik}
 \qquad ((\delta,y)\in[a_1,a_2]\times[b_1,b_2]).
\]
The listed rectangles cover
the original rectangle; their complete partition and coefficients
are verified by the accompanying exact arithmetic checker.
The last column is a coefficient minimum,
not a sampled value:
\begin{center}\small
\begin{tabular}{ccc}\toprule
$\delta$ interval&$y$ interval&$\min B_{ik}$\\\midrule
$[0,5/24]$&$[0,1/4]$&$1380407/18000$\\
$[5/24,5/16]$&$[0,1/8]$&$27677/2000$\\
$[5/16,35/96]$&$[0,1/16]$&$12437/8000$\\
$[35/96,25/64]$&$[0,1/32]$&$2301/16000$\\
$[25/64,5/12]$&$[0,1/32]$&$12581/32000$\\
$[35/96,5/12]$&$[1/32,1/16]$&$2952669/1024000$\\
$[5/16,5/12]$&$[1/16,1/8]$&$85933/192000$\\
$[5/24,5/16]$&$[1/8,1/4]$&$53811/2560$\\
$[5/16,5/12]$&$[1/8,1/4]$&$8213/800$\\
$[0,5/24]$&$[1/4,1/2]$&$336467/4000$\\
$[5/24,5/16]$&$[1/4,3/8]$&$946273/24000$\\
$[5/16,5/12]$&$[1/4,3/8]$&$666871/16000$\\
$[5/24,5/12]$&$[3/8,1/2]$&$70651/4000$\\
$[5/12,5/6]$&$[0,1/2]$&$319/125$\\\bottomrule
\end{tabular}
\end{center}
Thus $Q>0$ throughout, so $-E_0\ge1/5000$.

With $R_*,\delta,x$ fixed, changing $\ell$ at fixed $b$ gives
\[
 \partial_\ell E(h)=-1/4+\delta+\bar g+3R_*/2
 =5/4+\delta(x-1/2)
 +\frac{3(\alpha-\delta)\delta P_x}{4\mathcal J}\le45/32.
\]
We used $\mathcal J\ge\alpha P_x$ and
$\delta(\alpha-\delta)\le\alpha^2/4$. Therefore
$E(h)\le-1/5000+45/(32\cdot10000)=-19/320000$.
\end{proof}

The endpoint inequality alone is insufficient: elements of very small and very
large norm require different contours. We record a reusable division
of the norm ranges.

\begin{lemma}[Remaining ranges]\label{lem:extension-ranges}
Assume the general hypotheses \eqref{eq:balanced-hypotheses},
$B\le\beta_*\le7/8$, and $0\le\delta\le5/6$.
Then the following estimates hold, with
arbitrarily small fixed losses. Set $d_{\min}=1/100$.
\begin{enumerate}
\item For elements with $q_u\ll Z^{d_{\min}}$, including unit elements,
the normalized exponent relative to
$C(\beta_*)$ is at most
$13h/75-l_y/2+2d_{\min}$.
\item For $d_{\min}\le d\le1/2$ and $\delta\ge1/50$, use
$R=76/75-2\delta/3$ in \eqref{eq:general-extension-exponent}.
Its slope in $d$ is $101/150-\delta/6>0$; it is enough to bound
$E_B(1/2)$.
\item For $1/2\le d\le h$ use $R_*$.
The slope $R_*+\delta/2-z_0$ is at least $4/25$ and at most $2$;
it is enough to bound $E_B(h)$.
\item The lowest zero class $a=a_0$, without a detector, has $R=1$ and
$\bar g\le1/100$. Its endpoint exponent is at most
$-1/4+b/6+5\ell/4+(1/2+3\ell/2)/50$.
\item Extending the third range to $h+\zeta$ increases the exponent by at most $2\zeta$,
provided $\ell/(h+\zeta)>7/37$. Beyond $h+\zeta$, an absolute
contour with sufficiently large fixed $\Rea z=z_{\mathrm{new}}$ gives any
prescribed power saving.
\end{enumerate}
\end{lemma}
\begin{proof}
For the small-norm range retain the global reciprocal line
$\Rea s=\beta_*+\eb$ and put $\Rea w=1/2$, $\Rea z=z_0$.
Use the Euler-product domain \eqref{eq:D1}; the inequality
$\Rea(s+w)-1\ge B-1/2>0$.
Formula \eqref{eq:extended-G} and the local table imply
$G_p=O(1)$ for $p\nmid u$ and $G_p=O(q_p^{1/2})$ for $p\mid u$.
Consequently the full correction is
$\ll U^\eps Z^{\ell z_0}$: primes not dividing $u$ contribute $P_i^{z_0}$,
and primes dividing $u$ contribute at most
$U^\eps P_i^{z_0-1/2}$. The nonprincipal numerator contributes a factor
$U^{3/5+\eps}$ on this line, by the growth estimate used in
\Src{Lemma}{10.6}. Counting $O(U)$ original elements and the factor
$U^{-z_0}$ give $U^{63/50+\eps}$.

The exterior factor relative to $C(\beta_*)$ is
$Z^{13h/75-l_y/2+O(\eb)}$. Since $63/50<2$, the first claim follows.

For the intermediate norm range use the count without selected prime factors in
Proposition~\ref{prop:fixed-count}, allowing an extra $1/75$.
For the moderate norm range use its count with selected prime factors. Differentiating
\eqref{eq:general-extension-exponent} gives the slopes asserted;
for the latter use $R_*\ge1-\delta$ and $\delta\le5/6$.
The lowest zero class is bounded trivially by $O(U)$ and
$\bar g\le\delta/2$. The extension beyond $h$ follows from the
slope upper bound.
Finally, on $(\Rea s,\Rea w,\Rea z)=(2,2,z_{\mathrm{new}})$,
absolute convergence gives
\[
 \mathcal R(U;Z)\ll
 Z^{l_x/2+1+l_y+hz_{\mathrm{new}}+\eps}U^{1+\eps-z_{\mathrm{new}}}.
\]
Summing $U>Z^{h+\zeta}$ yields an exponent containing
$-\zeta z_{\mathrm{new}}$. Choose $z_{\mathrm{new}}$ after $\zeta$ is fixed.
All statements concern the complete correction
$\mathfrak H_{\eta,u,Z}$ of \eqref{eq:modified-H}, without division by
its local factors, which may vanish.
\end{proof}

For the rational parameters the saving in the small-norm range is at least
$78699/1000000$, and the saving for the lowest zero class is at least $1093/192000$.
For the intermediate norm range, direct substitution gives
\[
 \max_{1/50\le\delta\le5/6}E_{B_{\mathrm r}}(1/2)
 \le-\frac{49}{14400}+\frac{177}{2000000}
 =-\frac{59657}{18000000}.
\]
Together with Lemma~\ref{lem:rational-saving}, these give positive exponent savings independent of the target. Choose
\[
 0<\zeta<\min\{5\ell-h,\ (19/320000)/64\}.
\]
This also ensures $\ell/(h+\zeta)>1/5>7/37$.

\section{Proof of the rational zero-free bound by the continuation criterion}
\label{sec:extension-assembly}

In this section, we assemble the direct bound \eqref{eq:balanced-low},
the principal estimate \eqref{eq:extension-principal-error} and the
nonprincipal bounds of Lemmas~\ref{lem:rational-saving} and
\ref{lem:extension-ranges}
to prove the rational extension by Proposition~\ref{prop:criterion-summary}.
Proposition~\ref{prop:extension-assembly} states the resulting
comparison \eqref{eq:extended-comparison}, with its uniformity explicit so that
Parts III and IV can apply it with their respective parameters and counts.

\begin{proposition}[Comparison from strict exponent inequalities]
\label{prop:extension-assembly}
Assume the balanced hypotheses
\eqref{eq:balanced-hypotheses}. Suppose also that:
\begin{enumerate}
\item the direct bound \eqref{eq:balanced-low} with
$l_x/2+b/12=C_b(B)$;
\item the principal estimates of Proposition~\ref{prop:extension-principal};
\item the exponents for the small and intermediate norm ranges and the lowest zero class in
Lemma~\ref{lem:extension-ranges} are negative, and
$E_B(h)\le0$ for the count with selected prime factors.
\end{enumerate}
The chosen count has the form
$\#\mathcal B\ll U^{R+\eps}(1+T_1)^A$ with the common-character,
coefficient, coprimality, support and height hypotheses of Section~\ref{sec:moments}.
Here $\mathcal B$ is a joint zero and prime-sum size class, as in
\eqref{eq:zero-class} and Section~\ref{sec:prime-sum-classes},
and $E_B(d)$ is defined in \eqref{eq:general-extension-exponent}.
Assume its losses can be made arbitrarily small by fixed length
decrements and a sufficiently fine subdivision of the total prime length, chosen independently of $\eta$. Assume also that
$0\le\partial_d E_B\le2$ in the moderate norm range.
If $\Delta=\beta_*-B>0$, then there is a number $\sigma>0$, chosen
before the target $\eta$, such that
\begin{equation}\label{eq:extended-comparison}
 |J_\eta(Z)-f_\eta(Z)|\ll_\eta Z^{C_b(\beta_*)-\sigma}.
\end{equation}
The principal factor is holomorphic on $\Rea s>B$ and satisfies
\eqref{eq:extended-H}.
\end{proposition}
\begin{proof}
Replacing $B$ by $\beta_*$ in the central exponent subtracts the
same $\Delta$ in every class. Choose a common positive saving smaller than
$\Delta/4$, the strict savings for the small and intermediate norm ranges and the lowest zero class, and
$l_y/80,h/2400$. Fix $\zeta$ smaller still so that $2\zeta$ is
less than a quarter of that saving and all prime-length inequalities remain
strict. Choose the requested moment loss, positive gaps in the prime-length restrictions and maximum individual
prime lengths smaller than that saving. The chosen count and its underlying
moments supply these strict length inequalities and the upper bound on each prime
length.

Fix an even $K$ and the actual $\ell_i$
before the target. Now choose
$0<\kappa_P<B\min_i\ell_i$, decrease the common saving if necessary,
and finally choose $\eb$ and the remaining smooth-weight and
rounding losses so that their finite sum is smaller than a quarter
of this saving. All these real choices precede $\eta$.

The class count is finite apart from logarithmically many dyadic intervals and
admissible polynomial configurations, whose logarithmic factors are absorbed in the chosen exponent losses.
Lemmas~\ref{lem:extension-ranges} and \ref{lem:extension-euler},
the chosen count, and Propositions~\ref{prop:general-exponent} and
\ref{prop:extension-principal} therefore give a common $\mu>0$ with
\begin{equation}\label{eq:height-assembly}
 |J_\eta-f_\eta|\ll_\eta
 Z^{C_b(\beta_*)-\mu}(1+T_1)^{A_\eta}
 +Z^{B_\eta}T_1^{-N}.
\end{equation}
The internal smoothness orders, $A_\eta$ and $B_\eta$ are fixed
before $N$. They may depend on the target; $\mu$ does not.
If $\tau_{\rm det}>0$ is the permissible detector height exponent,
choose
\[
 0<\tau_\eta\le
 \min\{d_{\min}/100,\tau_{\rm det},\mu/[4(A_\eta+1)]\},
 \qquad T_1=Z^{\tau_\eta}.
\]
Then choose $N$ so large that
$B_\eta-N\tau_\eta<C_b(\beta_*)-\mu/2$.
Equation~\eqref{eq:height-assembly} proves
\eqref{eq:extended-comparison} with $\sigma=\mu/2$.

The dependence of the height cutoff on $\eta$ is harmless:
$J_\eta,f_\eta$ were fixed in \eqref{eq:general-extension-signal}.
\end{proof}

\begin{proof}[Proof of Theorem~\ref{thm:rational-extension}]
Suppose $\Delta_{\mathrm r}=\beta_*-B_{\mathrm r}>0$.
Part I already gives $\beta_*\le7/8$, so \eqref{eq:fixed-kappa} holds.
The length inequalities \eqref{eq:rational-margins}, the direct
bound \eqref{eq:balanced-low}, Lemma~\ref{lem:rational-saving} and the
remaining-range estimates of Lemma~\ref{lem:extension-ranges} verify Proposition~\ref{prop:extension-assembly}.
For any $0<\omega_0<\Delta_{\mathrm r}$, the normalized direct bound is
\[
 |J_\eta(Z)|\ll Z^{C_b(B_{\mathrm r})+\omega_0}.
\]
Apply Proposition~\ref{prop:criterion-summary} with
$\sigma_0=B_{\mathrm r}$ and $c=-11/16$.
Its principal factor is holomorphic and within $1/2$ of one by
\eqref{eq:extended-H}; its comparison hypothesis is
\eqref{eq:extended-comparison}. The criterion implies
$\beta_*\le B_{\mathrm r}$, contradicting $\Delta_{\mathrm r}>0$.
The transfer proof of Corollary~\ref{cor:dirichlet-conclusion} applies
with this new positive boundary and gives the Dirichlet assertion.
\end{proof}

\part{An algebraic extension of the bound}
\label{part:optimized}

\section{The algebraic zero-free bound and its parameters}
\label{sec:optimized-statement}

In this section, we state the stronger bound obtained by varying $b$ and $\ell$
together: \eqref{eq:optimized-bound} gives the bound and
\eqref{eq:optimized-geometry} specifies its algebraic parameters.
The following sections verify these choices using the analytic
construction of Part II and the character and amplitude notation
of Part I. Whereas Part II improved $7/8$ by $1/40000$ with
$b=b_{\mathrm r}=1/8$ fixed, the joint variation gives
$7/8-B_{\mathrm{alg}}=4.293020083\ldots\times10^{-5}$.

\begin{theorem}[Algebraic extension]\label{thm:optimized-extension}
Every finite-order Hecke $L$-function over
$F$, and every Dirichlet $L$-function, is zero-free in
\begin{equation}\label{eq:optimized-bound}
 \Rea s>B_{\mathrm{alg}}:=\frac{1507-2\sqrt{921}}{1653}
 =0.874957069799\ldots.
\end{equation}
Principal poles are allowed and the boundary is excluded.
\end{theorem}

Let $r_{921}=\sqrt{921}$,
and set
\begin{equation}\label{eq:optimized-geometry}
\begin{gathered}
 \delta_{\rm c}=\frac{49-r_{921}}{48},\qquad
 \ell_{\mathrm{alg}} =\frac{33+8r_{921}}{1653},\qquad
 b_{\mathrm{alg}}=-\frac4{29}+\frac{230r_{921}}{26709},\\
 l_{x,\mathrm{alg}}=\frac{1-b_{\mathrm{alg}}-\ell_{\mathrm{alg}}}{2},\qquad
 l_{y,\mathrm{alg}}=\frac{1+b_{\mathrm{alg}}-\ell_{\mathrm{alg}}}{2},\qquad
 h_{\mathrm{alg}} =\frac{1+b_{\mathrm{alg}}+3\ell_{\mathrm{alg}}}{2},\\
 B_{\mathrm{alg}}=\frac{11}{12}-\frac{\ell_{\mathrm{alg}}}{4},\qquad
 C_{\mathrm{alg}}(s)=s-\frac{4+b_{\mathrm{alg}}}6.
\end{gathered}
\end{equation}
The subscript $\mathrm{alg}$ denotes the algebraic extension and its
associated lengths and comparison exponent. In particular,
$b_{\mathrm{alg}}=l_{y,\mathrm{alg}}-l_{x,\mathrm{alg}}$ is the difference
of two real length exponents, whereas $b_*$ in Part I is an arithmetic element. Below we write $b=b_{\mathrm{alg}}$,
$\ell=\ell_{\mathrm{alg}}$, $l_x=l_{x,\mathrm{alg}}$,
$l_y=l_{y,\mathrm{alg}}$ and $h=h_{\mathrm{alg}}$ when these parameter values are
understood.
The approximate values, used only for orientation, are
\[
\begin{array}{c|c@{\qquad}c|c}
 \ell_{\mathrm{alg}}&0.166838387470\ldots&b_{\mathrm{alg}}&0.123405436989\ldots\\
 l_{x,\mathrm{alg}}&0.354878087771\ldots&l_{y,\mathrm{alg}}&0.478283524759\ldots\\
 h_{\mathrm{alg}}&0.811960299699\ldots&\delta_{\rm c}&0.388583712271\ldots.
\end{array}
\]
Every sign and equality below uses the exact values in
\eqref{eq:optimized-geometry}.

Theorem~\ref{thm:optimized-extension} is an analytic deduction from
the cited estimates under all their stated hypotheses. Proposition~\ref{prop:scalar-optimality}
separately proves an algebraic lower bound under three explicit
constraints. Part IV proves necessity of these constraints for the explicitly
defined uniform exponent comparison system, also allowing the moment
parameter to vary.

After the first version of this paper was posted, we became aware of
several independent and contemporaneous GitHub/Zenodo posts reporting related improvements
of the $7/8$ bound:
\href{https://github.com/CrossNox/quasi-riemann-hypothesis-extension}{Ignacio Javier Mermet's GitHub repository},
\href{https://github.com/LW56e2/quasi-riemann-improved}{Leo Wei's GitHub repository},
and 
\href{https://zenodo.org/records/23221313}{Khamzat Chakhkiev's Zenodo post}.
The two GitHub repositories report the boundary $0.8749570698$,
and the Zenodo post reports the boundary $0.87496$.
These bounds use rational choices of $b$ and $\ell$ and are
slightly weaker than the bound in
Theorem~\ref{thm:optimized-extension}.

\section{The direct estimate with free total length}
\label{sec:free-direct}

In this section, we extend the direct estimate of Proposition~\ref{prop:balanced-low}
by allowing $M=l_x+l_y$ to vary independently of $\ell$.
Proposition~\ref{prop:free-low} gives the full direct exponent
\eqref{eq:full-low-exponent} and the resulting lower constraint
\eqref{eq:full-low-boundary}. Its individual exponent constraints will be used in
Section~\ref{sec:scalar-optimality} and in the full optimality proof
of Part IV. At the balanced algebraic parameters, the full estimate
also gives the direct bound checked in
Section~\ref{sec:optimized-assembly}.

Assume
\[
 l_x>\ell,\quad l_y>\ell,\quad b=l_y-l_x>0,
 \quad h=1-l_x+\ell>0.
\]
The length exponents lie in fixed bounded ranges. The principal-term exponent is
\begin{equation}\label{eq:free-signal}
 C(s)=s-5/6+l_x/3+\ell/6.
\end{equation}
In particular the additive constant changes when $b$ changes; the
Part I expression $s-11/16$ cannot be retained for the algebraic parameter choice.

\begin{lemma}[General functional-equation estimate]\label{lem:free-reflection}
Under the common-weight and coefficient hypotheses of
\Src{Lemma}{14.3}, let $B^{\mathcal I}_{a,\sigma}$ be the theta sums defined in Section~\ref{sec:balanced-direct}, now with
outer element-range exponent $m$ and remaining prime-length exponent $l$. Thus
$X'Y'\asymp Z^m$ and $Q=q_{b_*}X'Y'$; the remaining arithmetic,
coefficient and smooth-weight data are fixed as in that definition.
Their squared norm satisfies
\begin{equation}\label{eq:free-row-energy}
\begin{gathered}
 \sum_{q_a\ll Q}|B^{\mathcal I}_{a,\sigma}|^2
 \ll Z^{m+k(m,l)+\eps},\\
 k(m,l)=\max\left\{0,\frac{1+3l-2m}{4},m+l-1,\frac{3l-1}{2}\right\}.
\end{gathered}
\end{equation}
The estimate includes both signs of $m+l-1$.
\end{lemma}
\begin{proof}
Use the local lengths defined in the proof of
Lemma~\ref{lem:balanced-reflection}, now with $H\le m-O$, $z_a\le l$.
Suppress arbitrarily small offsets arising from the support intervals and logarithmic factors. Then
\[
 O,S_0,B_0,v,l_b,e_\lambda\ge0,\quad
 2A_0\le O,\quad0\le N_0\le A_0,
\]
\[
 T_d=2H+2A_0+2z_a-1-l-N_0-3B_0,
 \qquad v+3l_b+e_\lambda\le T_d,
\]
\begin{align*}
 E_{\rm ref}={}&O/2+\max(H,v+l_b)-S_0-B_0+z_a-u_{\rm hyb}
 -l_b-2e_\lambda/3\notag\\
 &\hspace{13mm}-\pos{T_d-v-3l_b-e_\lambda}/2,
 \qquad u_{\rm hyb}=\min\{v,z_a,(v+z_a)/3\}.
\end{align*}
There are four exhaustive cases. If the maximum is $H$ and
$u_{\rm hyb}=z_a$, then $E_{\rm ref}\le O/2+H\le m$.
If the maximum is $H$ and $u_{\rm hyb}\ne z_a$, then
$u_{\rm hyb}\ge v/2$ and the dual terms save at least $T_d/4$.

Substitution gives
\begin{align*}
 E_{\rm ref}&\le O/2+H/2-A_0/2+z_a/2
 +(1+l+N_0+3B_0)/4-S_0-B_0\\
 &\le m/2+(1+3l)/4.
\end{align*}
If the maximum is $v+l_b$ and $u_{\rm hyb}=z_a$, then
\[
 E_{\rm ref}\le O/2+v\le O/2+T_d\le2m+l-1.
\]
If the maximum is $v+l_b$ and $u_{\rm hyb}\ne z_a$, then
\[
 E_{\rm ref}\le O/2+v/2+z_a
 \le O/2+T_d/2+z_a\le m+(3l-1)/2.
\]
Their maximum is $m+k(m,l)$.

Restore all signed $O(1/\log Z)$ offsets, widths of the support intervals and losses from restricting the
summation ranges. If each is bounded by $\eps_0$, a common bound for losses $20\eps_0$
bounds their effect in these four linear inequalities. The homogeneous
joint-weight separation precedes squaring; its Gaussian translates
have summable seminorms and the complete norm ranges discarded in this partition have arbitrary
power saving. The remaining dyadic intervals are logarithmically many.

Choose $\eps_0$ inside the requested $\eps$. This proves the bound
with the original paper's coefficient classes and normalization of the outer element sum.
\end{proof}

\begin{proposition}[Full direct exponent]\label{prop:free-low}
The modified average satisfies
\begin{equation}\label{eq:full-low-exponent}
 |I_{\eta,\mathrm{mod}}(Z)|\ll Z^{L_{\rm low}+\eps},\qquad
 L_{\rm low}=\frac{l_x}{2}
 +\max\{0,b/12,b-l_y/2\}+\frac12 k(M,\ell).
\end{equation}
By \eqref{eq:free-signal}, the condition $L_{\rm low}\le C(B)$
for the direct estimate is equivalent to $B\ge B_{\rm low}$, where
\begin{equation}\label{eq:full-low-boundary}
 B_{\rm low}=L_{\rm low}+5/6-l_x/3-\ell/6.
\end{equation}
Under the condition $M+\ell=1$, this simplifies to
\begin{equation}\label{eq:balanced-full-low}
 B_{\rm low}=\frac{11}{12}-\frac\ell4
 +\frac{\pos{8b+3\ell-3}}{12}+\frac{\pos{5\ell-1}}{8}.
\end{equation}
\end{proposition}
\begin{proof}
Keep all three terms in \eqref{eq:Gram}; their contribution after applying
\eqref{eq:separated-cauchy-schwarz} is
$\max\{0,b/12,b-l_y/2\}$. For a subset $\mathcal I$ with total exponent
$d_{\mathcal I}$, Lemma~\ref{lem:free-reflection} applies with
$m=M-2d_{\mathcal I}$ and $l=\ell-d_{\mathcal I}$.
After tuple counting and rescaling as in
Proposition~\ref{prop:balanced-low}, its total exponent is
\[
 F(d_{\mathcal I})=l_x/2-d_{\mathcal I}
 +\max\{0,b/12,b-l_y/2+d_{\mathcal I}/2\}
 +\tfrac12 k_{\mathcal I},
\]
\[
 k_{\mathcal I}=\max\left\{0,
 \frac{1+3\ell-2M+d_{\mathcal I}}4,
 M+\ell-1-3d_{\mathcal I},
 \frac{3\ell-1-3d_{\mathcal I}}2\right\}.
\]
For $d_2>d_1$, the additive maximum increases by at most
$(d_2-d_1)/2$, and $k_{\mathcal I}$ by at most $(d_2-d_1)/4$.
Thus $F(d_2)-F(d_1)\le-3(d_2-d_1)/8$, so the empty subset controls
the direct exponent. Subtract \eqref{eq:free-signal} to obtain
\eqref{eq:full-low-boundary}. With $M=1-\ell$, the extra maximum from the functional-equation estimate is $\pos{5\ell-1}/4$, and the extra Gram maximum is
$\pos{8b+3\ell-3}/12$, proving \eqref{eq:balanced-full-low}.
\end{proof}

\section{Comparison of the character-count estimates}
\label{sec:count-optimization}

In this section, we compare the character-count estimates available at the algebraic
parameters, allowing for the supply of selected prime factors.
The permitted savings are given in \eqref{eq:capped-savings},
and \eqref{eq:crossing-candidates} lists the candidate crossing points.
This identifies the count used in Section~\ref{sec:optimized-endpoint}
and records the alternative choices examined in the optimality proof
of Part IV.

Retain $\delta=2a-1$ from \eqref{eq:zero-class} and
$x=\bar g/\delta$ from \eqref{eq:gbar}. The detecting lengths
$r,m$ satisfy \eqref{eq:lengths}--\eqref{eq:short-plain}.
For elements with $q_u\asymp U=Z^d$, let $z_{\rm sup}=\ell/d$ be the
available total selected-prime exponent relative to $U$. Put
$c_\kappa=1/(6\kappa)$, $3/4\le\kappa\le1$.
For $1\le t\le3/2$, $t-1/2\le r\le1$ and $m=t-r$,
the inverse and unweighted moment estimates give count exponents
$1-\delta F_I(r)$ and $1-\delta F_S(r)$, respectively, where
\begin{equation}\label{eq:capped-savings}
\begin{aligned}
 F_I(r)&=r+2x\min\{(1-r)/2,z_{\rm sup}\},\\
 F_S(r)&=2m+2x\min\{c_\kappa(1-2m),z_{\rm sup}\}.
\end{aligned}
\end{equation}
The count exponent for short detecting polynomials is
$1-\delta\min_r\max\{F_I(r),F_S(r)\}$.
The two savings are increasing and decreasing respectively, and their
crossing is one of
\begin{equation}\label{eq:crossing-candidates}
\begin{gathered}
 \frac{2t}{3},\qquad
 \frac{(2-4c_\kappa x)t+2c_\kappa x-2xz_{\rm sup}}{3-4c_\kappa x},\\
 \frac{2t+2xz_{\rm sup}-x}{3-x},\qquad
 \frac{(2-4c_\kappa x)t+(2c_\kappa-1)x}{3-x-4c_\kappa x}.
\end{gathered}
\end{equation}
Each expression is used only when the corresponding terms attain the minima in \eqref{eq:capped-savings}. If a crossing lies outside the interval, take its
appropriate endpoint. The count bound for long detecting polynomials remains
$R_{\rm long}(t)$ from Proposition~\ref{prop:fixed-count}.

An additional restriction on the available prime length can only
decrease $F_I$ or $F_S$ in \eqref{eq:capped-savings}.
Increasing $\kappa$ decreases $c_\kappa$ and can only decrease $F_S$.
Consequently, among these estimates with $\kappa\in[3/4,1]$, the
smallest count exponent is obtained without an additional
prime-length restriction and with $\kappa=3/4$; optimizing $t$
then gives $R_*$. With no selected primes, balancing
the short- and long-polynomial bounds without selected prime factors instead gives
\[
 t_N=\frac\alpha{\alpha-\delta/3},\qquad
 R_N=1-\frac{2\alpha\delta}{3\alpha-\delta}\ge R_*.
\]
This includes the no-prime alternative in the comparison.
Strict length inequalities in the moment estimates, coefficient restrictions and height factors
are handled exactly as in Proposition~\ref{prop:fixed-count}.
These count exponents describe the sufficient bounds furnished by the moment
estimates; they are not lower bounds on the numbers of elements.

\section{A polynomial proof of the algebraic endpoint inequality}
\label{sec:optimized-endpoint}

In this section, we derive the algebraic parameters from the critical class of the count
\eqref{eq:fixed-count-parameters} studied in Section~\ref{sec:count-optimization}, then prove the endpoint inequality
on the full rectangle of zero-class and amplitude parameters.
The exact discriminant identity \eqref{eq:optimized-discriminant} supplies the central nonprincipal estimate
for the final comparison in Section~\ref{sec:optimized-assembly}.
This proves the required endpoint inequality; the separate lower
bound in Section~\ref{sec:scalar-optimality} addresses optimality
and is not needed to establish the zero-free conclusion.

At $x=1/2$, formula \eqref{eq:fixed-count-parameters} reduces to
\begin{equation}\label{eq:critical-count}
 R_*(\delta,1/2)=\frac{185-288\delta+96\delta^2}{185-138\delta}.
\end{equation}
The equation $R_*=2/3$ has smaller root $\delta_{\rm c}$, since
\[
 288\delta^2-588\delta+185=0.
\]
At that root the coefficient of $b$ in
\eqref{eq:balanced-extension-exponent} vanishes, and $E_B(h)=0$
forces
\[
 \ell=\frac{5-6\delta_{\rm c}}{9+18\delta_{\rm c}}=\ell_{\mathrm{alg}}.
\]
Finally, the condition
$\partial_\delta E_B(h)=0$ at $(\delta,x)=(\delta_{\rm c},1/2)$,
with $R=R_*$, $\ell=\ell_{\mathrm{alg}}$ and $b$ held fixed during
differentiation in \eqref{eq:balanced-extension-exponent}, determines
$b=b_{\mathrm{alg}}$.
These equations locate the candidate; the following identity proves
its validity over the full rectangle.

\begin{lemma}[Algebraic endpoint inequality]\label{lem:optimized-endpoint}
At \eqref{eq:optimized-geometry}, for
$(\delta,x)\in[0,5/6]\times[0,1/2]$, $\bar g=x\delta$ and $R=R_*$,
\begin{equation}\label{eq:optimized-endpoint-sign}
 E_{B_{\mathrm{alg}}}(h_{\mathrm{alg}})\le0.
\end{equation}
Equality occurs only at $(\delta,x)=(\delta_{\rm c},1/2)$.
\end{lemma}
\begin{proof}
Put $y=1/2-x\ge0$ and
\[
 p(y)=7+18y+8y^2,\quad
 j(y,\delta)=185+170y+(-138+12y+96y^2)\delta,
\]
so $P_x=p(y)/9$, $\mathcal J=j/108$. Write
$C_0=-1/4+b_{\mathrm{alg}}/6+5\ell_{\mathrm{alg}}/4$.
Substitution in \eqref{eq:balanced-extension-exponent} gives
\begin{equation}\label{eq:optimized-Q}
 \mathscr Q:=216\mathcal J(-E_{B_{\mathrm{alg}}}(h_{\mathrm{alg}}))
 =-2jC_0+j(b_{\mathrm{alg}}+2\ell_{\mathrm{alg}}y)\delta
             -12h_{\mathrm{alg}}(5/6-\delta)\delta p(y).
\end{equation}
Write $\mathscr Q=\mathscr A(y)\delta^2+\mathscr D(y)\delta+\mathscr G(y)$;
these three coefficient polynomials are defined by this identity.
For $\mathscr A(y)=\sum_{i=0}^3A_i y^i$, direct expansion yields
\begin{align*}
 A_0&=31824/551-36688r_{921}/169157,\\
 A_1&=50916/551+214096r_{921}/169157,\\
 A_2&=17352/551+347296r_{921}/169157,\\
 A_3&=2112/551+512r_{921}/551.
\end{align*}
Every $A_i$ is positive. The exact discriminant identity is
\begin{equation}\label{eq:optimized-discriminant}
 4\mathscr A(y)\mathscr G(y)-\mathscr D(y)^2
 =y(d_1+d_2y+d_3y^2+d_4y^3),
\end{equation}
where
\begin{align*}
 d_1&=(3861265032800-122715443200r_{921})/279616521,\\
 d_2&=2386327618300/279616521-19093999200r_{921}/93205507,\\
 d_3&=-3389811480400/279616521+46823488000r_{921}/93205507,\\
 d_4&=-40221718800/4905553+4408611200r_{921}/14716659.
\end{align*}
Expansion with $r_{921}^2=921$ verifies the identity coefficient by
coefficient; the accompanying exact arithmetic checker verifies
this identity and the associated parameter inequalities. A rational isolating interval,
obtained by squaring its
endpoints, is
\begin{equation}\label{eq:root-bracket}
 30.347981810\ <\ r_{921}\ <\ 30.347981811.
\end{equation}

Substitution into these linear expressions gives
\[
 490<d_1<491,\quad2317<d_2<2318,\quad
 3122<d_3<3123,\quad892<d_4<893.
\]
The same bracket proves $A_0>0$; positivity of the other $A_i$ is
immediate. All decimal endpoints here are terminating rational numbers.
Complete the square:
\[
 4\mathscr A(y)\mathscr Q
 =(2\mathscr A(y)\delta+\mathscr D(y))^2
       +y(d_1+d_2y+d_3y^2+d_4y^3)\ge0.
\]
Since $\mathcal J\ge35/54$ on the required rectangle,
\eqref{eq:optimized-Q} proves the asserted sign. For $y>0$ it is
strict. At $y=0$, the exact polynomial is
$A_0(\delta-\delta_{\rm c})^2$, proving the equality assertion.
\end{proof}

\section{An algebraic lower bound for the comparison constraints}
\label{sec:scalar-optimality}

In this section, we show that three affine comparison constraints force
$B\ge B_{\mathrm{alg}}$, even when the total averaging length is free.
The constraints are \eqref{eq:dual-low-branches} and
\eqref{eq:dual-high-branch}; Proposition~\ref{prop:scalar-optimality}
proves their lower bound under \eqref{eq:admissible-exponent-system}.
This is the algebraic part of the optimality argument; Section~\ref{sec:iv-optimality}
will justify the constraints within the full specified comparison
system and allow $\kappa$ to vary.

Put $M=l_x+l_y$. The two terms $0$ and $M+\ell-1$ in
$k(M,\ell)$ of \eqref{eq:free-row-energy}, inserted into
\eqref{eq:full-low-boundary}, give the direct comparison constraints
\begin{equation}\label{eq:dual-low-branches}
 B\ge L_0:=5/6+M/12-\ell/6,\qquad
 B\ge L_K:=1/3+7M/12+\ell/3.
\end{equation}
At $(\delta,x)=(\delta_{\rm c},1/2)$, a count exponent $R\ge2/3$
in \eqref{eq:general-extension-exponent}, with $h>0$, also requires
\begin{equation}\label{eq:dual-high-branch}
 B\ge H_{\rm c}:=1+\delta_{\rm c}-(1+\delta_{\rm c})M/2
                   +\delta_{\rm c}\ell.
\end{equation}
The following proposition is a statement about these three affine
constraints. It assumes the constraints explicitly.

\begin{proposition}[Algebraic lower bound]\label{prop:scalar-optimality}
Let $M,\ell,B\in\mathbb R$ and
$\delta_{\rm c}=(49-\sqrt{921})/48$. If
\begin{equation}\label{eq:admissible-exponent-system}
 B\ge L_0,\qquad B\ge L_K,\qquad B\ge H_{\rm c},
\end{equation}
with $L_0,L_K,H_{\rm c}$ defined above, then $B\ge B_{\mathrm{alg}}$.
Equality is attained by $M=1-\ell_{\mathrm{alg}}$ and
$\ell=\ell_{\mathrm{alg}}$.
\end{proposition}
\begin{proof}
The positive weights
\[
 w_0=\frac{4+18\delta_{\rm c}}{9+18\delta_{\rm c}},\quad
 w_K=\frac2{9+18\delta_{\rm c}},\quad w_H=\frac3{9+18\delta_{\rm c}}
\]
sum to one and satisfy
\[
 w_0/12+7w_K/12-(1+\delta_{\rm c})w_H/2=0,
 \qquad-w_0/6+w_K/3+\delta_{\rm c}w_H=0.
\]
Their weighted sum cancels $M$ and $\ell$, giving
\[
 B\ge\max\{L_0,L_K,H_{\rm c}\}
 \ge w_0L_0+w_KL_K+w_HH_{\rm c}
 =\frac{7+18\delta_{\rm c}}{9+18\delta_{\rm c}}
 =\frac{1507-2r_{921}}{1653}=B_{\mathrm{alg}}.
\]
The cancellation proves the lower bound for every real $M,\ell$;
no numerical enumeration of parameter choices is involved.
At $M=1-\ell_{\mathrm{alg}}$ and $\ell=\ell_{\mathrm{alg}}$,
direct substitution gives $L_0=L_K=H_{\rm c}=B_{\mathrm{alg}}$.
\end{proof}

\paragraph{Scope of optimality.}
Part IV proves necessity of the three constraints in the explicitly
defined uniform exponent comparison system. Restricting the
moment parameter to $\kappa\ge3/4$, monotonicity of $F(\kappa)$
in Proposition~\ref{prop:iv-dual-optimality} gives
$B\ge F(3/4)=B_{\mathrm{alg}}$. The present parameter choice attains that
bound. This proves optimality for that comparison problem in the
restricted moment range, including free total length and the available
estimates without selected prime factors; it does not exclude sharper analytic estimates.

\section{Auxiliary estimates and proof of the algebraic zero-free bound}
\label{sec:optimized-assembly}

In this section, we combine the endpoint inequality of Section~\ref{sec:optimized-endpoint}
with the direct bound of Proposition~\ref{prop:balanced-low} and the
principal-term estimate of Proposition~\ref{prop:extension-principal}.
Checking the remaining strict auxiliary inequalities permits another application of
Proposition~\ref{prop:criterion-summary} and proves
Theorem~\ref{thm:optimized-extension}.

Recall that $J_\eta$ and $f_\eta$ now use the algebraic lengths
in \eqref{eq:optimized-geometry}, as prescribed in
\eqref{eq:general-extension-signal}. Their principal exponent is
$C_{\mathrm{alg}}(s)=s-(4+b_{\mathrm{alg}})/6$.

At \eqref{eq:optimized-geometry}, exact substitution gives
\begin{equation}\label{eq:optimized-geometric-margins}
\begin{gathered}
 b_{\mathrm{alg}}>0,\quad l_{x,\mathrm{alg}}-\ell_{\mathrm{alg}}>0,\quad
 l_{y,\mathrm{alg}}-\ell_{\mathrm{alg}}>0,\quad1-3\ell_{\mathrm{alg}}>0,\quad\ell_{\mathrm{alg}}<1/5,\\
 l_{y,\mathrm{alg}}-\ell_{\mathrm{alg}}-11b_{\mathrm{alg}}/6>0,\qquad
 \ell_{\mathrm{alg}}/h_{\mathrm{alg}}>1/5>7/37,\qquad5\ell_{\mathrm{alg}}-h_{\mathrm{alg}}>0.
\end{gathered}
\end{equation}
For example these inequalities follow by replacing $r_{921}$ by the
rational endpoints in \eqref{eq:root-bracket} after reducing the
expressions to $a+c r_{921}$; the denominators are positive.
Both additional positive-part terms therefore vanish in
\eqref{eq:balanced-full-low}, and
\begin{equation}\label{eq:optimized-low}
 |J_\eta(Z)|\ll Z^{C_{\mathrm{alg}}(B_{\mathrm{alg}})+\eps},\qquad
 C_{\mathrm{alg}}(B_{\mathrm{alg}})=l_{x,\mathrm{alg}}/2+b_{\mathrm{alg}}/12.
\end{equation}

Equation~\eqref{eq:optimized-geometric-margins} verifies the general
hypotheses \eqref{eq:balanced-hypotheses}; the ranges in
Lemma~\ref{lem:extension-ranges} therefore apply at these lengths.
For the lowest zero class $\delta_0=1/50$, recall the constant
$C_0=-1/4+b_{\mathrm{alg}}/6+5\ell_{\mathrm{alg}}/4$ from
Section~\ref{sec:optimized-endpoint}. The saving is
\[
 m_F=-C_0-(1/2+3\ell_{\mathrm{alg}}/2)\delta_0>0.
\]
For the intermediate norm range use $R=76/75-2\delta/3$; the upper endpoint
exponent is bounded by
\[
 E_{B_{\mathrm{alg}}}(1/2)\le
 \frac{(225\ell_{\mathrm{alg}}+50-75b_{\mathrm{alg}})\delta
              +78\ell_{\mathrm{alg}}-49b_{\mathrm{alg}}-73}{300}<0
 \quad(1/50\le\delta\le5/6).
\]
It is affine in $\delta$, and both endpoint values are negative by
\eqref{eq:root-bracket}. For the small-norm range the saving is
\[
 m_S=l_{y,\mathrm{alg}}/2-13h_{\mathrm{alg}}/75-2/100>0.
\]
For the moderate norm range,
$E_{B_{\mathrm{alg}}}(d)\le E_{B_{\mathrm{alg}}}(h_{\mathrm{alg}})\le0$ by
Lemma~\ref{lem:optimized-endpoint}; for $h_{\mathrm{alg}}<d\le h_{\mathrm{alg}}+\zeta$ the
increase in the exponent is at most $2\zeta$. The large-norm range is handled by the same
absolute contour as in Lemma~\ref{lem:extension-ranges}.
Equality in the endpoint inequality \eqref{eq:optimized-endpoint-sign}
occurs only at $(\delta,x)=(\delta_{\rm c},1/2)$.
Under the contradiction hypothesis $\beta_*>B_{\mathrm{alg}}$,
measuring the error relative to $C_{\mathrm{alg}}(\beta_*)$ subtracts
$\beta_*-B_{\mathrm{alg}}>0$, so the comparison has a strict saving.

The principal domain also needs no new argument:
\[
 437/500<B_{\mathrm{alg}}<B_{\mathrm r}<7/8.
\]
Lemma~\ref{lem:extension-euler} and
Proposition~\ref{prop:extension-principal} therefore apply with $B=B_{\mathrm{alg}}$
and the new lengths. In particular, the principal factor has no zeros
on $\Rea s>B_{\mathrm{alg}}$ and is the same factor appearing in the new principal Mellin integral.
In the small-norm range, the first-domain condition has the positive
gap $\Rea(s+w)-1\ge B_{\mathrm{alg}}-1/2>0$.

The non-strict inequality \eqref{eq:optimized-endpoint-sign} is covered
by Proposition~\ref{prop:extension-assembly}.
Indeed, under the contradiction hypothesis $\Delta_{\mathrm{alg}}=\beta_*-B_{\mathrm{alg}}>0$,
\[
 C_{\mathrm{alg}}(B_{\mathrm{alg}})+E_{B_{\mathrm{alg}}}(h_{\mathrm{alg}})-C_{\mathrm{alg}}(\beta_*)\le-\Delta_{\mathrm{alg}}.
\]
The finitely many contour, moment, subdivision into exponent intervals and normalization losses
have fixed finite coefficients. Proposition~\ref{prop:extension-assembly}
chooses them inside this gap before fixing the target.

\begin{proof}[Proof of Theorem~\ref{thm:optimized-extension}]
If $\Delta_{\mathrm{alg}}>0$, the direct bound \eqref{eq:optimized-low}, the
Euler domain of Lemma~\ref{lem:extension-euler}, the strict auxiliary
estimates above and Lemma~\ref{lem:optimized-endpoint} verify all hypotheses of
Proposition~\ref{prop:extension-assembly}.
Thus for some $\sigma>0$ independent of the target,
\[
 |J_\eta(Z)-f_\eta(Z)|\ll_\eta Z^{C_{\mathrm{alg}}(\beta_*)-\sigma}.
\]
The exponents of the height factors are allowed to depend on $\eta$;
choose $T_1=Z^{\tau_\eta}$ and then the tail order as in
\eqref{eq:height-assembly}. The modified average and principal Mellin integral are unchanged
by this choice of height cutoff.

Apply Proposition~\ref{prop:criterion-summary} with
$\sigma_0=B_{\mathrm{alg}}$, $c=-(4+b_{\mathrm{alg}})/6$ and, for example,
$\omega_0=\Delta_{\mathrm{alg}}/2$. Its hypotheses now all hold: the direct bound is
at $C_{\mathrm{alg}}(B_{\mathrm{alg}})+\omega_0$, the comparison saves $\sigma$, and the
holomorphic principal factor is within $1/2$ of one.
The criterion implies $\beta_*\le B_{\mathrm{alg}}$, contradicting $\Delta_{\mathrm{alg}}>0$.
Finally apply the norm-pullback argument of
Corollary~\ref{cor:dirichlet-conclusion} with $B=B_{\mathrm{alg}}>0$.
This gives all Dirichlet $L$-functions, including $\zeta$, and completes
the argument of Part III.
\end{proof}
\clearpage

\part{A moment-parameter extension and the optimal comparison bound}
\label{part:moment-extension}

\section{The limiting zero-free bound and its parameters}
\label{sec:iv-statement}

In this section, we state the limiting zero-free bound obtained by
varying the moment parameter together with the averaging lengths.
Theorem~\ref{thm:iv-limit} gives $B_{\mathrm{new}}$, with the explicit root
in \eqref{eq:iv-explicit-root} and the limiting parameter values
in \eqref{eq:iv-final-parameters}.
Part III gives a bound strong enough to permit $\kappa<3/4$, once the
unweighted moment has been extended to that range.
Sections~\ref{sec:iv-moment}--\ref{sec:iv-iteration}
prove it, and Section~\ref{sec:iv-optimality} proves its optimality within
the specified comparison system.

Theorem~\ref{thm:optimized-extension} gives
\[
 \beta_*\le B_{\mathrm{alg}}<7/8,
 \qquad2B_{\mathrm{alg}}-1<3/4.
\]
We extend \Src{Lemma}{18.1} to $13/18\le\kappa\le1$ and
iterate the resulting length optimization.

\begin{theorem}[Limiting zero-free bound]\label{thm:iv-limit}
Let
\begin{equation}\label{eq:iv-explicit-root}
 \xi_{\mathrm{new}}=\frac{20}{33}+\frac{\sqrt{1391}}{33}
 \cos\!\left(\frac13\arccos\frac{28470}{1391^{3/2}}
                         -\frac{2\pi}{3}\right),
\end{equation}
where $\arccos\in[0,\pi]$. Equivalently, $\xi_{\mathrm{new}}$ is the
unique root in $(0.38858,0.38859)$ of
$1188\xi^3-2160\xi^2+171\xi+190=0$. Then
\[
 \beta_*\le B_{\mathrm{new}}:=\frac{18\xi_{\mathrm{new}}+7}{18\xi_{\mathrm{new}}+9}
 =0.874957019420098946\ldots.
\]
Every finite-order Hecke $L$-function over $F$, and every Dirichlet
$L$-function, including $\zeta$, is zero-free for $\Rea s>B_{\mathrm{new}}$;
principal poles are allowed.
\end{theorem}

The separate scalar optimality statement is
Theorem~\ref{thm:iv-system-optimality}, following the definition of
the set $\mathscr A$ of admissible parameter choices in
Section~\ref{sec:iv-system-optimality}.

We retain the original modified average \eqref{eq:modified}, the two
unweighted factors of \eqref{eq:plain-moment} and the sixth-power
inverse amplification \eqref{eq:amplification}, with $\alpha=5/6$. The notation is
\[
 \delta=2a-1,\quad x=\bar g/\delta,\quad b=l_y-l_x,\quad
 h=1-l_x+\ell;
\]
$t$ denotes the detector cutoff and $\xi$ parametrizes the length choices.
The exact scalar verifier and the consolidated Lean formalization
are available in the accompanying repository \cite{LeanExtension}.

\section{Extending the unweighted fourth moment below three quarters}
\label{sec:iv-moment}
In this section, we extend the proof of \Src{Lemma}{18.1} from $[3/4,1]$
to $[13/18,1]$; its printed statement alone does not give this range.
Proposition~\ref{prop:iv-plain} states the extended estimate
\eqref{eq:iv-plain-cases}, for the parameter range
\eqref{eq:iv-kappa-range} and under its stated coefficient and zero-free
hypotheses. This is the analytic estimate needed for the variable-parameter count
in Section~\ref{sec:iv-count}, with its zero-free hypothesis supplied
at each step by the preceding bound.

\begin{proposition}[Extended unweighted product moment]
\label{prop:iv-plain}
Retain the notation, character families, fixed twists, common coprimality
conditions and bounded length ranges of \eqref{eq:plain-moment}.
In that notation, $n_1,n_2$ are the ideal-length exponents,
$z_{\rm p}=\sum_i z_i$ is the total prime-length exponent, and
$M=m_0+q_0$ is the combined range exponent of the moment estimate,
where $m_0$ bounds the element range and $q_0$ bounds the radical
of its variable residue-symbol factors. This $M$ is distinct from
the total averaging-length exponent $l_x+l_y$ used in
Section~\ref{sec:free-direct}.
The excluded group $\Theta=\langle\eta,\widehat{\mathcal T}\rangle$
is the finite group of Hecke characters defined in
Section~\ref{sec:average-separation}.
The prime polynomials have disjoint underlying prime supports before
imposing coprimality restrictions and extending the characters by zero. Then
\eqref{eq:plain-moment} holds uniformly for
\begin{equation}\label{eq:iv-kappa-range}
 \frac{13}{18}\le\kappa\le1,
\end{equation}
in the cases
\begin{equation}\label{eq:iv-plain-cases}
 \begin{array}{c|c|c}
 &\text{length condition}&\text{excluded inducing characters}\\ \hline
 z_{\rm p}=0&n_1,n_2\text{ arbitrary bounded}&\text{principal}\\
 z_{\rm p}>0&n_1+n_2+6\kappa z_{\rm p}\le M&\Theta.
 \end{array}
\end{equation}
If $z_{\rm p}>0$ and $\kappa<1$, assume
$\beta_*\le(1+\kappa)/2$. The upper bound $\ell_{\rm mesh}$ on individual prime lengths depends
only on the bounded real ranges and $\eps$, not on the number of prime-polynomial factors.
Constants and finite smooth-seminorm and polynomial-height orders
may depend on the fixed number of prime-polynomial factors and arithmetic data; they are
uniform over the character families, moduli, common coprimality
restrictions and fixed outer summation parameters in the prescribed ranges.
\end{proposition}

\begin{proof}
Put $z=z_{\rm p}$, the total prime-length exponent, and
$A=n_1+n_2+z$, the sum of the two ideal-length exponents and $z$.
Here $A$ is a real exponent, unrelated to the primary element $A$
in Section~\ref{sec:average}. We check the
$\kappa$-dependent steps of the induction on the combined range exponent in
\cite[Section~18]{QRH}. In these sums, each character remains zero on ideals divisible by
primes in its defining modulus or the fixed set $S$, including primes
absent from its primitive conductor. Cancelling residue-symbol factors
on ideals coprime to these primes does not remove their displayed prime
support from the bound $Z^{q_0}$ in \eqref{eq:plain-moment}; compare
\eqref{eq:zeros}. The induction without auxiliary prime-polynomial factors is
unchanged.

\paragraph{Prime-sum bounds and coprimality restrictions.}
For $\kappa<1$ put $s_\kappa=(1+\kappa)/2$. By
\Src{Lemma}{4.9}, whose real boundary ranges over $[1/2,1]$,
\[
 s_\kappa\ge\beta_*
 \quad\Longrightarrow\quad
 \frac{L'}{L}(s_\kappa+e+it,\psi)
 \ll_e\log\bigl[2Q_\psi(3+|t|)^2\bigr].
\]
For each character $\vartheta\in\Theta$ occurring in the prime coefficients,
\[
 \psi\notin\Theta,\quad\psi\vartheta\text{ principal}
 \quad\Longrightarrow\quad\psi=\vartheta^{-1}\in\Theta,
\]
a contradiction. Moving the Mellin contour for the smoothly weighted von Mangoldt
sum to $s_\kappa+e$, followed by division of its weight by
$\log(Py)$, gives the prime bound. Prime powers and bounded values of $P$
are admissible since
\[
 s_\kappa\ge31/36,\qquad
 P^{1/2+\eps_1}\le P^{s_\kappa+\eps_1}\quad(P\ge1).
\]
Dividing each prime sum of norm length $P_i$ by $P_i^{1/2}$, as in
the definition of $Q_k$ in \eqref{eq:plain-moment}, gives
\begin{equation}\label{eq:iv-prime-pointwise}
 |Q_{\boldsymbol\omega}|^2
 \le C Z^{\kappa z+\eps_1}p_j(\boldsymbol W)^c
                  (1+|\boldsymbol\omega|)^h.
\end{equation}

Here $p_j$ is a finite smooth seminorm, $\boldsymbol\omega$ comprises
normalized norm-twist heights, and $j,c,h$ are finite orders.
The contour loss is $2ez$ with bounded total $z$; derivatives of
the test functions or the normalized norm twists affect only these
finite seminorm and height orders.
For $\kappa=1$ use absolute prime counting.

Exact removal of the coprimality restrictions and application of the functional equation are independent of $\kappa$
\cite[Section~18]{QRH}. Removing the additional coprimality restrictions from the two unweighted
ideal sums $S_\psi(n_1)$ and $S_\psi(n_2)$ in \eqref{eq:plain-moment}
produces divisor terms with coefficients of absolute value at most
$q_d^{-1/2}$. A prime sum with an
additional coprimality restriction equals the same sum without that
restriction minus its terms at primes dividing the coprimality modulus;
the original zero extensions of the character are retained. The resulting
sums of inverse square roots of norms are $O(Z^{\eps_1})$ with arbitrary
fixed loss. The retained lengths satisfy
\[
 A'\le A,\quad z'\le z,\qquad
 A'+(6\kappa-1)z'\le A+(6\kappa-1)z\le M.
\]

If $z'=0$, use the estimate without auxiliary prime-polynomial factors or additional coprimality
restrictions. Applying the functional equation preserves
$N\mathfrak f\,N\mathfrak r\ll Z^M$, where $\mathfrak f$ is the
primitive conductor and $\mathfrak r$ is the product, each prime counted
once, of primes in the defining modulus that do not divide $\mathfrak f$.
This bound is uniform in the varying arithmetic parameters. It remains to prove the
estimate without additional coprimality restrictions.

\paragraph{Small ranges and comparison of lengths.}
Let $\rho>0$ be the lower threshold for the range exponent and $\epsilon_{\rm fr}>0$ the
allowance for enlarging the supports, independent of the zero-class $\delta$.
At $M\le\rho$, absolute counting introduces an exponent loss of at most
$\rho+\epsilon_{\rm fr}$ for the enlarged sum without auxiliary prime-polynomial factors. For $z>0$,
\begin{equation}\label{eq:iv-width-floor}
 A\ge z,\qquad6\kappa z\le M,\qquad
 z\le\frac{3M}{13},\qquad\kappa z\le\frac M6.
\end{equation}
Thus the additional exponent loss in this initial range is $\le\rho/6$; the original paper's
bound $z\le2M/9$ is not needed.

For $M>\rho$, first prove the estimate when $A\le5M/6$. If $A>5M/6$, set $L=M/4$.
Then
\[
 6\kappa-1\ge10/3,\qquad
 z<\frac{M}{6(6\kappa-1)}\le\frac M{20},\qquad
 A-z>\frac{47M}{60}>2L.
\]
If either original unweighted length exponent $n_1$ or $n_2$ is below $L$, apply the functional equation to the other factor.
Otherwise subtract the product of character sums with length
exponents $L,A-z-L$, as in \eqref{eq:rectangle}. Estimate this
comparison product separately by applying the functional equation
to its longer factor.
Applying the functional equation to the resulting difference of
the two products, including truncation
and the fixed enlargement factors for the two length ranges,
with aggregate error $\xi_{\rm ref}$ gives
\begin{equation}\label{eq:iv-reflected-comparison}
 A_{\rm comp}\le\frac32M-A+2z+\xi_{\rm ref}.
\end{equation}
Hence
\begin{equation}\label{eq:iv-comparison-margins}
 \begin{aligned}
 A_{\rm comp}&\le\frac{23}{30}M+\xi_{\rm ref},\\
 A_{\rm comp}+(6\kappa-1)z
 &\le\frac52M-2A+2z+\xi_{\rm ref}
 \le\frac{14}{15}M+\xi_{\rm ref}.
 \end{aligned}
\end{equation}

Both upper bounds lie below $M$ by at least
$M/15-\xi_{\rm ref}>0$ if
$\xi_{\rm ref}\le\rho/30$. For $z=0$ the gap below $M$ is
$M/6-\xi_{\rm ref}$. The comparison product is therefore covered by the estimate
already proved for smaller total ideal length at the same range
exponent $M$, before taking differences of products.

\paragraph{Poisson summation and reduction of the range.}
The first Poisson transform, complete support extraction,
sixth-power amplification of the Gauss sums, second complete correlation
and the full M\"obius expansion, followed by inclusion--exclusion for
divisibility among the factors, in \cite[Section~18]{QRH} are independent of $\kappa$. Retain the difference of the two products of test functions in
\eqref{eq:rectangle}, including its minus sign, until the squared
absolute value is formed. The common character, coprimality condition and norm power on each separated
side are preserved by \Src{Lemma}{18.2}; all displayed moving
primes remain counted, and the fixed coprimality restrictions remain common to both unweighted ideal
variables and both rectangles. The transformed range exponents satisfy
\[
 M'=M+J-g-g_2+t_2\le M-\sigma,
 \qquad0\le M'_{\rm act}-M'\le\delta_{{\rm fr},2}.
\]

Here $J$ is the first-transform length defect, $g,g_2$ the two
complete-support savings, $t_2$ the second common-support length,
and $\sigma>0$ the decrease in the range exponent, as in \cite[Section~18]{QRH}.

After bounding the norms of the retained Fourier frequencies, expanding the
divisibility conditions among the factors, truncating and removing the coprimality restrictions,
\[
 \begin{aligned}
 F_{\rm act}
 &=\pos{A_{\rm act}-M'_{\rm act}+(6\kappa-1)z_{\rm act}}\\
 &\le6\Delta_{\rm child}+\delta_{{\rm fr},1}+2\theta_N,
 \qquad z_{\rm act}\le z.
 \end{aligned}
\]
Here $\Delta_{\rm child}\ge0$ is the support loss in the transformed
sum and $\theta_N$ is the total error in its length exponents,
as defined in the transformed-length estimates of
\cite[Section~18]{QRH}; the inequality uses
only the original length condition and $6\kappa-1\ge0$.
Apply the pointwise bound \eqref{eq:iv-prime-pointwise} to whole prime-polynomial factors until the
remaining product satisfies the length condition or has no such factors. Removing length $d$
lowers the affine expression by
$6\kappa d$ and introduces an exponent loss $\kappa d$. With $\eta_{\rm mesh}=\ell_{\rm mesh}$,
if the removed lengths reach $F_{\rm act}/(6\kappa)$, the first partial sum
doing so exceeds that threshold by at most $\eta_{\rm mesh}$. Otherwise all factors
are removed and $d_z=z_{\rm act}<F_{\rm act}/(6\kappa)$. Thus
\begin{equation}\label{eq:iv-greedy-child}
 d_z\le\min\left\{z_{\rm act},
                      \frac{F_{\rm act}}{6\kappa}+\eta_{\rm mesh}\right\},
 \qquad\kappa d_z\le\frac{F_{\rm act}}6+\eta_{\rm mesh}.
\end{equation}

If all prime-polynomial factors have been estimated pointwise, apply the previously proved
estimate without prime-polynomial factors at the smaller range exponent; otherwise
apply the estimate with the remaining prime-polynomial factors and the restored
length condition. The remaining factors retain their coefficients and
disjoint supports. Both frequency errors are included before removing
prime-polynomial factors, and Cauchy--Schwarz averages the errors in the two
transformed sums.

For terms whose primitive inducing character belongs to the finite
group $\Theta$, use the same counting estimate as in
\cite[Section~18]{QRH} and \Src{Lemma}{18.3}: the original uncentered sum satisfies $A\le M$,
and the centered main terms cancel by the identity $X_1X_2=Y_1Y_2$ in \eqref{eq:rectangle}.
The threshold $5M/6$ is independent of $\kappa$.

\paragraph{Uniform parameter choices and finite induction.}
From the bounded ranges and $\eps$, choose
$\rho,\sigma,\epsilon_{\rm fr}>0$ with
$\epsilon_{\rm fr}\ll\min(\rho,\sigma)$ and
\[
 T_{\rm term}:=\rho+\epsilon_{\rm fr}+\rho/6+3\sigma+5\sigma/3
 <\eps/4,\qquad
 D=2+\lceil2M_{\max}/\sigma\rceil.
\]
Since $0<6\kappa-1\le5$,
\[
 |\Delta(A-M+(6\kappa-1)z)|
 \le|\Delta A|+|\Delta M|+5|\Delta z|.
\]
A numerical stage constant $C_*$ covers these fixed aggregate
operations. Before specifying the number of prime-polynomial factors, choose
\[
 \begin{aligned}
 \xi_{\rm ref}&<\min\{\epsilon_{\rm fr}/2,\rho/30,
                         \sigma/(4C_*),\eps/(16C_*D)\},\\
 \eta_{\rm mesh}&<\min\{\sigma/6,\eps/(16C_*D)\},\qquad
 \eps_0<\eps/(16C_*D).
 \end{aligned}
\]
Now fix the family of prime-polynomial factors and its aggregate support bound $H_N$;
impose $\log Z\gg H_N/\xi_{\rm ref}$ and the thresholds for applying the functional equation and
localizing the sums. The losses assigned to the local divisor and prime estimates sum to
$\eps_0$ per stage. These losses and thresholds may depend on $N$, while
$\eta_{\rm mesh}$ does not. For the simultaneous induction over surviving
subsets, take
\[
 \mathcal E_d=T_{\rm term}+(d+1)C_*(\eta_{\rm mesh}+\xi_{\rm ref}+\eps_0),
 \quad0\le d\le D-2,\qquad\mathcal E_{D-2}<\eps.
\]

First establish the estimates without auxiliary prime sums for all range
intervals. For estimates involving prime sums, each strict reduction of
the range decreases its exponent by at least $\sigma/2$, while comparisons
at a fixed range use only the preceding uncentered estimate. Every
application of induction therefore uses an estimate already proved.

Choose the seminorm and height orders needed for the integral kernels,
functional-equation transforms, prime sums and initial estimates, working backwards through
the $D$ induction steps; then choose the orders for the outer Fourier tails. Full
weighted Fourier integrals retain the height factors
of \eqref{eq:iv-prime-pointwise}: larger orders raise seminorms and
constants, not powers $Z^{J\xi_{\rm ref}}$. The coprimality moduli remain of
norm bounded by a fixed power of $Z$ through these finitely many steps.

This proves the estimate without additional coprimality restrictions;
the initial removal of those restrictions gives the proposition.
\end{proof}

\paragraph{Zero-free hypothesis.}
The required bound $\beta_*\le(1+\kappa)/2$ must precede a smaller
fixed $\kappa$. It may instead hold by equality when
$\kappa=2\beta_*-1$ in a contradiction argument.
No assertion is made below $\kappa=13/18$.

\section{Counting by simultaneous large values with a variable moment parameter}
\label{sec:iv-count}

In this section, we combine the extended moment of
Proposition~\ref{prop:iv-plain} with the inverse moment
\eqref{eq:inverse-moment}, its amplified form \eqref{eq:amplification}
and the simultaneous large values \eqref{eq:large-values}.
Keeping $\kappa$ variable yields the count exponent
$R_\kappa(\delta,x)$ in \eqref{eq:iv-balanced-count}, from which
Section~\ref{sec:iv-geometry} derives the improved length parameters.
Section~\ref{sec:iv-count-obstruction} also proves that the counting
exponent at the critical class cannot be improved using the
specified moment estimates and detecting lengths.

Recall that $\delta=2a-1$ comes from \eqref{eq:zero-class},
and $\bar g$ is the weighted mean in \eqref{eq:gbar}; the
variable $d$ satisfies $q_u\asymp Z^d$.
Put $x=\bar g/\delta\in[0,1/2]$ and
\begin{equation}\label{eq:iv-count-polynomials}
 c=\frac{2}{3\kappa},\qquad
 D_{\kappa,x}=3-(1+c)x,\qquad
 P_{\kappa,x}=(2-cx)(1-x).
\end{equation}
At $\kappa=3/4$, $(D_{\kappa,x},P_{\kappa,x})=(D_x,P_x)$.
The detecting cutoff is $t$; $d$ remains the element-norm exponent.

\begin{proposition}[Variable-parameter count]\label{prop:iv-count}
Retain the detector, character-family, coefficient, smooth-weight and height
hypotheses of Proposition~\ref{prop:fixed-count}, with the simultaneous
large values \eqref{eq:large-values} for one common character. Assume
\[
 13/18\le\kappa\le1,\qquad
 \beta_*\le(1+\kappa)/2,\qquad
 \ell/d\ge1/5+\rho_{\rm sup},\quad\rho_{\rm sup}>0.
\]
For $0<\delta\le\alpha=5/6$ and $0\le x\le1/2$, put
\begin{equation}\label{eq:iv-balanced-count}
 \mathcal J_\kappa=(\alpha-\delta)D_{\kappa,x}
                         +\delta P_{\kappa,x},\quad
 t_\kappa=1+\frac{\delta P_{\kappa,x}}{2\mathcal J_\kappa},\quad
 R_\kappa=1-\delta+
 \frac{(\alpha-\delta)\delta P_{\kappa,x}}{2\mathcal J_\kappa}.
\end{equation}
For each $\eps>0$, sufficiently fine fixed prime lengths give
\[
 \#\mathcal B\ll U^{R_\kappa+\eps}(1+T_1)^A.
\]
Here $T_1$ is the zero-detection height cutoff of
Section~\ref{sec:detector}, and $A\ge0$ is a fixed finite height
exponent for the prescribed moment parameters and arithmetic data.
\end{proposition}
\begin{proof}
Apply Proposition~\ref{prop:iv-plain} and the inverse moment
\eqref{eq:inverse-moment} to
the same character and the pair of large values in \eqref{eq:large-values}.
Here $r,m$ are the norm-length exponents of the two detecting
polynomials of Section~\ref{sec:detector}, and we set $m=t-r$ in
the exponent comparison. The maximal selected prime-length
exponents permitted by the inverse and unweighted moments are
\[
 z_M=(1-r)/2,\qquad z_S=(1-2m)/(6\kappa).
\]
The savings divided by $\delta$ are
\[
 F_I(r)=x+(1-x)r,\qquad
 F_S(r)=2(t-r)+\frac c2x\bigl[1-2(t-r)\bigr].
\]
Their crossing is
\begin{equation}\label{eq:iv-crossing}
 r_\kappa(t,x)=
 \frac{(2-cx)t+(c/2-1)x}{3-(1+c)x}.
\end{equation}
For $1\le t\le3/2$ and $0\le x\le1/2$,
\[
\begin{gathered}
 \partial_t r_\kappa=\frac{2-cx}{D_{\kappa,x}}>0,\qquad
 \partial_x r_\kappa=\frac{(2-c)(2t-3)}{2D_{\kappa,x}^2}\le0,\\
 \partial_t(t-r_\kappa)=\frac{1-x}{D_{\kappa,x}}>0,\qquad
 \partial_x(t-r_\kappa)=-\partial_xr_\kappa\ge0,\\
 r_\kappa(3/2,x)=1,\qquad
 \frac{6-c}{10-2c}\le r_\kappa\le1,\qquad
 \frac13\le t-r_\kappa\le\frac12.
\end{gathered}
\]

Use the inverse moment for $r\ge r_\kappa$ and the unweighted moment
for $r\le r_\kappa$. The admissible prime lengths satisfy
\[
 z_M\le\frac{4-c}{20-4c}<\frac15,\qquad
 z_S\le\frac c{12}<\frac15.
\]
Decrease each selected length by a fixed $\nu>0$ and choose the
individual prime lengths sufficiently small that selecting whole prime
factors introduces a rounding loss smaller than $\eps$. Then
\[
\begin{gathered}
 1-r-2z\ge2\nu,\\
 3-2r-8z=4(1-r-2z)+(2r-1)
 \ge8\nu+\frac1{5-c}>0.
\end{gathered}
\]
If no positive prime length remains available, use the moments without
selected prime factors.

Substitution of \eqref{eq:iv-crossing} gives the short count
\[
 R_{\kappa,\rm short}(t)=1-\delta+
 \frac{\delta P_{\kappa,x}}{D_{\kappa,x}}(3/2-t).
\]
The long count and the balancing identities are
\[
\begin{gathered}
 R_{\rm long}(t)=1-\alpha+(\alpha-\delta)t,\\
 R_{\kappa,\rm short}(t_\kappa)=R_{\rm long}(t_\kappa)=R_\kappa,\\
 D_{\kappa,x}-P_{\kappa,x}=1+x-cx^2>0,
 \qquad1\le t_\kappa\le3/2.
\end{gathered}
\]
At $\delta=\alpha$, $t_\kappa=3/2$ and $R_\kappa=1-\alpha$;
use the amplified inverse estimate without selected prime factors, as in
Proposition~\ref{prop:fixed-count}. All coefficient classes,
presentations and smooth separations are unchanged.
\end{proof}

\section{Balanced averaging lengths and the critical counting class}
\label{sec:iv-geometry}

In this section, we choose the averaging lengths by balancing the
direct estimate of Proposition~\ref{prop:balanced-low} with the critical
class of the count in Proposition~\ref{prop:iv-count}.
The resulting one-parameter family is defined in
\eqref{eq:iv-geometry}, on the interval \eqref{eq:iv-interval},
and contains the Part III parameter
choice \eqref{eq:optimized-geometry}; Lemma~\ref{lem:iv-endpoint}
verifies its endpoint inequality for
all zero and amplitude classes.

Recall that $b=l_y-l_x$, $\ell=\sum_i\ell_i$, and
$h=1-l_x+\ell$. The balanced relation
$l_x+l_y+\ell=1$ and the direct comparison give
$B=11/12-\ell/4$, as in \eqref{eq:balanced-signal}.
For the parameter $\xi$ in
\begin{equation}\label{eq:iv-interval}
 I=[38858/100000,38859/100000]
\end{equation}
define
\begin{equation}\label{eq:iv-geometry}
\begin{gathered}
 \kappa(\xi)=\frac{18\xi^2-57\xi+20}
                       {6(36\xi^2-78\xi+25)},\qquad
 \ell(\xi)=\frac{5-6\xi}{9+18\xi},\qquad
 B(\xi)=\frac{18\xi+7}{18\xi+9},\\
 b(\xi)=\frac{8(3\xi-1)(12\xi^2-40\xi+15)}
                    {9(2\xi+1)(24\xi^2-20\xi+5)},\\
 l_x(\xi)=\frac{1-b(\xi)-\ell(\xi)}2,\qquad
 l_y(\xi)=\frac{1+b(\xi)-\ell(\xi)}2,\qquad
 h(\xi)=\frac{1+b(\xi)+3\ell(\xi)}2,\\
 C_\xi(s)=s-\frac{4+b(\xi)}6,
 \qquad B(\xi)=\frac{11}{12}-\frac{\ell(\xi)}4.
\end{gathered}
\end{equation}
At fixed $\xi$, write $b,\ell,\kappa,h$ for
$b(\xi),\ell(\xi),\kappa(\xi),h(\xi)$. For
$\xi_0=(49-\sqrt{921})/48$, \eqref{eq:iv-geometry} recovers
\eqref{eq:optimized-geometry} and $\kappa(\xi_0)=3/4$.

\begin{lemma}[Critical class and balanced parameters]\label{lem:iv-geometry}
For the length parameters \eqref{eq:iv-geometry}, at $(\delta,x)=(\xi,1/2)$,
\[
 R_\kappa=2/3,\qquad E_{B(\xi)}(h)=0,\qquad
 \partial_\delta E_{B(\xi)}(h)=0.
\]
The equation $R_\kappa(\delta,1/2)=2/3$ is equivalent to
\begin{equation}\label{eq:iv-critical-polynomial}
 (216\kappa-18)\delta^2-(468\kappa-57)\delta
                           +150\kappa-20=0.
\end{equation}
The solution $\delta=\xi\in I$ is its smaller root.
\end{lemma}
\begin{proof}
At $x=1/2$, $D=(5-c)/2$, $P=1-c/4$, and
\[
 R_\kappa-\frac23=
 \frac{(216\kappa-18)\delta^2-(468\kappa-57)\delta
                   +150\kappa-20}{216\kappa\mathcal J_\kappa}.
\]
Thus \eqref{eq:iv-critical-polynomial} gives $\kappa(\xi)$.
The sum of its roots exceeds two; since $\xi<39/100$, the other
root exceeds one. Substitution in
\eqref{eq:balanced-extension-exponent} gives
\[
\begin{gathered}
 R_\kappa(\xi,1/2)=2/3
 \quad\Longrightarrow\quad
 \partial_bE_B(h)=R_\kappa/2-1/3=0,\\
 E_B(h)=0
 \quad\Longleftrightarrow\quad
 \ell=\frac{5-6\xi}{9+18\xi},\\
 \partial_\delta E_B(h)=0
 \quad\Longleftrightarrow\quad
 b=\frac{8(3\xi-1)(12\xi^2-40\xi+15)}
          {9(2\xi+1)(24\xi^2-20\xi+5)}.
\end{gathered}
\]
The uniform sign is proved below.
\end{proof}

\section{A uniform proof of the endpoint inequality}
\label{sec:iv-certificate}

In this section, we prove the endpoint inequality uniformly along the
family \eqref{eq:iv-geometry}: the inequality to be proved is
\eqref{eq:iv-endpoint}, including the critical class of
Lemma~\ref{lem:iv-geometry}. The polynomial identity \eqref{eq:iv-Q}
and the auxiliary bounds of Lemma~\ref{lem:iv-reserves} supply the
scalar estimates needed for Proposition~\ref{prop:iv-bootstrap}.

\begin{lemma}[Variable-parameter endpoint]\label{lem:iv-endpoint}
For every $\xi\in I$, with $I$ defined in \eqref{eq:iv-interval},
use \eqref{eq:iv-geometry} and the count exponent
$R=R_{\kappa(\xi)}$ of \eqref{eq:iv-balanced-count}.
Let $E_B(d)$ be the relative exponent in
\eqref{eq:general-extension-exponent}. Then
\begin{equation}\label{eq:iv-endpoint}
 E_{B(\xi)}(h)\le0
 \qquad(0\le\delta\le5/6,\quad0\le x\le1/2).
\end{equation}
Equality holds only at $(\delta,x)=(\xi,1/2)$.
\end{lemma}
\begin{proof}
Put $y=1/2-x$, $c=2/(3\kappa(\xi))$, and
\[
 D=(5-c)/2+(1+c)y,\qquad P=1-c/4+2y+cy^2,
 \qquad \mathcal J=(\alpha-\delta)D+\delta P.
\]
Write $E=E_{B(\xi)}(h)$. By \eqref{eq:balanced-extension-exponent},
\[
 E=C_0+(1/2+\ell)\delta+\ell x\delta-h(1-R_\kappa),
 \qquad C_0=-1/4+b/6+5\ell/4.
\]
Hence
\begin{equation}\label{eq:iv-Q}
 Q=-2\mathcal J E
 =-2C_0\mathcal J+(b+2\ell y)\delta\mathcal J
                    -h(\alpha-\delta)\delta P.
\end{equation}
Write $Q=A(y)\delta^2+B_q(y)\delta+C_q(y)$ and
$\mathscr D(y)=4A(y)C_q(y)-B_q(y)^2$.
Use the following polynomial coefficients, with distinct notation
from the prime cutoff $P_0$ and height cutoff $T_1$:
\[
 \mathsf D_0=(5-c)/2,\quad \mathsf D_1=1+c,\quad \mathsf P_0=1-c/4,
 \quad \mathsf T_0=-3/2+c/4,\quad \mathsf T_1=1-c,\quad \mathsf T_2=c.
\]
The coefficients are

\begin{align*}
 A_0&=b\mathsf T_0+h\mathsf P_0,& A_1&=b\mathsf T_1+2\ell \mathsf T_0+2h,\\
 A_2&=bc+2\ell \mathsf T_1+hc,& A_3&=2\ell c,\\
 B_0&=-2C_0\mathsf T_0+\alpha \mathsf D_0b-\alpha h\mathsf P_0,\\
 B_1&=-2C_0\mathsf T_1+\alpha \mathsf D_1b+2\alpha\ell \mathsf D_0-2\alpha h,\\
 B_2&=-2C_0c+2\alpha\ell \mathsf D_1-\alpha hc,\\
 C_{q,0}&=-2C_0\alpha \mathsf D_0,& C_{q,1}&=-2C_0\alpha \mathsf D_1.
\end{align*}
The coefficients $d_j$ of $\mathscr D$ are
\begin{align*}
 d_0&=4A_0C_{q,0}-B_0^2=0,\\
 d_1&=4(A_0C_{q,1}+A_1C_{q,0})-2B_0B_1,\\
 d_2&=4(A_1C_{q,1}+A_2C_{q,0})-B_1^2-2B_0B_2,\\
 d_3&=4(A_2C_{q,1}+A_3C_{q,0})-2B_1B_2,\\
 d_4&=4A_3C_{q,1}-B_2^2.
\end{align*}
The identity $d_0=0$ and the following strict bounds hold on $I$:
\begin{center}
\begin{tabular}{c|c@{\qquad}c|c}\toprule
quantity&lower bound&quantity&lower bound\\\midrule
$A_0$&$47/100$&$d_1$&$41/1000$\\
$A_1$&$121/100$&$d_2$&$198/1000$\\
$A_2$&$86/100$&$d_3$&$267/1000$\\
$A_3$&$29/100$&$d_4$&$76/1000$\\\bottomrule
\end{tabular}
\end{center}
For $F=N/D$, cancel common factors and normalize the common sign.
For a polynomial of
degree $n$, write
\[
\begin{gathered}
 N(L+(U-L)z)=\sum_{j=0}^n a_jz^j,\qquad
 \mathfrak b_i(N)=\sum_{j=0}^i a_j
                   \frac{\binom ij}{\binom nj},\\
 L=38858/100000,\qquad U=38859/100000.
\end{gathered}
\]
The Bernstein identity \eqref{eq:bernstein-basis} gives
$N(L+(U-L)z)=\sum_i\mathfrak b_i(N)\mathsf B_{i,n}(z)$. For every
table entry, $\mathfrak b_i(N),\mathfrak b_i(D)>0$, so
\[
 F(\xi)\ge
 \frac{\min_i\mathfrak b_i(N)}{\max_i\mathfrak b_i(D)}.
\]
The exact quotients exceed the table bounds; their coefficients and
\eqref{eq:iv-Q} are verified by
\path{combine_v2-2_supplement/check_kappa_endpoint.py}. Therefore
\[
\begin{gathered}
 A(y)>0\quad(y\ge0),\qquad
 \mathscr D(y)>0\quad(y>0),\\
 4A(y)Q=(2A(y)\delta+B_q(y))^2+\mathscr D(y)\ge0,\qquad
 \mathcal J>0\quad\Longrightarrow\quad E\le0,\\
 Q(\delta,0)=A_0(\delta-\xi)^2.
\end{gathered}
\]

Thus equality holds precisely at $(\delta,x)=(\xi,1/2)$.
\end{proof}

\begin{lemma}[Uniform positive bounds in the auxiliary inequalities]\label{lem:iv-reserves}
All parameter choices in \eqref{eq:iv-geometry} satisfy
\[
\begin{gathered}
 b>0,\quad0<\ell<1/5,\quad l_x-b-\ell>0,
 \quad l_y-\ell>11b/6,\\
 h<1,\qquad5\ell>h,\qquad437/500<B<7/8.
\end{gathered}
\]
The inequality $l_x-b-\ell>0$ is stronger than
$l_x-\ell>0$, since $b>0$.
The estimates for the lowest zero class $a=a_0$ of
\eqref{eq:zero-class} and the intermediate and small norm ranges
of Lemma~\ref{lem:extension-ranges} have uniform positive
exponent savings; the detecting and rounded count formulas have uniformly
positive denominators.
\end{lemma}
\begin{proof}
The exact bounds for the Bernstein coefficients give the lower bound $1/25000$
on $I$ for each expression in
\begin{equation}\label{eq:iv-reserve-list}
\begin{gathered}
 36\xi^2-78\xi+25,\quad \kappa',\quad B',\\
 \kappa'-1000B',\quad \kappa-13/18,\quad4/5-\kappa,\quad-p',\\
 b,\quad\ell,\quad1/5-\ell,\quad l_x-b-\ell,\quad l_y-\ell,\\
 l_y-\ell-11b/6,\quad1-h,\quad5\ell-h,\\
 B-43/50,\quad B-437/500,\quad7/8-B,\\
 -C_0-(1/2+3\ell/2)/50,\\
 l_y/2-13h/75-1/50-(7/8-B),\\
 -\bigl[(225\ell+50-75b)\delta+78\ell-49b-73\bigr]/300
 \quad(\delta=1/50,19/25),\\
 c,\quad1-c,\quad1/5-(4-c)/(20-4c),\quad1/5-c/12,\\
 (\alpha-19/25)(5-c)/2+(19/25)(1-c/4)-3/5,\\
 3-(1+c)51/100-2,\qquad
 (2-c51/100)(49/100)-3/4.
\end{gathered}
\end{equation}
Here $p(\xi)=1188\xi^3-2160\xi^2+171\xi+190$.
The 29 side bounds and eight endpoint bounds are recorded in the
exact checker. Put
\[
\begin{gathered}
 m_F=-C_0-(1/2+3\ell/2)/50,\\
 m_S=l_y/2-13h/75-1/50,\\
 m_I(\delta)=-\bigl[(225\ell+50-75b)\delta+78\ell-49b-73\bigr]/300.
\end{gathered}
\]
Then
\[
\begin{gathered}
 m_F>1/25000,\qquad m_S-(7/8-B)>1/25000,\\
 m_I(1/50),m_I(19/25)>1/25000
 \quad\Longrightarrow\quad
 m_I(\delta)>1/25000\quad(1/50\le\delta\le19/25),
\end{gathered}
\]
since $m_I$ is affine. These are the exponent savings for the lowest zero class and the small and
intermediate norm ranges of Lemma~\ref{lem:extension-ranges}.
The initial bound and small fixed class widths give
\[
 \delta\le2\beta_*-1+O(\eb)\le3/4+O(\eb)<19/25.
\]

On the enlarged rounding box $0\le x\le51/100$,
\eqref{eq:iv-reserve-list} gives $D>2$, $P>3/4$, and
$D-P=1+x-cx^2>0$. Hence
\[
\begin{gathered}
 \mathcal J=(\alpha-\delta)D+\delta P
 \ge2(\alpha-\delta)+(3/4)\delta
 \ge43/60>3/5,\\
 3/2-t_\kappa=\frac{(\alpha-\delta)D}{2\mathcal J}
 \ge\frac{\alpha-19/25}{2\alpha}=11/250>0.
\end{gathered}
\]
The rational expressions have uniformly bounded derivatives on this
compact box, so fixed rounding and class errors are absorbed in the
prescribed losses. The detecting cutoff retains its original range.
\end{proof}

\section{The one-step zero-free bound via the continuation criterion}
\label{sec:iv-assembly}

In this section, we combine Proposition~\ref{prop:iv-count} and the
uniform endpoint inequality of Lemma~\ref{lem:iv-endpoint}
into one application of Proposition~\ref{prop:criterion-summary}.
Proposition~\ref{prop:iv-bootstrap} gives the new bound
$\beta_*\le B(\xi)$ under the preceding-bound hypothesis
\eqref{eq:iv-antecedent}; Section~\ref{sec:iv-iteration} repeatedly
applies this implication.

Here $I$ is the interval \eqref{eq:iv-interval}, and
$B(\xi),\kappa(\xi)$ are the functions defined in
\eqref{eq:iv-geometry}; $\beta_*$ retains its definition
\eqref{eq:zeros}.
Fix $\xi\in I$, $\kappa(\xi)\le3/4$, and use
\eqref{eq:iv-geometry}. The zero-free hypothesis required for the moment estimate is
\begin{equation}\label{eq:iv-antecedent}
 \beta_*\le\frac{1+\kappa(\xi)}2.
\end{equation}
Theorem~\ref{thm:hecke-conclusion} also gives $\beta_*\le7/8$.

\begin{proposition}[One step of the iterative improvement]\label{prop:iv-bootstrap}
For $\xi\in I$ and $\kappa(\xi)\le3/4$,
\[
 \beta_*\le\frac{1+\kappa(\xi)}2
 \quad\Longrightarrow\quad \beta_*\le B(\xi).
\]
\end{proposition}
\begin{proof}
Assume \eqref{eq:iv-antecedent} and $\Delta:=\beta_*-B>0$.
Retain the modified average \eqref{eq:modified}, with the
normalization \eqref{eq:general-extension-signal}, and the
coefficients, zero extensions, disjoint supports and finite
character expansions of Part I.

\paragraph{The direct estimate.}
For $d_{\mathcal I}=\sum_{i\in\mathcal I}\ell_i\le\ell$,
Lemma~\ref{lem:iv-reserves} gives
\[
\begin{gathered}
 l_x-d_{\mathcal I}>0,\qquad
 l_y-d_{\mathcal I}-11b/6>0,\\
 M'=1-\ell-2d_{\mathcal I}\ge1-3\ell>0,\qquad
 P_a=q_{b_*}^{-1}Z^b,\qquad P_a^2/Y'\ll P_a^{1/6}.
\end{gathered}
\]
Thus $Q,Y',P_a\ge1$ for large $Z$, and the functional-equation and Gram
estimates apply to every rescaled subset. The resulting change in the exponent is
\[
 -d_{\mathcal I}+\pos{d_{\mathcal I}+5\ell-1}/8
 \le-7d_{\mathcal I}/8\le0.
\]
Proposition~\ref{prop:balanced-low}, with the positive normalization factor for an even
number of prime factors, yields
\begin{equation}\label{eq:iv-low}
 |J_\eta(Z)|\ll_{\eta,\eps}Z^{C_b(B)+\eps},\qquad
 C_b(B)=l_x/2+b/12,\qquad C_b(s)=s-(4+b)/6.
\end{equation}

\paragraph{Small norms.}
Take $(s_r,w_r,z_r)=(\beta_*+\eb,1/2,17/50)$, and fix the first-domain gap $\eps_0$ in \eqref{eq:D1}
with $0<\eps_0<B-1/2$. For $p\nmid u$, the error exponents in
\eqref{eq:extended-G} are
\[
 -s_r,\quad-51/25,\quad49/25-5s_r,\quad-77/50<0.
\]
For $p\mid u$, the boundary exponents are
\[
 -1/2,\quad3/2-2s_r,\quad
 (3/2-2s_r,2-3s_r),\quad2-3s_r,\quad3-5s_r<0.
\]
Since $s_r\ge B>437/500$, the quotient-free formula gives
\[
 G_p\ll\begin{cases}1,&p\nmid u,\\q_p^{1/2},&p\mid u,\end{cases}
 \qquad\text{full tuple sum}\ll U^\eps Z^{17\ell/50}.
\]
The numerator bound, the $O(U)$ choices of $u$ and the Mellin factor contribute
$U^{3/5+\eps}U U^{-17/50}=U^{63/50+\eps}$. Hence, for
$d\le d_{\min}:=1/100$, the relative exponent is at most
\[
 \frac{13h}{75}-\frac{l_y}{2}+2d_{\min}+O(\eb+\eps)
 =-m_S+O(\eb+\eps),\qquad
 m_S:=\frac{l_y}{2}-\frac{13h}{75}-\frac1{50}>0.
\]
This includes terms with $u$ a unit and terms for which the character
in the denominator $L$-function is principal; the character in the
numerator $L$-function is nonprincipal. Lemma~\ref{lem:iv-reserves} gives
$m_S>7/8-B$.

\paragraph{The lowest zero class and intermediate norms.}
The lowest zero class has
\[
 a_0=51/100,\quad\delta_0=1/50,\quad R=1,\quad
 \bar g\le\delta_0/2,\quad
 E_B(h)\le-m_F<0,
\]
where $m_F=-C_0-(1/2+3\ell/2)/50$ and
$C_0=-1/4+b/6+5\ell/4$. Its positive $d$-slope controls $d\le h$.
For $d_{\min}\le d\le1/2$, the unchanged unselected count gives
\[
 R=76/75-2\delta/3,\qquad
 \partial_dE_B=101/150-\delta/6>0,
\]
\[
 E_B(d)\le E_B(1/2)\le
 \frac{(225\ell+50-75b)\delta+78\ell-49b-73}{300}<0.
\]
The last affine expression is negative at $\delta=1/50,19/25$ by
Lemma~\ref{lem:iv-reserves}; all actual classes satisfy
\[
 \delta\le2\beta_*-1+O(\eb)\le3/4+O(\eb)<19/25<\alpha.
\]

\paragraph{Moderate and large norms.}
For $1/2\le d\le h$, Proposition~\ref{prop:iv-count} uses the same
character and the simultaneous large values \eqref{eq:large-values}
in both moment estimates. The length selected on the
appropriate side of the crossing satisfies
\[
 z_{\rm sel}<1/5<\ell/h.
\]
A fixed length decrement, followed by a sufficiently fine partition
into prime intervals, preserves the strict length inequalities for the
inverse moment and the original coefficients and supports. Lemma~\ref{lem:iv-endpoint} gives
\[
 E_B(h)\le0,\qquad
 0<33/50-\delta/2\le\partial_dE_B
   =R+\delta/2-17/50<2.
\]
Thus $E_B(d)\le0$ for $d\le h$. Choose $\zeta>0$ with
$\ell/(h+\zeta)>1/5$; extending to $h+\zeta$ increases the exponent by at most
$2\zeta$. On $(s_r,w_r,z_r)=(2,2,z_{\mathrm{new}})$,
\[
 \sum_{U>Z^{h+\zeta}}|\mathcal R_\eta(U;Z)|
 \ll Z^{l_x/2+1+l_y+(h+\zeta)(1+\eps)
                  -\zeta z_{\mathrm{new}}+\eps}.
\]

Choose $z_{\mathrm{new}}>2$ after $\zeta$ to obtain any prescribed saving.
For the moderate norm ranges, measuring the error relative to
$C_b(\beta_*)$ rather than $C_b(B)$ subtracts $\Delta$, as in
Proposition~\ref{prop:general-exponent}. The uniform bound for the
reciprocal denominator allows the amplitude-class and rounding
errors to be chosen smaller than this positive difference.

\paragraph{The principal term.}
Since $437/500<B<7/8$, Lemma~\ref{lem:extension-euler} gives
\[
 H_\eta(s)=1+O(P_0^{-4/5}),\qquad
 |H_\eta(s)-1|\le1/2\quad(\Rea s\ge B).
\]
Choose $P_0$ independently of $\eta$. After fixing $\eta$, enlarge
$S$ to include the primes dividing its chosen defining modulus; this
preserves the bound. The only principal frequency is $u=1$, with
residues at $w=1,z=1/6$ and normalizing constant $c_S$, including the factor
$1/6$.
Proposition~\ref{prop:extension-principal} supplies the principal term
\eqref{eq:general-extension-signal} and exponent savings
\[
 l_y/20,\qquad h/600,\qquad0<\kappa_P<B\min_i\ell_i.
\]
As in Proposition~\ref{prop:extension-principal}, bound the
approximation error on $\Rea s=\beta_*+\eb$ and move only the
main Gaussian integral to $\Rea s=2$.

\paragraph{Application of Proposition~\ref{prop:criterion-summary}.}
Let $m_I>0$ be the saving for the intermediate norm range and choose
\[
 0<\mu_0<\min\{\Delta/4,m_F/4,m_S/4,m_I,l_y/80,h/2400\},
 \qquad2\zeta<\mu_0.
\]
The order of the remaining fixed choices is
\[
\begin{gathered}
 \text{moment loss}\ \longrightarrow\ \text{prime-length decrements}
 \ \longrightarrow\ \max_i\ell_i
 \ \longrightarrow\ \text{even }K\text{ and lengths }\ell_i\\
 \longrightarrow\ \kappa_P
 \ \longrightarrow\ \text{class, amplitude and weight losses}
 \ \longrightarrow\ \eta.
\end{gathered}
\]
Decrease the common saving as necessary. Logarithmically many
configurations are absorbed in these losses.
The proof of Proposition~\ref{prop:extension-assembly} uses the chosen
count only through its exponent $R$, the displayed slope bounds and
the uniform hypotheses on exponent losses and individual prime lengths, all verified above for $R_\kappa$.

Proposition~\ref{prop:extension-assembly}, with the count of
Proposition~\ref{prop:iv-count}, now gives
\[
 |J_\eta-f_\eta|
 \ll_\eta Z^{C_b(\beta_*)-\mu}(1+T_1)^{A_\eta}
            +Z^{B_\eta}T_1^{-N},\qquad\mu>0,
\]
Here $\mu$ is target-independent; $A_\eta,B_\eta$ are finite
nonnegative exponents, and $N$ is the freely chosen order for
the discarded-height estimate, as in \eqref{eq:height-assembly}.
Fix the arithmetic data and
internal Fourier orders before $A_\eta,B_\eta$, and then choose
\[
 T_1=Z^{\tau_\eta},\qquad
 0<\tau_\eta\le\min\{d_{\min}/100,\tau_{\rm det},
                         \mu/[4(A_\eta+1)]\},
\]
\[
 B_\eta-N\tau_\eta<C_b(\beta_*)-\mu/2
 \quad\Longrightarrow\quad
 |J_\eta-f_\eta|\ll_\eta Z^{C_b(\beta_*)-\mu/2}.
\]
The external order $N$ changes seminorm constants, not the fixed
internal degrees; $J_\eta,f_\eta$ are independent of this height choice.
Apply Proposition~\ref{prop:criterion-summary} with
\[
 (\sigma_0,c,\omega_0,\sigma)
   =\bigl(B,-(4+b)/6,\Delta/2,\mu/2\bigr).
\]

The direct bound \eqref{eq:iv-low} and the bound $|H_\eta-1|\le1/2$ imply
continuation across a zero sufficiently near $\beta_*$. This
contradiction proves $\beta_*\le B$.
\end{proof}

\section{Proof of the limiting zero-free bound by iteration}
\label{sec:iv-iteration}

In this section, we iterate the one-step improvement and pass to the
limit to prove Theorem~\ref{thm:iv-limit}. The limiting root satisfies
\eqref{eq:iv-cubic}, its explicit expression is proved in
Lemma~\ref{lem:iv-explicit-root}, and the limiting parameters are
given in \eqref{eq:iv-final-parameters}.
Proposition~\ref{prop:iv-bootstrap} applies the continuation criterion
of Proposition~\ref{prop:criterion-summary} at each finite step.
Section~\ref{sec:iv-optimality} proves optimality in the stated
comparison system.

Put
\begin{equation}\label{eq:iv-cubic}
 p(\xi)=1188\xi^3-2160\xi^2+171\xi+190.
\end{equation}
The endpoint signs and $p'<0$ on $I$ give a unique zero
$\xi_{\mathrm{new}}\in I$.

\begin{lemma}[Explicit root]
\label{lem:iv-explicit-root}
The root $\xi_{\mathrm{new}}$ is given by \eqref{eq:iv-explicit-root}.
\end{lemma}
\begin{proof}
With $u=66\xi-40$,
\[
 242p\!\left(\frac{u+40}{66}\right)
   =u^3-4173u-56940.
\]
Set
\[
 \theta=\frac13\arccos\frac{28470}{1391^{3/2}}
 \in(0,\pi/6).
\]
For $u=2\sqrt{1391}\cos(\theta-2\pi/3)$,
\[
 u^3-4173u
 =2\cdot1391^{3/2}\cos(3\theta)=56940.
\]
Moreover,
\[
 -\sqrt{1391}<u<0
 \quad\Longrightarrow\quad
 0<\frac{40-\sqrt{1391}}{66}<\xi<\frac{20}{33}<1.
\]
Since $p(0)>0>p(1)$ and
$p$ has a root in each of $(-\infty,0)$, $(0,1)$ and $(1,\infty)$,
this is its unique root in $(0,1)$, hence the isolated root in $I$.
\end{proof}

Define
\begin{equation}\label{eq:iv-final-parameters}
 \kappa_{\mathrm{new}}=\kappa(\xi_{\mathrm{new}}),\quad
 B_{\mathrm{new}}=B(\xi_{\mathrm{new}}),\quad b_{\mathrm{new}}=b(\xi_{\mathrm{new}}),\quad
 \ell_{\mathrm{new}}=\ell(\xi_{\mathrm{new}}).
\end{equation}
The identity
\begin{equation}\label{eq:iv-fixed-point-identity}
 \kappa(\xi)-(2B(\xi)-1)
 =-\frac{p(\xi)}{18(2\xi+1)(36\xi^2-78\xi+25)}
\end{equation}
gives $\kappa_{\mathrm{new}}=2B_{\mathrm{new}}-1$.

\begin{proof}[Proof of Theorem~\ref{thm:iv-limit}]
On $I$,
\[
 \kappa'(\xi)=\frac{9(24\xi^2-20\xi+5)}
                       {2(36\xi^2-78\xi+25)^2}>0,
 \qquad B'(\xi)>0,\qquad\kappa'-1000B'>0.
\]
Let $\xi(\kappa)$ denote the inverse of the strictly increasing
function $\xi\mapsto\kappa(\xi)$ on $I$, and set
$g(\kappa)=2B(\xi(\kappa))-1$. Then
\[
 0<g'(\kappa)=\frac{2B'(\xi(\kappa))}
                         {\kappa'(\xi(\kappa))}<1/500,
\]
\[
 \kappa_{\mathrm{new}}<\kappa\le3/4
 \quad\Longrightarrow\quad
 \kappa_{\mathrm{new}}<g(\kappa)<\kappa,
\]
by \eqref{eq:iv-fixed-point-identity} and $p'<0$.
Starting with the known $7/8$ bound, define
\[
 \kappa_0=3/4,\qquad\xi_n=\xi(\kappa_n),\qquad
 B_n=B(\xi_n),\qquad\kappa_{n+1}=g(\kappa_n)=2B_n-1.
\]
Proposition~\ref{prop:iv-bootstrap} gives
\[
 \beta_*\le(1+\kappa_n)/2
 \quad\Longrightarrow\quad
 \beta_*\le B_n=(1+\kappa_{n+1})/2.
\]
Thus $B_0=B_{\mathrm{alg}}$ and
\[
 \kappa_n\downarrow\kappa_{\mathrm{new}},\qquad
 B_n\downarrow B_{\mathrm{new}},\qquad
 \beta_*\le\inf_{n\ge0}B_n=B_{\mathrm{new}}.
\]
At each finite step the preceding bound supplies the zero-free
hypothesis required by Proposition~\ref{prop:iv-bootstrap}.
Passing the real inequalities $\beta_*\le B_n$ to the limit does
not require the analytic constants or prime intervals to be
uniform in $n$.
Corollary~\ref{cor:dirichlet-conclusion} supplies the Dirichlet
transfer.
\end{proof}

The limiting parameters are
\[
\begin{array}{c|c@{\qquad}c|c}
 B_{\mathrm{new}}&0.874957019420098946\ldots&
 \kappa_{\mathrm{new}}&0.749914038840197892\ldots\\
 b_{\mathrm{new}}&0.123405685419541002\ldots&
 \ell_{\mathrm{new}}&0.166838588986270882\ldots.
\end{array}
\]
Substituting $\xi=(9B-7)/(18(1-B))$ into \eqref{eq:iv-cubic} gives
\begin{equation}\label{eq:iv-boundary-cubic}
 7884B_{\mathrm{new}}^3-18819B_{\mathrm{new}}^2+14643B_{\mathrm{new}}-3686=0.
\end{equation}
The choice of root is specified by \eqref{eq:iv-explicit-root}.

\begin{corollary}[An explicit rational bound]\label{cor:iv-rational-checkpoint}
Every $L$-function in Theorem~\ref{thm:iv-limit} is zero-free for
\[
 \Rea s>B_{\rm cp}:=\frac{34998281}{40000000}=0.874957025.
\]
\end{corollary}
\begin{proof}
For $\xi_1=(9B_{\rm cp}-7)/(18(1-B_{\rm cp}))\in I$,
\[
 p(\xi_1)<0\quad\Longrightarrow\quad
 \xi_{\mathrm{new}}<\xi_1\quad\Longrightarrow\quad B_{\mathrm{new}}<B_{\rm cp}.
\]
\end{proof}
The two improvements are
\[
\begin{aligned}
 B_{\mathrm{alg}}-B_{\rm cp}&=4.4799168740\ldots\times10^{-8},\\
 B_{\mathrm{alg}}-B_{\mathrm{new}}&=5.03790697938897\ldots\times10^{-8}.
\end{aligned}
\]
Both require changing the lengths: fixing $\ell=\ell_{\mathrm{alg}}$
leaves $11/12-\ell/4=B_{\mathrm{alg}}$.

\section{Proof of optimality in the full comparison system}
\label{sec:iv-optimality}

In this section, we prove that the specified exponent comparison system
admits no bound smaller than $B_{\mathrm{new}}$.
Theorem~\ref{thm:iv-limit} establishes that $B_{\mathrm{new}}$ is a valid zero-free bound.
Proposition~\ref{prop:iv-dual-optimality} first proves the affine lower
bound from \eqref{eq:iv-dual-constraints} and
\eqref{eq:iv-endpoint-compatibility}. We then show that the permitted
moment, length, detector and contour choices cannot evade its
constraints, leading to Theorem~\ref{thm:iv-system-optimality}.

Recall that $M=l_x+l_y$ is the total averaging-length exponent,
$b=l_y-l_x$, and $\ell$ is the total selected prime length.
The comparison exponent $E_B(d)$ is defined in
\eqref{eq:general-extension-exponent}, with $R$ chosen from the
permitted moment estimates.
For $13/18\le\kappa\le1$, let $\xi(\kappa)$ be the smaller root of
\eqref{eq:iv-critical-polynomial}, and set
\[
 I_0=[19/50,39/100],\qquad
 F(\kappa)=\frac{18\xi(\kappa)+7}{18\xi(\kappa)+9}.
\]
Writing the critical quadratic as $q_\kappa(\xi)$, one has
\[
\begin{array}{c|cc}
 &\kappa=13/18&\kappa=1\\\hline
 \xi=19/50&2776/1875&1507/625\\
 \xi=39/100&-4003/15000&-871/5000.
\end{array}
\]
Since $q_\kappa$ is affine in $\kappa$, its leading coefficient is
positive and $q_\kappa(1)=19-102\kappa<0$,
\[
 19/50<\xi(\kappa)<39/100<1<\text{the second root}.
\]
On $I_0$, put $D_0(\xi)=36\xi^2-78\xi+25$. Then
\[
 \begin{gathered}
 D_0' <0,\qquad D_0\ge D_0(39/100)=139/2500>0,\\
 \kappa'(\xi)=\frac{9(24\xi^2-20\xi+5)}{2D_0(\xi)^2}>0,
 \qquad B'(\xi)=\frac{36}{(18\xi+9)^2}>0,\\
 p''(\xi)<0,\qquad p'(\xi)\le p'(19/50)=-597474/625<0.
 \end{gathered}
\]
Thus $\xi(\kappa)$ inverts the rational function
\eqref{eq:iv-geometry} on this larger algebraic interval.
The uniform endpoint inequality is proved on $I$.

For arbitrary $M,\ell\in\mathbb R$, define
\begin{equation}\label{eq:iv-dual-branches}
\begin{gathered}
 L_0=5/6+M/12-\ell/6,\qquad
 L_K=1/3+7M/12+\ell/3,\\
 H_\xi=1+\xi-(1+\xi)M/2+\xi\ell.
\end{gathered}
\end{equation}

\begin{proposition}[Variable-moment algebraic lower bound]
\label{prop:iv-dual-optimality}
Let $13/18\le\kappa\le1$ and $M,\ell,B\in\mathbb R$. If
\begin{equation}\label{eq:iv-dual-constraints}
 B\ge L_0,\qquad B\ge L_K,\qquad B\ge H_{\xi(\kappa)},
\end{equation}
then $B\ge F(\kappa)$, with equality at
\[
 \ell=\frac{5-6\xi(\kappa)}{9+18\xi(\kappa)},\qquad M=1-\ell.
\]
Under the additional condition on the candidate bound,
\begin{equation}\label{eq:iv-endpoint-compatibility}
 B\le(1+\kappa)/2,
\end{equation}
one has $\kappa\ge\kappa_{\mathrm{new}}$ and $B\ge B_{\mathrm{new}}$.
Equality is attained by \eqref{eq:iv-final-parameters}.
\end{proposition}
\begin{proof}
For $\xi=\xi(\kappa)$, the weights
\[
 (w_0,w_K,w_H)=\frac{(4+18\xi,2,3)}{9+18\xi}
\]
are positive and satisfy
\[
 \begin{gathered}
 w_0+w_K+w_H=1,\\
 w_0/12+7w_K/12-(1+\xi)w_H/2=0,\\
 -w_0/6+w_K/3+\xi w_H=0.
 \end{gathered}
\]
Consequently
\begin{equation}\label{eq:iv-dual-identity}
 B\ge\max\{L_0,L_K,H_\xi\}
 \ge w_0L_0+w_KL_K+w_HH_\xi
 =\frac{18\xi+7}{18\xi+9}=F(\kappa).
\end{equation}
At the displayed $M,\ell$, $L_0=L_K=H_\xi=F(\kappa)$.
Moreover \eqref{eq:iv-fixed-point-identity} gives
\[
 \begin{aligned}
 F(\kappa)\le B\le(1+\kappa)/2
 &\Longrightarrow g(\kappa)=2F(\kappa)-1\le\kappa\\
 &\Longleftrightarrow p(\xi)\le0
 \Longleftrightarrow \xi\ge\xi_{\mathrm{new}}
 \Longleftrightarrow\kappa\ge\kappa_{\mathrm{new}}.
 \end{aligned}
\]
Monotonicity yields $B\ge F(\kappa)\ge F(\kappa_{\mathrm{new}})=B_{\mathrm{new}}$.
The exact identities for this positive weighted sum are checked in the supplement.
\end{proof}

The condition $B\le(1+\kappa)/2$ alone does not justify applying the
moment estimate; this requires a previously proved bound for the family. For a known
$\beta_*\le\widehat B$, the relations are
\[
 \kappa\ge2\widehat B-1,\qquad
 B\ge F(\kappa)\ge F(2\widehat B-1).
\]
If $2\widehat B-1\in\kappa(I)$ and is at most $3/4$,
Proposition~\ref{prop:iv-bootstrap} realizes the last boundary.
In particular $F(2B_{\mathrm{new}}-1)=B_{\mathrm{new}}$.

\subsection{The optimal counting exponent in the stated moment system}
\label{sec:iv-count-obstruction}

In this subsection, we optimize the exponents from the separate moment estimates with zero losses and
no additional bound on the total selected prime length.
We show that varying the detecting lengths and truncation exponent
cannot reduce the supplied counting exponent below $2/3$ at the
critical class, as stated in Lemma~\ref{lem:iv-count-obstruction}.
Lemma~\ref{lem:iv-necessary-branches} then derives the three necessary
comparison constraints, including the nonprincipal constraint
used in Theorem~\ref{thm:iv-system-optimality}.
The length exponents $r,m$ of $M_r,S_m$ in
Section~\ref{sec:detector} and the detecting cutoff $t$ of
Proposition~\ref{prop:iv-count} satisfy
\begin{equation}\label{eq:iv-witness-domain}
 t-1/2\le r\le t,\qquad0\le m\le1/2,
 \qquad r+m\ge t,\qquad1\le t\le3/2.
\end{equation}
Here $\kappa$ is the smallest parameter allowed by the zero-free hypothesis;
$\kappa\mapsto(1-2m)/(6\kappa)$ is decreasing.
Fix
\[
 \delta=\xi(\kappa)\in I_0,\qquad x=1/2,\qquad
 c=2/(3\kappa)\in[2/3,12/13],\qquad\alpha=5/6.
\]
The optimized exponents with combined range parameter $M=1$ are
\begin{equation}\label{eq:iv-optimal-marginal-counts}
\begin{aligned}
 e_I(r)&=1-\delta(1+r)/2 &&(r\le1),\\
 e_S(m)&=1-\delta\bigl[2m+(c/4)(1-2m)\bigr]
                                      &&(0\le m\le1/2),
\end{aligned}
\end{equation}
and the amplified long exponent is
\begin{equation}\label{eq:iv-optimal-long-count}
 e_L(r)=1-\alpha+(\alpha-\delta)r\qquad(r\ge1).
\end{equation}
The unselected alternatives satisfy
\[
 1-\delta r\ge e_I(r)\quad(r\le1),\qquad
 \{z>0:r+2z\le1\}=\varnothing\quad(r>1).
\]
For the set $\mathcal W_t$ of admissible length pairs in \eqref{eq:iv-witness-domain},
put
\[
 e_{\mathrm{inv}}(r)=
 \begin{cases}e_I(r),&r\le1,\\e_L(r),&r\ge1,\end{cases}
 \qquad
 R_{\mathrm{cert}}(t)=\sup_{(r,m)\in\mathcal W_t}
       \min\{e_{\mathrm{inv}}(r),e_S(m)\}.
\]
These are supplied exponents, not lower bounds for actual counts.

\begin{lemma}[Optimal counting exponent at the critical class]\label{lem:iv-count-obstruction}
For the moment estimates and admissible length pairs above, the smallest
uniform upper-bound exponent obtainable from these estimates is
\[
 \inf_{1\le t\le3/2}R_{\mathrm{cert}}(t)
 =R_\kappa(\delta,1/2)=2/3.
\]
Upper bounds on the selected prime lengths, estimates without selected prime factors
and H\"older interpolation between these estimates cannot lower it.
\end{lemma}
\begin{proof}
On $m=t-r$, the crossing in \eqref{eq:iv-crossing} satisfies
\[
 \begin{gathered}
 t-1/2\le r_\kappa(t,1/2)\le1,\qquad
 1/3\le t-r_\kappa(t,1/2)\le1/2,\\
 e_I'=-\delta/2<0,\qquad
 \frac{d}{dr}e_S(t-r)=\delta(2-c/2)>0.
 \end{gathered}
\]
Hence this admissible split imposes
\[
 \begin{aligned}
 R_{\mathrm{cert}}(t)&\ge
 \min\{e_I(r_\kappa),e_S(t-r_\kappa)\}\\
 &=R_{\kappa,\mathrm{short}}(t)
 =1-\delta+\frac{\delta P_{\kappa,1/2}}{D_{\kappa,1/2}}(3/2-t).
 \end{aligned}
\]
At $t=3/2$, $(r,m)=(1,1/2)$; the inverse and unweighted estimates
without auxiliary prime factors both give $1-\delta$. This avoids
applying an estimate with selected prime factors when its strict length inequalities hold only at equality.

The second admissible split $(r,m)=(t,0)$ imposes
$R_{\mathrm{cert}}(t)\ge R_{\mathrm{long}}(t)=e_L(t)$, because
\[
 e_S(0)-e_L(t)
 \ge-5/12+\delta(3/2-c/4)
 \ge-5/12+(19/50)(5/4)=7/120>0.
\]
If a discrete presentation omits $m=0$, take
$(r,m)=(t-\epsilon,\epsilon)$ and pass to the scalar limit.
The short/long derivation in Proposition~\ref{prop:iv-count} gives
 the reverse uniform bound. Thus
\[
 \begin{gathered}
 R_{\mathrm{cert}}(t)=
 \max\{R_{\kappa,\mathrm{short}}(t),R_{\mathrm{long}}(t)\},\\
 R_{\kappa,\mathrm{short}}'<0,\qquad
 R_{\mathrm{long}}'=\alpha-\delta>0,\\
 R_{\kappa,\mathrm{short}}(t_\kappa)
 =R_{\mathrm{long}}(t_\kappa)=R_\kappa(\delta,1/2)=2/3.
 \end{gathered}
\]
Restoring the upper bounds on selected prime lengths weakens the savings
from the separate moment estimates. At each split,
for $0\le\theta\le1$,
\[
 \theta e_I+(1-\theta)e_S\ge\min(e_I,e_S).
\]
The same minimum includes the estimates without selected prime factors; hence neither
choice nor H\"older interpolation lowers this optimal uniform exponent.
\end{proof}

Enlarging the combined range exponent in the ordinary moment estimate also cannot help. For $M'\ge1$,
write
\[
 I'=M'-\delta r-\delta z,\qquad S'=M'-2\delta m-\delta z.
\]
The affine inequalities for their range exponents give
\[
\begin{array}{ll}
 r\le1,\ z>z_I=(1-r)/2:
 &I'-e_I(r)\ge(2-\delta)(z-z_I)>0,\\
 r>1:
 &I'\ge r(1-\delta)\ge e_L(r),\\
 z>z_S=(1-2m)/(6\kappa):
 &S'-e_S(m)\ge(6\kappa-\delta)(z-z_S)>0.
\end{array}
\]
For selected lengths below these bounds, $M'\ge1$ and smaller
amplitude savings likewise give no improvement. Other displayed
unselected exponents satisfy
\[
 \begin{gathered}
 R_{\mathrm I}(\delta)\ge1-\delta/2\ge161/200>2/3,\\
 R_N=1-\frac{2\alpha\delta}{3\alpha-\delta}
 \ge1-\frac{2\alpha(39/100)}{3\alpha-39/100}>2/3.
 \end{gathered}
\]
At $x=1/2$, each auxiliary prime sum gives a maximal saving $\delta/2$ per unit length.
Redistribution, finite unions and separate detector choices introduce
no stronger exponent: each detector permits the two adverse splits
above. A relation between the large values arising from different detector choices,
or a new joint moment estimate,
would require an additional estimate.

\begin{lemma}[Necessary constraints for a putative improvement]
\label{lem:iv-necessary-branches}
Suppose a construction in the stated comparison system has
$B<B_{\mathrm{new}}$, with $l_x,l_y>\ell>0$, $b=l_y-l_x>0$,
$M=l_x+l_y$ and $h=1-l_x+\ell>0$. Use the full direct majorant
\eqref{eq:full-low-boundary} and uniform nonprincipal comparisons
for every permitted class. If $13/18\le\kappa\le1$ and the supplied
critical count exponent is at least $2/3$, then necessarily
\[
 B\ge L_0,\qquad B\ge L_K,\qquad B\ge H_{\xi(\kappa)}.
\]
If also $B\le(1+\kappa)/2$, these inequalities contradict
Proposition~\ref{prop:iv-dual-optimality}.
\end{lemma}
\begin{proof}
For $G=\max\{0,b/12,b-l_y/2\}\ge b/12$,
\[
 \begin{aligned}
 B\ge B_{\mathrm{low}}
 &=5/6+l_x/6+G-\ell/6+k(M,\ell)/2\\
 &\ge L_0+k(M,\ell)/2\ge\max\{L_0,L_K\},
 \end{aligned}
\]
using the inequalities $k\ge0$ and $k\ge M+\ell-1$.
The inequalities for the length parameters consequently imply
\[
 \begin{gathered}
 M-2\ell>b>0,\qquad B\ge L_0>5/6,\\
 M-2\ell\le12B-10,\qquad
 h\ge6-6B+b/2>3/4,\qquad h<1,\\
 M<13/14,\qquad l_y<M,\qquad
 (1+\xi)/2<139/200<5/6<B.
 \end{gathered}
\]
Thus $d=h$ is a moderate norm and the critical zero class is below
any upper cutoff supplied by $\beta_*>B$.

At $(d,\delta,x)=(h,\xi,1/2)$, direct substitution gives
\begin{equation}\label{eq:iv-critical-physical-necessity}
 E_B(h)=H_\xi-B+h(R-2/3).
\end{equation}
All $z_0$ terms cancel, and the lower bound $H_\xi$ is independent of $b$.
Since $h>0$ and $R\ge2/3$, $E_B(h)\le0$ requires $B\ge H_\xi$.
Raising the lower zero-class threshold to $\delta_f\ge\xi$ instead gives
\[
 \begin{gathered}
 E_{B,\mathrm{floor}}(h)=H_{\delta_f}-B+h/3,\\
 H_{\delta_f}-H_\xi
 =(\delta_f-\xi)(1-M/2+\ell)\ge0,
 \qquad1-M/2+\ell>3/4.
 \end{gathered}
\]
Its trivial-count comparison therefore requires the stronger
$B\ge H_\xi+h/3$. A threshold below $\xi$ retains the critical class.

\paragraph{Global reciprocal contours.}
For these comparisons assume $B<\beta_*\le7/8$.
Put $s_r=\beta_*+\gamma$, $\gamma>0$, $w_r=\omega$ and $z_r=\zeta$.
For $\omega>0$, $\zeta>1/6$ and $s_r+\omega>1$, the local error exponents for $p\nmid u$
\[
 -s_r-\omega,\ -s_r-6\zeta,\ -\omega-6\zeta,
 \ 4-6s_r-6\zeta,\ 1-s_r-\omega-6\zeta
\]
are $<-1$ and the local error exponents for $p\mid u$ are negative. The local table and
quotient-free selected formula give
\[
 \begin{gathered}
 G_p=O(1)\ (p\nmid u),\qquad
 G_p\ll q_p^{\max(1-\omega,0)}\ (p\mid u),\\
 \text{full positive tuple sum}
 \ll U^\eps Z^{\ell\zeta},\qquad
 P_i^{\zeta-\min(\omega,1)}\le P_i^\zeta.
 \end{gathered}
\]
For $0<\omega\le11/10$, the strip bound \cite[Lemma~4.8]{QRH}
contributes $U^{3/5+\eps}$, giving the relative dyad exponent
\[
 \gamma+h(\zeta-1/6)+l_y(\omega-1)+d(8/5-\zeta)+O(\eps).
\]
At $d=h$, it is bounded below by
\[
 \begin{aligned}
 \gamma+43h/30+l_y(\omega-1)
 &\ge43h/30-\beta_*l_y+\gamma(1-l_y)\\
 &>43/40-(7/8)(13/14)=21/80>0.
 \end{aligned}
\]
For $\omega>11/10$, Euler convergence gives instead
$\gamma+5h/6+l_y(\omega-1)>0$.

For $-1/100\le\omega\le0$, admissibility requires
$s_r+\omega>1$ and $\zeta\ge17/50$.
Even when the additional nonnegative losses are omitted, the endpoint is
$>21/80$; the selected-divisor factor adds $-\omega\ell\ge0$ and
deleted numerator factors add only nonnegative conductor powers.

\paragraph{Absolute and principal contours.}
For any eligible fixed $s_r,w_r>1$ and sufficiently large $z_r$,
Euler convergence and the same local table give an endpoint exponent
relative to $C(B)$ of
\[
 (s_r-B)+l_y(w_r-1)+5h/6>1-B+5h/6>3/4.
\]
It remains positive relative to $C(\beta_*)$, since
$\beta_*-B<1/24$. At $d=h$ all $z_r$ terms cancel.
Principal contour changes affect only the exponent savings for the error terms: the residues
$w=1$, $z=1/6$ and the exponent of the principal term are fixed.
Thus the displayed central, unselected, global and absolute comparisons
cannot evade $B\ge H_\xi$. Only the stated local, strip-growth and
Euler-convergence estimates are included; new estimates or residue
extractions change the system.
\end{proof}

\subsection{The optimality theorem for the specified comparison system}
\label{sec:iv-system-optimality}

In this subsection, we specify the full comparison system and combine
the preceding lower bounds to prove
Theorem~\ref{thm:iv-system-optimality}. The admissible set $\mathscr A$
is defined immediately below; the theorem proves
$\inf_{\mathfrak C\in\mathscr A}B(\mathfrak C)=B_{\mathrm{new}}$.

Let $\mathscr A$ be the set of parameter choices satisfying the following
comparison inequalities, with all exponent losses set to zero:
\[
 \begin{gathered}
 l_x,l_y>\ell>0,\quad b=l_y-l_x>0,\quad
 M=l_x+l_y,\quad h=1-l_x+\ell>0,\\
 13/18\le\kappa\le1,\qquad B\le(1+\kappa)/2,\\
 B\ge B_{\mathrm{low}},\qquad
 E_B(d)\le0\quad\text{for every permitted zero and amplitude class}.
 \end{gathered}
\]
Here $E_B(d)$ is the exponent in \eqref{eq:general-extension-exponent},
and the zero and amplitude parameters come from
\eqref{eq:zero-class} and \eqref{eq:gbar}.
The supplied count choices are those of
Section~\ref{sec:iv-count-obstruction}; the contour choices, lower zero-class thresholds and norm
ranges are those examined in the proof of Lemma~\ref{lem:iv-necessary-branches}. There is no additional bound on the total selected prime length, and
all positive losses are suppressed. Analytic applications retain the
original estimates' coefficient, coprimality, support and domain conditions.
Uniform scalar comparisons include classes whose nonemptiness is
unknown; a fixed positive gap $\beta_*-B$ does not replace the
requirement $E_B(d)\le0$.

\begin{theorem}[Optimality of the displayed exponent comparison]
\label{thm:iv-system-optimality}
For the admissible parameter choices $\mathfrak C\in\mathscr A$,
\[
 \inf_{\mathfrak C\in\mathscr A}B(\mathfrak C)=B_{\mathrm{new}}.
\]
Every such choice satisfies $B(\mathfrak C)\ge B_{\mathrm{new}}$; equality is
attained when all exponent losses are set to zero, using
\eqref{eq:iv-final-parameters}.
This includes free $M$, all detecting cutoffs in $[1,3/2]$, the
specified estimates with and without selected prime factors, and the auxiliary choices above.
\end{theorem}
\begin{proof}
If $B<B_{\mathrm{new}}$, Lemmas~\ref{lem:iv-count-obstruction} and
\ref{lem:iv-necessary-branches} give
$B\ge\max(L_0,L_K,H_\xi)$. Proposition~\ref{prop:iv-dual-optimality}
and $B\le(1+\kappa)/2$ then give $B\ge B_{\mathrm{new}}$, a contradiction.
Conversely, at \eqref{eq:iv-final-parameters},
\[
 B_{\mathrm{low}}=B_{\mathrm{new}}=L_0=L_K=H_{\xi_{\mathrm{new}}},
 \qquad B_{\mathrm{new}}=(1+\kappa_{\mathrm{new}})/2.
\]
Lemma~\ref{lem:iv-endpoint} gives every central comparison;
Lemma~\ref{lem:iv-reserves} supplies the remaining strict auxiliary inequalities.
Hence these parameters satisfy every defining inequality of
$\mathscr A$ when the exponent losses are set to zero. Their analytic
conclusion is the limit of the finite steps justified by the preceding zero-free bounds in
Theorem~\ref{thm:iv-limit}, not a moment application with an
unproved endpoint hypothesis.
\end{proof}

\begin{corollary}[Arbitrary successive admissible constructions]
\label{cor:iv-all-constructions}
At each improving step of this system, allow arbitrary admissible choices of
the length parameters, the zero-detector truncation exponent $t$, the moment
parameter $\kappa$ and the other auxiliary parameters.
If each moment hypothesis comes from the preceding proved boundary,
every resulting bound and every real limit is at least $B_{\mathrm{new}}$.
The sequence in Section~\ref{sec:iv-iteration} attains this infimum.
\end{corollary}
\begin{proof}
For an improving step justified by the preceding zero-free bound,
\[
 \beta_*\le\widehat B,\qquad
 \widehat B\le(1+\kappa)/2,\qquad
 B\le\widehat B
 \Longrightarrow B\le(1+\kappa)/2.
\]
The theorem applies independently to each step; its closed lower
bound persists under real limits.
\end{proof}

The optimality statement concerns the uniform counting and
comparison exponents obtained from the estimates specified above.
An estimate excluding a critical class, a relation between the large values arising from different detector
choices, or a stronger analytic estimate changes the system.

\subsection{Rational approximation and the scope for further improvements}
\label{sec:iv-further-parameters}

In this subsection, we give a one-step improvement with rational
parameters in Proposition~\ref{prop:iv-rational-step}, with the exact
inequalities in \eqref{eq:iv-rational-certificate}, and determine the
infimum when the parameters are required to be rational.
The rational construction gives an explicit approximation to the
algebraic limiting bound without changing its infimum.
We then distinguish variation of parameters already covered by
Theorem~\ref{thm:iv-system-optimality} from stronger analytic estimates
that could improve the bound.

\begin{proposition}[A one-step bound with rational parameters]
\label{prop:iv-rational-step}
Recall the rational boundary $B_{\mathrm{cp}}$ of
Corollary~\ref{cor:iv-rational-checkpoint}. Put
\[
 (\kappa_{\mathrm{cp}},b_{\mathrm{cp}},\ell_{\mathrm{cp}})
 =\left(\frac{74991414}{10^8},\frac{617}{5000},
                      \frac{5005157}{30000000}\right),\qquad
 11/12-\ell_{\mathrm{cp}}/4=B_{\mathrm{cp}}.
\]
With the coefficient, coprimality, support and smooth-weight hypotheses of
the modified average \eqref{eq:modified} and its normalization
\eqref{eq:general-extension-signal}, these parameters give
\[
 \beta_*\le B_{\mathrm{alg}}
 \quad\Longrightarrow\quad\beta_*\le B_{\mathrm{cp}}.
\]
Thus the explicit rational bound also follows in one step from Part III.
\end{proposition}
\begin{proof}
Suppress the subscript $\mathrm{cp}$ and set
$l_x=(1-b-\ell)/2$, $l_y=l_x+b$, $h=1-l_x+\ell$ and
$c=2/(3\kappa)$. The moment hypothesis follows by squaring positive quantities:
\[
\begin{gathered}
 q:=\frac{1507-1653(1+\kappa)/2}{2}
   =\frac{6069596329}{200000000}>0,\\
 921-q^2=\frac{402989723759}{40000000000000000}>0
 \quad\Longrightarrow\quad B_{\mathrm{alg}}<(1+\kappa)/2,
 \qquad13/18<\kappa<3/4.
\end{gathered}
\]
Use \eqref{eq:iv-Q} at this rational triple, without imposing the
critical-class equations, and write
\[
 Q=A(y)\delta^2+B_q(y)\delta+C_q(y),\quad
 A(y)=\sum_{j=0}^3A_jy^j,\quad
 4A(y)C_q(y)-B_q(y)^2=\sum_{j=0}^4d_jy^j.
\]
Let $\mathcal S_{\mathrm{cp}}$ consist of the 24 expressions in
\eqref{eq:iv-reserve-list} after omitting
$36\xi^2-78\xi+25$, $\kappa'$, $B'$, $\kappa'-1000B'$ and $-p'$,
and substituting the rational parameters above into the remaining
expressions. Exact substitution gives the following \emph{system of exact rational inequalities} $\mathcal C_{\mathrm{cp}}$:
\begin{equation}\label{eq:iv-rational-certificate}
 A_j>\frac14\ (0\le j\le3),\qquad
 d_j>\frac1{10^7}\ (0\le j\le4),\qquad
 s>\frac1{25000}\ (s\in\mathcal S_{\mathrm{cp}}).
\end{equation}

The accompanying exact-arithmetic certificate file \cite{LeanExtension}
records these lower bounds and all 33 exact values in its
\texttt{rational\_checkpoint} entry. Consequently, for $y=1/2-x\ge0$,
\[
 4A(y)Q=(2A(y)\delta+B_q(y))^2+\sum_{j=0}^4d_jy^j>0,
 \qquad E_B(h)=-Q/(2\mathcal J_\kappa)<0
\]
on $[0,5/6]\times[0,1/2]$.

The auxiliary inequalities give the balanced direct hypotheses,
$\ell/h>1/5$, the principal domain and all required exponent savings for the lowest zero class, the
intermediate and small norm ranges, the prime lengths and the rounding errors used in
Section~\ref{sec:iv-assembly}. Proposition~\ref{prop:iv-count} therefore
applies with the preceding moment hypothesis, and its exponent satisfies
$0<\partial_dE_B<2$ on the actual class range.
If $\Delta=\beta_*-B>0$, the same small-, intermediate- and large-norm
estimates, with $\mathcal C_{\mathrm{cp}}$ in place of the two
$\xi$-family certificates, verify Proposition~\ref{prop:extension-assembly}.
Together with the direct and principal estimates, this gives
\[
 |J_\eta(Z)|\ll_\eta Z^{C_b(B)+\Delta/2},\qquad
 |J_\eta(Z)-f_\eta(Z)|\ll_\eta Z^{C_b(\beta_*)-\sigma},\qquad
 \sigma>0,
\]
with $\sigma$ independent of $\eta$ and $|H_\eta-1|\le1/2$ on
$\Rea s>B$. Proposition~\ref{prop:criterion-summary} is the desired
contradiction; Corollary~\ref{cor:dirichlet-conclusion} transfers the
bound to Dirichlet $L$-functions.
\end{proof}

At the limit,
\[
 \ell_{\mathrm{new}}=11/3-4B_{\mathrm{new}},\qquad
 53835237b_{\mathrm{new}}^3+381756141b_{\mathrm{new}}^2
                 +94706196b_{\mathrm{new}}-17602208=0.
\]
The $B$-cubic \eqref{eq:iv-boundary-cubic} and this $b$-cubic reduce
modulo $7$ to
\[
 2X^3-3X^2-X+3,\qquad X^3-3X^2-3X-1,
\]
with value lists
\[
 (3,1,5,6,2,5,6),\qquad(6,1,3,4,3,6,5).
\]
Their leading coefficients are nonzero and neither has a root;
Gauss's lemma proves irreducibility. Thus
\[
 [\mathbb Q(B_{\mathrm{new}}):\mathbb Q]
 =[\mathbb Q(\ell_{\mathrm{new}}):\mathbb Q]
 =[\mathbb Q(b_{\mathrm{new}}):\mathbb Q]=3.
\]
The supplement also verifies the eliminations.

Rational parameters have the same infimum. Given $\epsilon>0$,
choose, after finitely many iterations, a bound $B_n<B_{\mathrm{new}}+\epsilon/2$ and a
slightly larger rational bound
$B_n<\widetilde B<\min\{B_{\mathrm{new}}+\epsilon,(1+\kappa_n)/2\}$.
This makes the direct and central inequalities strict. On the compact
class rectangle their exponents are continuous; a rational
$\widetilde\kappa>\kappa_n$ sufficiently close to $\kappa_n$
preserves the preceding hypothesis, and sufficiently close rational
lengths preserve all strict exponent comparisons and auxiliary inequalities. Hence
\[
 \inf_{\mathscr A\,\mathrm{with\ rational\ parameters}} B
 =\inf_{\mathscr A}B=B_{\mathrm{new}}.
\]
The first infimum is not attained: positivity of the weights
forces $L_0=L_K=H_\xi=B_{\mathrm{new}}$ at equality, so
$\ell=\ell_{\mathrm{new}}\notin\mathbb Q$.

The total length $M$ is already free, balance occurs at equality,
and $t$ is already minimized. The variables $\delta,x,d$ are class
coordinates requiring uniform control. Upper bounds on selected prime lengths and the losses from rounding and
smoothing can only weaken the relaxed comparison; the central
and remote contour coordinates cancel at $d=h$.
A stronger analytic estimate can change the obstruction. For example,
with the present $D,P$, formally vary $\alpha$ and put
\[
 R=1-\delta+\frac{(\alpha-\delta)\delta P}{2\mathcal J},
 \qquad\mathcal J=(\alpha-\delta)D+\delta P.
\]
Then
\[
 \partial_\alpha R=\frac{\delta^2P^2}{2\mathcal J^2}>0,
 \qquad \partial_\alpha e_L(r)=r-1>0\quad(r>1).
\]
A smaller \emph{proved} amplification exponent than $5/6$ improves
these upper-bound exponents. Larger admissible prime lengths in the moment estimates, a
joint moment estimate for the detecting polynomials,
better functional-equation or Gram estimates, or restrictions on actual
critical classes could likewise alter the comparison. Such changes
require new results with the original coefficient, coprimality and height
hypotheses; they are not additional free parameters in the theorem.

\appendix
\section{Theta coefficient conventions}
\label{app:theta-conventions}

In this section, we match the theta coefficients of Section~\ref{sec:theta-definition} with
the cited cusp expansions, recording the conjugations and fixed
normalizing factors explicitly.
This identifies the same arithmetic normalization throughout
Parts I--IV.

The coefficient
formulas, support and normalization are those of
\cite[equations~(5.7)--(5.16) and Appendix~A]{DR},
with the conjugation convention in
\cite[Proposition 5.1, proof, cusp expansion and infinity coefficient]{QRH}.
Precisely, if $d_i^{\rm DR}$ denotes the coefficient in
\cite[equation~(5.12)]{DR}, then our three functions are
\[
 (d_0,d_1,d_2)(\mu)=
 \bigl(\overline{d_1^{\rm DR}(-\mu)},
       \overline{d_{19}^{\rm DR}(-\mu)},
       \overline{d_{10}^{\rm DR}(-\mu)}\bigr);
\]
the corresponding coefficient formulas are rows $1,19,10$ of
\cite[Appendix~A]{DR}, reproducing
\cite[Table~III, p.~151]{Patterson}.
The matrices $H_j$ are the representatives
$\gamma_1,\gamma_{19},\gamma_{10}$ specified in
\cite[Section~5.1, before equation~(5.9)]{DR}.
These locators refer to the version
specified in the bibliography.

For clarity, write $\tau^{\rm DR}$ for the infinity coefficient
in \cite[equation~(5.7)]{DR}, and
$g^{\rm DR}(1,c)=\sum_{x\bmod c}(x/c)_3\check e(x/c)$
for its unnormalized cubic Gauss sum. The fourth case of that
formula, with its integer exponent equal to zero, gives
\[
 \tau^{\rm DR}(\pm\lambda^{-3}cn^3)
 =3^{5/2}\frac{|n|}{|c|}\overline{g^{\rm DR}(1,c)}.
\]
\cite[equation~(5.16), p.~19]{DR} uses the conjugated coefficient at
$-\mu$, while \cite[equation~(5.7), p.~17]{DR} conjugates the
Gauss sum in the infinity coefficient;
hence, since $|c|=\sqrt{q_c}$,
\[
 d_0(\lambda^{-3}cn^3)
 =\overline{\tau^{\rm DR}(-\lambda^{-3}cn^3)}
 =3^{5/2}|n|\widetilde g_3(c).
\]
Finally, \cite[equations~(5.6), (5.15), (5.16)]{DR} gives the
constant term $3^{5/2}v^{2/3}/2$ for the representative
$\gamma_1$, and zero for $\gamma_{19},\gamma_{10}$, exactly
as in Section~\ref{sec:theta-definition}.

\section{Index of cited results}
\label{app:inputs-index}

In this section, we list the cited results used in Parts I--IV and
indicate their roles in the argument.
The first table indexes the cited results and deductions in Part I. The second
records how the same results are used at the new parameters in Parts II
and III. The third records the additional results and deductions in Part IV.
Each citation identifies the proved result used in that step.

\small
\begin{longtable}{@{}>{\raggedright\arraybackslash}p{0.28\textwidth}>{\raggedright\arraybackslash}p{0.67\textwidth}@{}}
\toprule Reference & Function in the argument\\\midrule
\endfirsthead
\toprule Reference & Function in the argument\\\midrule
\endhead
\bottomrule\endfoot
\cite[Proposition 2.1]{QRH} & Holomorphic continuation with a uniform exponent.\\
\cite[Lemmas 4.1--4.4]{QRH} & Ray characters, Gauss identities and reciprocity phases.\\
\cite[Lemmas 4.5--4.10]{QRH} & Smooth separation, transform decay and Hecke bounds.\\
\cite[Proposition 5.1; Lemmas 5.2--5.8]{QRH} & Cubic theta functional equation and unweighted mean square.\\
\cite[Lemmas 6.1--6.2; Proposition 6.3]{QRH} & Balanced additive mean square and $Z^{1/4+\eps}$ estimate.\\
\cite[Lemma 7.1]{QRH} & Exact scalar Euler quotient and holomorphic correction.\\
\cite[Lemmas 8.1--8.2; Proposition 8.3]{QRH} & Zero-free rectangles and simultaneous large character polynomials.\\
\cite[Lemma 9.1; Proposition 9.2]{QRH} & Sextic sieve and preliminary character count.\\
\cite[Lemmas 10.2--10.6]{QRH} & Contour tails, exponent accounting and principal residues.\\
\cite[Lemma 11.1; Proposition 11.2; Theorem 3.1]{QRH} & Height choice, first comparison and $11/12$ conclusion.\\
\cite[Lemmas 13.1--13.4]{QRH} & Fixed-ray prime asymptotics and complete correlations.\\
\cite[Corollary 14.1; Lemmas 14.2--14.3]{QRH} & Divisibility-restricted theta functional equation and mean square.\\
\cite[Lemma 15.1; Propositions 15.2--15.3]{QRH} & Modified theta and additive norms; $Z^{3/16+\eps}$.\\
\cite[Proposition 16.1; Lemma 16.2]{QRH} & Joint local errors and absolute tuple bounds.\\
\cite[Lemmas 17.1--17.3, 17.5--17.6; Corollary 17.4]{QRH} & M\"obius-weighted moment, finite induction and sixth-power amplification.\\
\cite[Lemmas 18.1--18.3]{QRH} & Unweighted fourth moment and cancellation under $X_1X_2=Y_1Y_2$.\\
\cite[Lemma 19.1; Proposition 19.2]{QRH} & Prime sum bounds and refined character counts.\\
\cite[Lemmas 20.1--20.2; Proposition 20.3]{QRH} & Final exponent inequality and uniform parameter choices.\\
\cite[Proposition 11.3]{QRH} & Transfer to Dirichlet functions.\\
\end{longtable}
\normalsize

\small
\begin{longtable}{@{}>{\raggedright\arraybackslash}p{0.31\textwidth}>{\raggedright\arraybackslash}p{0.64\textwidth}@{}}
\toprule Cited result & Use in Parts II and III\\\midrule
\endfirsthead
\toprule Cited result & Use in Parts II and III\\\midrule
\endhead
\bottomrule\endfoot
\Src{Lemma}{14.3} & General exponents from the functional equation:
Lemmas~\ref{lem:balanced-reflection} and \ref{lem:free-reflection}.\\
\Src{Proposition}{15.2} & Direct bounds with the full additive majorant:
Propositions~\ref{prop:balanced-low} and \ref{prop:free-low}.\\
\Src{Lemma}{7.1} & Extended principal Euler region:
Lemma~\ref{lem:extension-euler}.\\
\Src{Lemma}{13.1}; \Src{Lemmas}{10.2--10.6} & Normalization, residues after changing the averaging lengths
and tails: Proposition~\ref{prop:extension-principal} and
Lemma~\ref{lem:extension-ranges}.\\
\Src{Proposition}{8.3}; \Src{Lemmas}{17.1, 17.6, 18.1, 19.1} &
Fixed-$\kappa$ count with strict length inequalities:
Proposition~\ref{prop:fixed-count}.\\
\Src{Proposition}{16.1}; \Src{Lemma}{16.2} & Joint conductor deficit and
nonprincipal exponent for the new lengths: Proposition~\ref{prop:general-exponent}.\\
\Src{Lemmas}{4.5--4.10, 10.2--10.4, 11.1} & Smooth separation and combination of height
estimates: Proposition~\ref{prop:extension-assembly}.\\
Proposition~\ref{prop:criterion-summary} and
Corollary~\ref{cor:dirichlet-conclusion} & Continuation contradiction and
Dirichlet transfer at both new boundaries.\\
\end{longtable}
\normalsize

\small
\begin{longtable}{@{}>{\raggedright\arraybackslash}p{0.31\textwidth}>{\raggedright\arraybackslash}p{0.64\textwidth}@{}}
\toprule Cited result & Use in Part IV\\\midrule
\endfirsthead
\toprule Cited result & Use in Part IV\\\midrule
\endhead
\bottomrule\endfoot
\Src{Lemma}{4.9}; \cite[Section~18, proof of Lemma~18.1]{QRH} &
Prime displacement, induction on the combined range exponent, and the extended moment:
Proposition~\ref{prop:iv-plain}. The proof in the original paper is extended, rather
than its printed parameter range invoked unchanged.\\
\Src{Proposition}{8.3}; \Src{Lemmas}{17.1, 17.6, 19.1} &
The original common detector and inverse amplification, combined
with the extended unweighted moment in Proposition~\ref{prop:iv-count}.\\
\eqref{eq:balanced-extension-exponent} and the exact Bernstein
calculation of Section~\ref{sec:iv-certificate} &
Critical length parameters, the universal endpoint inequality, and strict
auxiliary inequalities: Lemmas~\ref{lem:iv-geometry}, \ref{lem:iv-endpoint},
and \ref{lem:iv-reserves}.\\
Propositions~\ref{prop:balanced-low}, \ref{prop:extension-principal},
\ref{prop:general-exponent}; Lemma~\ref{lem:extension-euler} &
Direct and transformed representations for the new lengths, principal residues,
and every norm range in Proposition~\ref{prop:iv-bootstrap}.\\
Proposition~\ref{prop:criterion-summary};
Corollary~\ref{cor:dirichlet-conclusion} &
One-step continuation contradiction, iteration of the zero-free bounds, and
Dirichlet transfer in Theorem~\ref{thm:iv-limit}.\\
\end{longtable}

\end{document}